\documentclass[a4paper,oneside,11pt]{article}
\usepackage{amsmath}
\usepackage{amssymb}
\usepackage{amsthm}
\usepackage{enumerate}
\usepackage{bm}
\usepackage[margin = 3cm]{geometry}
\usepackage{color}
\usepackage{url}
\usepackage{enumitem} 

\newcommand{\fhi}{\varphi}

\newcommand{\N}{\mathbb{N}}
\newcommand{\RR}{\mathbb{R}}
\newcommand{\phip}{\varphi_{p}}
\newcommand{\phid}{\varphi_{d}}
\newcommand{\phih}{\varphi_{h}}
\newcommand{\vu}{\bm{u}}
\newcommand{\de}{\partial}
\newcommand{\pdnu}{\partial_{{\norma}}}
\DeclareMathOperator{\inte}{int}

\newcommand{\norma}{\bm{n}}

\newcommand{\io}{\int_\Omega}

\newcommand{\OO}{_{\Omega}}

\newcommand{\Triangle}{\Delta}

\newcommand{\eps}{\varepsilon}

\newcommand{\abs}[1]{\left| #1 \right|}
\newcommand{\norm}[1]{\| #1 \|}

\newcommand{\inner}[2]{\langle #1 , #2 \rangle}

\newcommand{\Laplace}{\Delta}

\renewcommand{\div}{\, \mathrm{div}\,}

\newtheorem{thm}{Theorem}[section]
\newtheorem{lemma}[thm]{Lemma}

\newtheorem{remark}{Remark}[section]

\newtheorem{defn}{Definition}[section]
\newtheorem{assump}{Assumption}[section]
\numberwithin{equation}{section}

\newenvironment{giuliorem}{\color{red}}{\color{black}}
\newcommand{\III}{\begin{giuliorem}}
\newcommand{\EEE}{\end{giuliorem}}

 \newenvironment{bettirem}{\color{magenta}}{\color{black}}
\newcommand{\RRR}{\begin{bettirem}}
\newcommand{\FFF}{\end{bettirem}}

\newenvironment{Andrewrem}{\color{blue}}{\color{black}}
\newcommand{\AAA}{\begin{Andrewrem}}
\newcommand{\BBB}{\end{Andrewrem}}

\newenvironment{Sergiorem}{\color{cyan}}{\color{black}}
\newcommand{\SSS}{\begin{Sergiorem}}
\newcommand{\TTT}{\end{Sergiorem}}

\begin{document}

\title{On a multi-species Cahn--Hilliard--Darcy tumor growth model with singular potentials}
\author{Sergio Frigeri \footnotemark[1] \and Kei Fong Lam \footnotemark[2] \and Elisabetta Rocca \footnotemark[3] \and Giulio Schimperna \footnotemark[3]}

\date{\today}

\maketitle

\renewcommand{\thefootnote}{\fnsymbol{footnote}}
\footnotetext[1]{Dipartimento di Matematica e Fisica,
Universit\`a Cattolica del Sacro Cuore,
Via dei Musei 41, I-25121 Brescia, Italy ({\tt sergio.frigeri.sf@gmail.com}).}
\footnotetext[2]{Department of Mathematics,
The Chinese University of Hong Kong,
Shatin, N.T., Hong Kong
({\tt kflam@math.cuhk.edu.hk}).}
\footnotetext[3]{Di\-par\-ti\-men\-to di Ma\-te\-ma\-ti\-ca, Uni\-ver\-si\-t\`a di Pa\-via, and IMATI - C.N.R., Pavia, Via Fer\-ra\-ta 1, I-27100, Pa\-via, Italy
({\tt e\-li\-sa\-bet\-ta.rocca@unipv.it, giusch04@unipv.it}).}

\begin{abstract}
We consider a model describing the evolution of a tumor inside a host tissue in terms
of the parameters $\phip$, $\phid$ (proliferating and dead cells, respectively), $\vu$ (cell velocity) and $n$ (nutrient concentration).  The variables $\phip$, $\phid$
satisfy a Cahn--Hilliard type system with nonzero forcing term (implying that their spatial means are not conserved in time), whereas $\vu$ obeys a form of the Darcy law and $n$ satisfies a
quasi-static diffusion equation.  The main novelty of the present work stands in the fact that we are able to consider a configuration potential of singular type implying that the concentration vector $(\phip,\phid)$ is constrained to remain in the range of physically admissible values.
On the other hand, in view of the presence of nonzero forcing terms, this choice gives rise to a number of mathematical difficulties, especially related to the control of the mean values of $\phip$ and $\phid$.  For the resulting mathematical problem, by imposing suitable initial-boundary conditions, our main result concerns 
the  existence of weak solutions in a proper regularity class.
\end{abstract}




\section{Introduction}

Tumor growth remains an active area of scientific research due to the impact on the quality of life for those diagnosed with cancer.  Starting with the seminal work of
Burton \cite{Burton} and Greenspan \cite{Greenspan}, many mathematical models have been proposed to emulate the complex biological and chemical processes that occur in tumor growth
with the aim of better understanding and ultimately controlling the the behavior of cancer cells.
In recent years, there has been a growing interest in the mathematical modelling of cancer, see for example \cite{Araujo,Bellomo,Byrne,Cristini,Fasano,Friedman}.
Mathematical models for tumor growth may have different analytical features: in the present work
we are focusing on the subclass of  continuum models, namely diffuse interface models. In this framework, the tumor and surrounding host tissue occupy regions of a domain and are
subject to various balance laws mimicking the biological processes one would like to model.  While it is intuitive to represent the interfaces between the tumor and healthy tissues as
idealized surfaces of zero thickness, leading to a sharp interface description that differentiates the tumor and the surrounding host tissue cell-by-cell, these kinds of sharp
interface models are often difficult to analyze mathematically, and may break down when the interface undergoes a topological change.
In particular, sharp interface models may fail to describe the process of metastasis, which is the spreading of cancer to other parts of the body, and the primary characteristic that makes cancer so deadly.

On the other hand, diffuse interface models consider the interface between the tumor and the healthy tissues as a narrow layer in which tumor and healthy cells are mixed.
This alternative representation of the interface gives rise to model equations that are better amenable to mathematical analysis, and the mathematical description remains
valid even when the tumor undergoes topological changes.  Hence, the recent efforts in the mathematical modeling of tumor growth have been mostly focused on diffuse interface
models, see for example \cite{CLLW,Cristini,Frieboes,GLNS,GLSS,Hawkins,Oden,Wise}, and their numerical simulations demonstrating complex changes in tumor morphologies
due to mechanical stresses and interactions with chemical species such as nutrients or toxic agents.

The interaction of multiple tumor cell species can be described by using multiphase mixture models
(see, e.g., \cite{Araujo, Dai, Frieboes, GLNS, Sciume, Wise}). Indeed, using multiphase porous media mechanics,
the authors of \cite{Sciume} represented a growing tumor as a multiphase medium containing an
extracellular matrix, tumor and host cells, and interstitial liquid. Numerical simulations were also
performed that characterized the process of cancer growth in terms of the initial tumor-to-healthy cell
density ratio, nutrient concentration, mechanical strain, cell adhesion, and geometry.
The interactions of a growing tumor and a basement membrane were studied
in \cite{Bresch}.  In \cite{Chen} the authors adapted the approach from \cite{Bresch}
to the case of multiphase mixture models. Although the model studied in \cite{Chen} contains
numerous simplifications, the underlying approach has been successfully used
by other authors, also with the purpose of comparing the outcome of numerical tests with
experimental data (e.g., \cite{Frieboes}). Furthermore, the modeling strategy and the numerical
methods presented in \cite{Chen} are generalizable to more sophisticated and thermodynamically
consistent models that account for additional biophysical details (e.g., \cite{Cristini}). For these
reasons, we will use here a modeling approach closely related to that of \cite{Chen},
with few modifications that will be detailed here below.

In terms of the theoretical analysis of diffuse interface models, most of the recent literature is restricted
to the two-phase variant, i.e., models that only account for the evolution of a tumor surrounded by
healthy tissue. In this setting, there is no differentiation among the tumor cells  that exhibit heterogeneous growth behavior,
and consequently this kind of two-phase models are just capable to describe the growth of a young tumor before
the onset of quiescence and necrosis.  Analytical results related to existence of weak and strong solutions, and,
in some cases, continuous dependence on data, asymptotic limits and long-time behavior have been established in
\cite{CGH,CGRSAA,CGRSVV,FGR,FLR,GLDirichlet,GLDarcy,GLNeumann,RoccaScala}
for tumor growth models based on the coupling of Cahn--Hilliard (for the tumor density)
and reaction--diffusion (for the nutrient or other chemical factors) equations,
and in \cite{GLDarcy,GGW,JWZ,LTZ,Melchionna} for models of Cahn--Hilliard--Darcy type.
There have also been some studies involving the optimal control and sliding modes
for diffuse interface tumor growth, see, e.g., \cite{CGMR,CGRSOpCon,GLR}.

Comparatively, there have been fewer analytical results for the multi-phase variants,
which distinguish between the proliferating and necrotic tumor cells.
We just mention and briefly discuss the recent works of \cite{Dai, GLNS}. In \cite{GLNS}, a Cahn--Hilliard--Darcy model
is derived to describe multiphase tumor growth taking interactions with multiple chemical species into account as
well as the simultaneous occurrence of proliferating, quiescent and necrotic regions. A new feature of the modeling
approach of \cite{GLNS} is that a volume-averaged velocity is used, which dramatically simplifies the resulting
equation for the mixture velocity. With the help of formally matched asymptotic analysis the authors also develop new sharp interface models.
On the other hand, in \cite{Dai}, the authors analyze the multi-species Cahn--Hilliard--Darcy tumor model of~\cite{Chen}
in the case of constant and identical mobilities for all the species and establish the existence of a weak solution.
We actually believe it is important to analyze these multi-species models, as once
the tumor experiences quiescence and necrosis due to lack of nutrient, it will start secreting growth factors that induce the
development of new blood vessels towards the tumor, a process known as angiogenesis,
and it is through these new blood vessels that tumors tend to metastasize into other parts of the body.

In this contribution, we actually consider a multi-species tumor model posed in a smooth bounded domain
$\Omega \subset \RR^{d}$, $d \in \{2,3\}$, and over a reference time interval $(0,T)$ with no restriction
on the magnitude of $T$. Our model describes the evolution of proliferating tumor cells, necrotic tumor
cells and healthy host cells.  We denote the corresponding volume fractions as $\phip, \phid, \phih \in [0,1]$, respectively,
so that $\phip + \phid + \phih = 1$ almost everywhere in $\Omega \times (0,T)$.  By this relation, it suffices to track the
evolution of $\phip$ and $\phid$ in order to deduce the evolution of $\phih$; for this reason it is also
natural to assume that the vector $\bm{\varphi} := (\phip, \phid)^{\top}$ lies in the simplex
$\Triangle := \{ \bm{y} \in \RR^{2} :~0 \leq y_{1}, y_{2},~ y_{1} + y_{2} \leq 1 \} \subset \RR^{2}$.
This {\sl constraint}\/ will be one of the key points in our approach and we will explain below
how it is enforced by the equations.

In order to introduce our mathematical model we start by making some assumptions on the biological mechanisms experienced by the tumor cells.
Let $n$ denote the concentration of a nutrient chemical species that is present in $\Omega$.  Then, the proliferating tumor cells
may grow by consuming the nutrient, and may die by apoptosis or necrosis in case of nutrient deficiency.  Correspondingly,
necrotic cells will increase in mass due to the apoptosis/necrosis of proliferating tumor cells, but they may also
disintegrate back into basic components. On the other hand, the dynamics of the healthy cells occurs on a much slower timescale
than that of the tumor cells, since we expect that compared to the strictly regulated mitosis cycle of the healthy cells, in the tumor
certain growth-inhibited proteins have been switched off by mutations and leads to unregulated growth behavior
which is only limited by the supply of nutrients.

Assuming the cells are tightly packed and move together, we may introduce a cell velocity field $\vu$ which
is the same for all types of cells.  Letting $\bm{J}_{i}$, $i \in \{p, d, h \}$, denote the mass fluxes for the tumor cells,
then the general balance law for the volume fractions reads as
\begin{align}\label{proto:phi}
\de_{t} \varphi_{i} + \div( \varphi_{i}  \vu) = -\div \bm{J}_{i} + S_{i} \text{ for } i \in \{p,d,h \},
\end{align}
where, in view of the above discussion, we set $S_{h} = 0$, whereas $S_{p}, S_{d}$ may depend on $n$, $\phip$ and $\phid$.
Note that it is sufficient to determine the evolution of $\phip$ and $\phid$ in view of the fact
that $\phih$ can be determined at any point as $1-\phip - \phid$.
We assume that the tumor growth process tends to evolve towards (local) minima of the following
free energy functional $E(\phip, \phid)$ of Ginzburg--Landau type:
\begin{align}\label{Energy}
E(\phip, \phid) := \int_{\Omega} F(\phip, \phid) + \frac{1}{2} \abs{\nabla \phip}^{2} + \frac{1}{2} \abs{\nabla \phid}^{2} dx,
\end{align}
where $F$ is a multi-well configuration potential for the variables $\phip$ and $\phid$.

More precisely, we assume $F$ to be the sum of a smooth non-convex part $F_{1}$
and of a nonsmooth {\sl singular}\/ convex part $F_{0}$. Namely, $F_{0}$ is identically set to $+\infty$
outside the set $\overline{\Triangle} = \{ (s,r) \in \RR^2: ~s\geq 0, r \geq 0, s+r \leq 1 \}$ of the ``physically admissible'' configurations.
In other words, if $\int_{\Omega} F_{0}(\phip, \phid) \, dx$ is finite, then we necessarily have $\phip, \phid \geq 0$ and $\phip + \phid \leq 1$.
Moreover, as a further consequence, due to the fact that $\phih = 1 - \phip - \phid$, we also have $0 \leq \phih \leq 1$.
In other words, the finiteness of the ``configuration energy'' $\int_{\Omega} F(\phip, \phid) \, dx$
automatically ensures the natural bounds
\begin{align}\label{nat:bound}
   0 \leq \phip, \phid, \phih \leq 1 \text{ a.e. in } \Omega.
\end{align}
An example of singular potential $F_0$ that we can include in our analysis and that was already introduced in \cite{Wise} is
\begin{equation}\label{log}
\begin{aligned}
F_{0}(\phip, \phid) &:= \phip \log \phip + \phid \log \phid + (1-\phip-\phid) \log (1-\phip-\phid).
\end{aligned}
\end{equation}
This can be seen as a generalization of the standard one-component logarithmic potential
commonly used in the framework of Cahn--Hilliard equations (cf.~for example \cite{CF, MZ}).
Let us also comment that, in light of the above
considerations, the pure phase consisting of proliferating tumor cells is characterized by
the region $\{ \phip = 1,~\phid = \phih = 0\}$, whereas the pure phase corresponding to the necrotic
cells is the region $\{ \phid = 1,~\phip = \phih = 0\}$. The fluxes $\bm{J}_{i}$, $i = p,d$, are defined as follows:
\begin{align*}
\bm{J}_{i} = - M_{i} \nabla \mu_{i}, \quad \mu_{i} := \frac{\delta E}{\delta \varphi_{i}} = - \Laplace \varphi_{i} + F_{,\varphi_{i}} \text{ for } i = p,d,
\end{align*}
where $M_{p}, M_{d}$ are positive constants which can be different from each other.  We set $\bm{J}_{h} = - \bm{J}_{p} - \bm{J}_{d}$, then
upon summing up the three equations \eqref{proto:phi}, for $i = p,d,h$, using the fact that $\phip+\phid+\phih=1$ and $S_h=0$, we deduce the following relation:
\begin{align*}
\div \vu=S_p+S_d=:S_t \,.
\end{align*}
The velocity field $\vu$ is assumed to fulfill Darcy's law (cf.~\cite{Wise}):
\begin{align*}
\vu = - \nabla q - \phip \nabla \mu_{p} - \phid \nabla \mu_{d},
\end{align*}
where $q$ denotes the cell-to-cell pressure and the subsequent two terms have the meaning
of Korteweg forces. Regarding the nutrient $n$, since the time scale of nutrient diffusion is much faster (minutes) than the rate of cell
proliferation (days), the nutrient is assumed to evolve quasi-statically:
\begin{align}
0 = -\Laplace n + \phip n,
\end{align}
where the second term on the right-hand side models consumption by the proliferating tumor cells. Notice that we have neglected the nutrient uptake by host tissue because this is small compared with the
uptake by tumor cells (e.g., see \cite{Wise} for more details).
Moreover we have also neglected here the capillarity term $T_C=B(\phip, \phid)(n_C-n)$
considered, e.g., in \cite{Dai}, just for the sake of simplicity. Actually, we could include it in our analysis
assuming that the function $B(\phip, \phid)$ is smooth and non-negative and the coefficient
$n_C$ is strictly between $0$ and $1$ (cf, e.g., \cite[(2.6)]{Dai}). These assumptions
would indeed guarantee the maximum principle is still true for the modified nutrient equation.
Hence, setting $Q := \Omega \times (0,T)$ and $\Gamma := \de \Omega \times (0,T)$,
we are led to the following multi-species tumor model:
\begin{subequations}\label{Model}
\begin{alignat}{2}
\label{eq:p}
\de_{t}\phip & = M_{p} \Laplace \mu_{p} - \div (\phip \vu) + S_{p},  \quad \mu_{p} = F_{,\phip} - \Laplace \phip && \text{ in } Q, \\
\label{eq:d}
\de_{t}\phid & = M_{d} \Laplace \mu_{d} - \div (\phid \vu) + S_{d}, \quad \mu_{d} = F_{,\phid} - \Laplace \phid && \text{ in } Q, \\
\label{eq:Sp}
S_{p} &= \Sigma_{p}(n, \phip, \phid) + m_{pp} \phip + m_{pd} \phid && \text{ in } Q,\\
\label{eq:Sd}
S_{d} &= \Sigma_{d}(n, \phip, \phid) + m_{dp} \phip + m_{dd} \phid && \text{ in } Q, \\
\label{eq:u}
\div \vu & = S_{p} + S_{d} && \text{ in } Q, \\
\label{eq:q}
\vu & = - \nabla q - \phip \nabla  \mu_{p} - \phid \nabla \mu_{d} && \text{ in } Q, \\
\label{eq:n}
0 & = -\Laplace n + \phip n && \text{ in } Q,
\end{alignat}
\end{subequations}
where the precise assumptions on the mass source terms $\Sigma_p$, $\Sigma_d$ will be specified in Section~\ref{sec:mainresults}.

We just give here an example of source terms in \eqref{Model} that fulfill the assumption \eqref{struct:2} stated in Section~\ref{sec:mainresults}.
Namely, we may set
\begin{align}\label{sp}
S_p & = \lambda_{M} g(n) - \lambda_{A} \phip, \\
\label{sd}
S_d & = \lambda_A \phip - \lambda_L \phid,
\end{align}
for positive constants $\lambda_M, \lambda_A, \lambda_L$ and a bounded positive function $g$ such that $0 < g(s) \leq 1$.
The archetypal example is $g(s) = \max(n_{c}, \min(s,1))$ for some constant $n_{c} \in (0,1)$. The biological effects that we want to model
here are:  the growth of the proliferating tumor cells due to nutrient consumption at a constant rate $\lambda_M$, the death of
proliferating tumor cells at a constant rate $\lambda_A$, which leads to a source term for the necrotic cells,
and the lysing/disintegration of necrotic cells at a constant rate $\lambda_L$.

We supplement the system with the following boundary conditions:
\begin{align}\label{bc}
M_{i} \pdnu \mu_{i} - \varphi_{i} \vu \cdot \norma = 0, \quad \pdnu \varphi_{i} = 0, \quad  q = 0, \quad n  = 1 \text{ on } \Gamma,
\end{align}
where $\pdnu$ denotes the outer normal derivative to $\Gamma=\partial\Omega$,
and with the initial conditions
\begin{align}\label{init}
\phip(x,0) = \varphi_{p,0}(x), \quad \phid(x,0) = \varphi_{d,0}(x)  \text{ in } \Omega.
\end{align}
Note that, implicitly, we are also assuming that $\phih(x,0) = 1- \varphi_{p,0}(x) - \varphi_{d,0}(x)$ in $\Omega$.

\bigskip

Let us now compare \eqref{Model} and the model of \cite{Chen} which was analyzed by \cite{Dai}:
\begin{itemize}
\item We use a simpler nutrient equation \eqref{eq:n} which only accounts for diffusion and consumption by the proliferating tumor cells.
In particular, we have neglected the nutrient capillary source term for simplicity (though, as remarked before, such a term could
be treated by standard modifications).
\item In \cite{Chen}, the effect of a basement membrane on the growing tumor is also considered, which leads to additional coupling
of the model with a Cahn--Hilliard equation transported by the velocity $\vu$.  In this work we do not consider such effects.
\item The key distinction is that in our choice of a multi-well potential $F$ in \eqref{Energy}, we included interfacial energy for the
proliferating-necrotic tumor interface and also for the tumor-host interfaces.  On the other hand, in \cite{Chen} the free energy depends only on the total
tumor volume fraction $\varphi_{T} = \phip + \phid$, i.e., $E(\varphi_{T}) = \int_{\Omega} f(\varphi_{T}) + \frac{1}{2} \abs{\nabla \varphi_{T}}^{2} \, dx$
for scalar double-well potential $f$ with minima at $0$ and $1$.  This reduction to the total tumor volume fraction implies that the proliferating-necrotic
tumor interface in \cite{Chen} is not energetic.
\item More precisely, in \cite{Chen}, like in the multiphase models studied in \cite{Frieboes, Wise, Youssefpour}, the differentiation between proliferating
and necrotic tumor cells is done a posteriori based on the local density of nutrients after computing $\varphi_{T}$.  In contrast, our model \eqref{Model}
follows a similar approach to \cite{GLNS} in which $\phip$ and $\phid$ are computed without any post processing.
\item Moreover, we consider here different boundary conditions with respect to \cite{Chen}, where
a zero Dirichlet boundary datum was taken for the
chemical potentials, while here we consider a coupled condition for $\mu_i$ and $\vu$
(the first of \eqref{bc}). It is worth noting that the (easier) case of Dirichlet boundary conditions for $\mu_i$
could also be treated, but we preferred to handle \eqref{bc} which seems to be more reasonable from the modeling point of view.
On the contrary, the case of no-flux conditions for $\mu_i$ (which would also be meaningful) seems
not easy to be treated mathematically.
\item Finally, differently from \cite{Dai}, we assume here that different cell species are characterized by different mobility coefficients.
Actually, this choice gives rise to a number of mathematical complications. In particular, here we cannot reduce the evolution of the tumor cells
(as was done in \cite{Dai}) to a single Cahn--Hilliard equation coupled with a transport-type relation, but we need to consider a vectorial Cahn--Hilliard
system, for which, however, several issues are still open. On the other hand, here we get stronger regularity due to the fact that we do not have a transport equation anymore.
\end{itemize}
Mathematically speaking, the main novelty of our model, and also its main difficulty from the analytical point of view, comes from the singular component $F_{0}$
of the configuration potential coupled with the nonzero source terms in the Cahn--Hilliard relations
\eqref{eq:p}-\eqref{eq:d}. Indeed, integrating the first relations in \eqref{eq:p}, \eqref{eq:d}
we obtain an evolution law for the spatial mean values $y_{i} := \frac{1}{\abs{\Omega}} \int_{\Omega} \varphi_{i} \, dx$ of $\varphi_{i}$ for $i = p,d$
(cf.~\eqref{sist:mean} below) which is satisfied by any hypothetical solution
to the system.  Such a relation, however, does not involve directly the singular part $F_{0}$.  Hence, the evolution of $y_{p}, y_{d}$
are not automatically compatible with the physical constraint \eqref{nat:bound} and this compatibility (i.e., the fact that $y_{p}$,
$y_{d}$ remain well inside the set of meaningful values) has to be carefully proved (see Subsec.~\ref{subs:apr})
by assuming proper conditions on coefficients
and making a careful choice of the boundary conditions. In particular, the first condition in \eqref{bc} linking the boundary
values of $\vu$, $\varphi_{i}$ and $\mu_{i}$ seems to be necessary in order for our arguments to work. It is worth noting that very few results are nowadays available for multi-component Cahn--Hilliard systems.
Among these, we mention the recent contributions \cite{Boyer, Conti, Li}, related to the case of regular (multi-well)
potentials, whereas up to our knowledge, the only paper dealing with a singular multi-well potential
like the logarithmic one~\eqref{log} introduced in \cite{Wise} is \cite{CKRS}. Indeed, some estimates
proved in \cite{CKRS} will be used in the proof of our results.

Concerning the approach we employed to prove existence of weak solutions to system \eqref{eq:p}--\eqref{eq:n},
the first step is to consider a regularized version of this problem, which is obtained by replacing
the singular potential $F_0$ by a regular one depending on an approximation parameter $\eps>0$,
and also by introducing some suitable truncation functions.
We then present two independent methods that permit us to prove existence of a solution to the regularized system.
The first one is based on a further regularization and a Schauder fixed point argument, whereas
the second proof exploits a direct implementation of a Faedo-Galerkin scheme.
Actually, we have decided to detail two different proofs, since, in our opinion, both of them
present some independent interest. Indeed, the first method only exploits
elementary existence and uniqueness results methods for PDEs and also permits us to split
our complicated model into its basic components and show how they can be treated separately.
On the other hand, the Faedo-Galerkin method is more direct (no further regularizing terms are introduced),
and constructive (hence, it may be used for a numerical approximation of the problem).
%

\paragraph{Plan of the paper.} The assumptions and main results are stated in Section~\ref{sec:mainresults}.
The proof is carried out in the remainder of the paper and is subdivided into several steps: namely, in Section~\ref{sec:scheme}
regularized version of our model is introduced, where the singular part of the potential $F_{0}$ is replaced by a smooth
approximation and some terms are truncated in order to maintain some boundedness
properties that are used in the a priori estimates.
Moreover, still in Section~\ref{sec:scheme}, the fixed point method for existence
is outlined, whereas the alternative Faedo-Galerkin procedure is presented
in Section~\ref{sec:Galerkin}. Once existence is established for the regularized model,
in Section~\ref{sec:limit} some bounds that are
independent of the regularization parameters are derived in
order to pass to the limit in the approximation scheme via compactness
tools and to obtain in this way existence of weak solutions for the original system.


\section{Assumptions and main result}\label{sec:mainresults}
We start presenting our assumptions on parameters and data:
\begin{assump}\label{ass:pd}
\begin{enumerate}[label=$(\mathrm{A \arabic*})$, ref = $(\mathrm{A \arabic*})$]
\item $M_{p}, M_{d}$ are strictly positive constants.
\item We set $\Sigma := (\Sigma_p,\Sigma_d)$ and denote as $\underline{\underline{M}} = (m_{ij})$, $i,j \in \{p,d\}$,
the matrix of the coefficients in \eqref{eq:Sp}, \eqref{eq:Sd}.  Then we assume that there exists
$c \geq 0$ such that
\begin{align}\label{struct:Sig}
  \Sigma \in C^{0,1}(\RR^3; \RR^2), \quad \norm{\Sigma}_{L^{\infty}(\RR^3;\RR^2)} + \norm{D \Sigma}_{L^{\infty}(\RR^3;\RR^6)} \leq c.
\end{align}
In other words, $\Sigma$ is globally Lipschitz. More precisely, we assume that there exist a closed and sufficiently regular subset
$\Delta_{0}$ contained in the open simplex $\Delta$ and constants $K_{p,-},K_{p,+},K_{d,-},K_{d,+}\in \RR$,
with $K_{p,-}\le K_{p,+}$ and $K_{d,-}\le K_{d,+}$, such that $\Sigma(\RR^3)\subset [K_{p,-},K_{p,+}]\times [K_{d,-},K_{d,+}]$.
Moreover,
for any $\bm{x}=(x_p,x_d) \in [K_{p,-},K_{p,+}]\times [K_{d,-},K_{d,+}]$,
there holds
\begin{equation}\label{struct:2}
( \underline{\underline{M}} \bm{y} + \bm{x} ) \cdot \norma < 0 \text{ for all }\, \bm{y} \in \de \Delta_0,
\end{equation}
where $\norma$ denotes the outer unit normal vector to $\Delta_0$.
\item The potential $F$ is the sum of a convex part $F_{0}$ and of a (possibly nonconvex) perturbation $F_{1}$.  More precisely, we assume that
$F_{0} : \RR^{2} \to [0,+\infty]$, with the {\it effective domain} of $F_0$ (i.e., the set where $F_0$ assumes {\it finite} values) being given
either by $\Delta$ or by the closure $\overline{\Delta}$.  Moreover, we assume $F_1 \in C^{1,1}(\RR^{2})$ with
\begin{align}
| \nabla F_{1}(s,r) | \leq C (1 + |s | + |r| ) & \quad \forall r, s \in \RR,
\end{align}
while $F_0 \in C^{1}(\Delta ;[0,\infty))$, i.e., $F_0$ is smooth once restricted to the
simplex $\Delta$ and there exists constants $c_{1}, c_{3} > 0$ and $c_{2}, c_{4}\ge 0$ such that
\begin{align}
\label{F0:super:quad}
F_0(s,r) \geq c_{1}( \abs{s}^{2} + \abs{r}^{2}) - c_{2} \quad \forall (s,r) \in \overline{\Delta},
\end{align}
and
\begin{align}\label{Key:Delta}
\nabla F_{0}(s,r) \cdot (s - S, r - R)^{\top} \ge c_{3} | \nabla F_{0}(s,r)| - c_{4}
\end{align}
for all $(s,r) \neq (S,R) \in \Delta$.
\item The initial conditions satisfy $\varphi_{p,0}, \varphi_{d,0} \in H^{1}(\Omega)$ with
\begin{align}\label{in:Delta}
 0 \leq \varphi_{p,0}, \quad 0\le \varphi_{d,0}, \quad \varphi_{p,0} + \varphi_{d,0} \leq 1 \,
   \text{ a.e. in } \Omega.
\end{align}
Moreover, the mean values
\begin{equation}\label{def:Y0}
Y_{i,0}:= \frac{1}{\abs{\Omega}} \int_{\Omega} \varphi_{i,0}(x) \, dx
\end{equation}
for $i = p,d$ satisfy
\begin{align}\label{assump:initial<1}
 ( Y_{p,0} , Y_{d,0} ) \in \inte \Delta_0.
\end{align}
Finally, we assume that
\begin{align}\label{fin:ener}
F_{0}(\varphi_{p,0},\varphi_{d,0}) \in L^{1}(\Omega).
\end{align}
%
%
\end{enumerate}
\end{assump}

\smallskip
\noindent
\paragraph{Examples.}  In order to clarify the above assumptions, and particularly those
regarding the potential $F$, we introduce one example
which is particularly significant and will be considered as a model case in the sequel.
Namely, we consider the multi-phase logarithmic potential
\begin{equation}\label{log:pot}
\begin{aligned}
F_{0}(s,r) &:= s \log s + r \log r + (1-s-r) \log (1-s-r), \\
F_{1}(s,r) & := \frac{\chi}{2} \big ( r (1-r) + s(1-s) + (1-r-s)(r+s) \big ),
\end{aligned}
\end{equation}
for a fixed positive constant $\chi$.

Then, the assumption \eqref{F0:super:quad} is easily fulfilled by \eqref{log:pot} due to the boundedness of the simplex $\Delta$,
and the assumption \eqref{Key:Delta} is also fulfilled as we will prove in Lemma~\ref{l:2.2} below.
Moreover, as a consequence of \eqref{Key:Delta} we obtain by interchanging the roles of $(s,r)$ and $(S,R)$ an analogous inequality
\begin{align*}
\nabla F_{0}(S,R) \cdot ( S - s,  R - r)^{\top} \ge c_{3} | \nabla F_{0}(S,R)| - c_{4}.
\end{align*}
Then, a short computation shows that
\begin{equation}\label{F0:CKRS:cond}
\begin{aligned}
& \big ( \nabla F_{0} (s,r) - \nabla F_{0}(S,R) \big ) \cdot (s - S, r - R)^{\top} \\
& \quad = \nabla F_{0}(s,r) \cdot (s - S, r - R)^{\top} + \nabla F_{0}(S,R) \cdot (S - s, R - r)^{\top} \\
& \quad \geq c_{3} | \nabla F_0(s,r)| + c_{3} | \nabla F_0(S,R)| - 2 c_4 \\
& \quad \geq c_3 | \nabla F_0(s,r) - \nabla F_0(S,R)| - 2 c_4.
\end{aligned}
\end{equation}
In particular, the inequality \eqref{F0:CKRS:cond} together with \eqref{F0:super:quad} shows that $F_0$ fulfills the hypotheses
of \cite[Prop.~2.10]{CKRS}. This property would be important later when we consider the Yosida approximation of $F_0$ for the derivation of uniform estimates.

Let us also sketch a couple of examples of functions $\Sigma_p, \Sigma_d$ and $\underline{\underline{M}}$
that fulfill the assumption \eqref{struct:2}.
\begin{itemize}
\item First, we consider the source terms
$S_p$, $S_d$ introduced in \eqref{sp}, \eqref{sd},
for positive constants $\lambda_M, \lambda_A, \lambda_L$ and a bounded positive function $g$ such that $0 < g(s) \leq 1$.
For instance, one can take $g(s) = \max(n_{c}, \min(s,1))$ for some constant $n_{c} \in (0,1)$.
Note that with this choice one can take $K_{p,-}=0$, $K_{p,-}=\lambda_M$, $K_{d,-}=K_{d,+}=0$.
In what follows the mean of a function $f$ over $\Omega$ is denoted as $(f)\OO$. Then, integrating \eqref{eq:p} and
\eqref{eq:d} over $\Omega$, and applying the boundary conditions \eqref{bc} leads to the following
ODE system for the mean values $Y_{p} := (\phip)\OO$ and $Y_{d} := (\phid)\OO$:
\begin{align*}
  \frac{d}{dt} \left ( \begin{array}{c}
  Y_{p} \\ Y_{d} \end{array} \right) = \left ( \begin{array}{cc} -\lambda_A & 0 \\ \lambda_A & -\lambda_L
  \end{array} \right ) \left ( \begin{array}{c} Y_p \\ Y_d
  \end{array}
    \right ) + \left ( \begin{array}{c} \lambda_M (g(n))\OO \\ 0 \end{array} \right ).
\end{align*}
The matrix $\underline{\underline{M}}$ is invertible with eigenvalues $\{-\lambda_A, -\lambda_L\}$, hence the fixed point
\begin{align*}
\left ( \begin{array}{c} Y_{p}^{*} \\ Y_{d}^{*} \end{array} \right ) = - \underline{\underline{M}}^{-1} \left ( \begin{array}{c} \lambda_M (g(n))\OO \\ 0 \end{array} \right )
  = \left ( \begin{array}{c} \frac{\lambda_M}{\lambda_A} (g(n))\OO  \\ \frac{\lambda_M}{\lambda_L} (g(n))\OO \end{array} \right )
\end{align*}
is asymptotically stable.  Under the following constraints on the rates:
\begin{align*}
  \lambda_M (\lambda_A + \lambda_L) < \lambda_A \lambda_L, \quad \lambda_A < 2 \lambda_L,
\end{align*}
we can easily show that $(Y_{p}^{*}, Y_{d}^{*})$ lies in the interior of the simplex $\Delta$, and \eqref{struct:2} holds when we take $\Delta_{0}$
to be a ball centered at $(Y_{p}^{*}, Y_{d}^{*})$ with sufficiently small radius $\eta > 0$.  Indeed, thanks to $n_{c} > 0$ we easily see
that $Y_{p}^{*}, Y_{d}^{*} > 0$, while using $g \leq 1$ shows that
\begin{align*}
Y_{p}^{*} + Y_{d}^{*} \leq \lambda_M \left ( \frac{1}{\lambda_A} + \frac{1}{\lambda_L} \right ) < 1
\end{align*}
when we assume the hypothesis $\lambda_M (\lambda_A + \lambda_L) < \lambda_A \lambda_L$.  Furthermore, taking a parameterization of the
circle $\de \Delta_0$ as $(\eta \cos \theta + Y_{p}^{*}, \eta \sin \theta + Y_{d}^{*})$ for $\theta \in [0,2 \pi]$ with normal $\norma = (\cos \theta, \sin \theta)$,
a short computation shows that
\begin{align*}
& \left [ \underline{\underline{M}} \left ( \begin{array}{c} \eta \cos \theta + Y_{p}^{*} \\ \eta \sin \theta + Y_{d}^{*} \end{array}
\right ) + \left ( \begin{array}{c} \lambda_M (g(n))\OO \\ 0 \end{array} \right ) \right ] \cdot \left ( \begin{array}{c} \cos \theta \\ \sin \theta \end{array} \right ) \\
& \quad = - \lambda_A \eta \cos^{2} \theta + \lambda_A \eta \cos \theta \sin \theta - \lambda_L \eta \sin^{2} \theta \\
& \quad \leq - \frac{\lambda_A}{2} \eta \cos^{2} \theta - \left ( \lambda_L - \frac{\lambda_A}{2} \right ) \eta \sin^{2} \theta \leq - \frac{1}{2} \min ( \lambda_A, 2 \lambda_L - \lambda_A) \eta < 0
\end{align*}
under the assumption $2 \lambda_L > \lambda_A$.
\item As a second model case, for $\lambda>0$ we take $\underline{\underline{M}}$ as $-\lambda$ times the identity matrix
(a more general negative definite diagonal matrix could also be considered) and
\begin{equation}\label{ex:M}
  \Sigma(\phip,\phid,n) = \bm{k} + \Sigma_0(\phip,\phid,n),
\end{equation}
where $\bm{k} = (\lambda/3,\lambda/3)$ and $\Sigma_0$ is a $C^1$ function of its arguments such that
\begin{equation}\label{struct:Sig2}
  \| \Sigma_0 \|_{L^\infty(\RR^3;\RR^2)} \le K
\end{equation}
for some $K>0$. This corresponds in fact to
\begin{equation}\label{ex:M2}
 \left ( \begin{array}{c} S_{p} \\ S_{d} \end{array} \right ) = \underline{\underline{M}}  \left ( \begin{array}{c} \phip - 1/3 \\ \phid - 1/3 \end{array} \right )  + \Sigma_0(\phip,\phid,n)
\end{equation}
and in particular we can take $K_{p,-}=K_{d,-}=\lambda/3-K$
and $K_{p,+}=K_{d,+}=\lambda/3+K$. Note also that the point $(1/3,1/3)$ can be seen as the ``center'' of the simplex (indeed
it represents the configuration where all the species have the same proportion). Hence, here we are decomposing
$(S_p,S_d)$ as the sum of an affine part that tends to keep the configuration close to the center of the simplex,
plus the perturbation $\Sigma_0$.

\bigskip

With this choice, we now check that, at least if $\lambda$ is large enough (depending on $K$), then there exists
$\eps>0$ such that
$$
  Y_i = \eps~\Rightarrow~Y_i'>0,
   \qquad ( 1 - Y_p - Y_d ) = \eps~\Rightarrow~Y_p' + Y_d' < 0.
$$
Indeed, let $a\in[-K,K]$. Then, for $Y_i = \eps$ we have
$$
  - \lambda \Big( Y_i - \frac13 \Big) + a
   = \lambda \Big( \frac13 - \eps \Big) + a
   \ge \lambda \Big( \frac13 - \eps \Big) - K
$$
which, for $\eps<1/3$, is greater than $0$ if $\lambda$ is large enough
compared to $K$. Analogously, for $a,b\in[-K,K]$ and $Y_p+Y_d=1-\eps$,
$$
  - \lambda \Big( Y_p + Y_d - \frac23 \Big) + a + b
   = - \lambda \Big( \frac13 - \eps \Big) + a + b
   \le - \lambda \Big( \frac13 - \eps \Big) + 2 K
$$
which is negative under the same conditions as before. Consequently,
one can take $\Delta_0$ as the set $\{Y_p\ge \eps,~Y_d\ge \eps, Y_p+Y_d \le 1-\eps\}$,
which is larger for $\eps$ closer to 0. Namely, $\Delta_0$ is closer
to $\Delta$ if the constant $\lambda$ is big compared to the
$L^\infty$-norm of the perturbation $\Sigma_0$.
To be more precise, we have to remark that the above choice of $\Delta_0$ is
nonsmooth (and the normal $\norma$ is not defined in the vertices). On the other
hand it is easy to check that taking a smaller $\Delta_0$ whose vertices
are smoothed out the above computations are still effective.
\end{itemize}

\smallskip

We can now define a suitable notion of weak solution to the initial-boundary value
problem for system \eqref{eq:p}-\eqref{eq:n}:
\begin{defn}\label{defn:Soln}
We say that a multiple $(\phip,\mu_{p},\eta_{p},\phid,\mu_{d},\eta_{d},\vu,q,n)$
is a weak solution to the multi-species tumor model \eqref{Model} over the interval $(0,T)$ if
\begin{enumerate}
\item [\bf (1)]~~the following regularity
properties hold:
\begin{subequations}
\begin{align}
\label{reg:fhi}
\varphi_{i} & \in H^{1}(0,T;H^{1}(\Omega)') \cap L^{\infty}(0,T;H^{1}(\Omega)) \cap L^{2}(0,T;H^{2}(\Omega)),\\
\notag & \quad \text{ with } 0 \leq \varphi_{i} \leq 1, \quad \phip + \phid \leq 1 \text{ a.e. in } Q, \\
\label{reg:mu}
\mu_{i} & \in L^{2}(0,T;H^{1}(\Omega)),\\
\label{reg:eta}
\eta_{i} & \in L^{2}(Q),\\
\label{reg:u}
\vu & \in L^{2}(Q) \text{ with } \div \vu \in L^{2}(Q),\\
\label{reg:q}
q & \in L^{2}(0,T;H^{1}_{0}(\Omega)),\\
\label{reg:n}
n & \in ( 1 + L^{2}(0,T;H^{2}(\Omega) \cap H^{1}_{0}(\Omega))) ,\quad 0\leq n \leq 1 \text{ a.e. in } Q,
\end{align}
\end{subequations}
for $i = p,d$.
\item [\bf (2)]~~Equations \eqref{eq:p}-\eqref{eq:n} hold, for a.e.~$t\in (0,T)$
and for $i=p,d$, in the following weak sense:
\begin{subequations}
\begin{align}
& \inner{ \de_{t}\varphi_{i}}{ \zeta}  + \int_{\Omega} M_{i} \nabla \mu_{i} \cdot \nabla \zeta - \varphi_{i} \vu \cdot \nabla \zeta \, dx = \int_{\Omega} S_{i} \zeta \, dx \quad \forall \zeta \in H^{1}(\Omega), \label{eq:fhiw} \\
& \int_{\Omega} \mu_{i} \zeta \, dx = \int_{\Omega} \nabla \varphi_{i} \cdot \nabla \zeta + \eta_{i} \zeta + F_{1,\varphi_i}(\varphi_{p},\varphi_{d})\zeta \, dx \quad \forall \zeta \in H^1(\Omega),   \label{eq:muw} \\
& \int_{\Omega} \vu \cdot \nabla \xi \, dx  = - \int_{\Omega} (S_{p} + S_{d}) \xi \, dx \quad \forall \xi \in H^{1}_{0}(\Omega),   \label{eq:uw} \\
& \int_{\Omega} \vu \cdot \bm{\zeta} \, dx = \int_{\Omega} - \nabla q \cdot \bm{\zeta} - \varphi_{p} \nabla  \mu_{p} \cdot \bm{\zeta} - \varphi_{d} \nabla \mu_{d} \cdot \bm{\zeta} \, dx \quad \forall \bm{\zeta} \in (L^{2}(\Omega))^{d},   \label{eq:qw}\\
& 0 = -\Laplace n + \varphi_{p} n \quad \text{ a.e.~in } \Omega, \label{eq:nw}\\
& \eta_{i} = F_{0,\varphi_{i}}(\varphi_{p},\varphi_{d}) \quad \text{ a.e.~in } \Omega, \label{eq:etaw} \\
& S_{p} = \Sigma_{p}(n, \phip, \phid) + m_{pp} \phip + m_{pd} \phid \quad \text{ a.e.~in } \Omega, \label{eq:Spw}\\
& S_{d} = \Sigma_{d}(n, \phip, \phid) + m_{dp} \phip + m_{dd} \phid \quad \text{ a.e.~in } \Omega. \label{eq:Sdw}
\end{align}
Moreover, there hold the initial conditions
\begin{align}
& \phip(x,0) = \varphi_{p,0}(x), \quad \phid(x,0) = \varphi_{d,0}(x) \quad \text{ a.e.~in } \Omega, \label{eq:initw}
\end{align}
\end{subequations}
where $\inner{\cdot}{\cdot}$ denotes the duality pairing between $H^{1}(\Omega)$ and its dual $H^{1}(\Omega)'$.
%
%
\end{enumerate}
\end{defn}
\noindent%
It is worth noting that now the first two boundary conditions in \eqref{bc} have been incorporated in the weak formulations \eqref{eq:fhiw}, \eqref{eq:muw}.
Moreover, the boundary conditions $q = 0$ and $n = 1$ a.e. on $\Sigma$ are built into the function spaces in \eqref{reg:q} and \eqref{reg:n}.
Furthermore, the attainment of the initial conditions \eqref{eq:initw} is due to the continuous embedding
\begin{align*}
H^{1}(0,T;H^{1}(\Omega)') \cap L^{\infty}(0,T;H^{1}(\Omega)) \subset C^{0}([0,T];L^{2}(\Omega)),
\end{align*}
and thus the initial conditions \eqref{eq:initw} makes sense as equalities in the space $L^{2}(\Omega)$.
Finally, it is worth saying some words about the auxiliary variables $\eta_{p}$, $\eta_{d}$. Actually, using the language of
convex analysis, relations \eqref{eq:etaw} for $i=p,d$ may be equivalently stated by saying that the vector
$\bm{\eta}=(\eta_p,\eta_d)$ belongs at almost every point $(x,t)\in Q$ to the subdifferential
$\partial F_0(\phip,\phid)$ which is a maximal monotone graph in $\RR^2 \times \RR^2$.  In principle such an object may be a multivalued mapping; here, however,
in view of the fact that $F_0$ is assumed to be smooth in $\Delta$ (cf.~(A3)), $\partial F_0$
may be simply identified with the gradient $\nabla F_0$. On the other hand, the use of some
techniques from convex analysis and monotone operators will be required in the last part of the proof.

\bigskip

We are now ready to state the main result of this paper
\begin{thm}\label{thm:main}
Let the hypotheses stated in Assumption~\ref{ass:pd} hold. Then there exists at
least one weak solution $(\phip,\mu_{p},\eta_{p}, \phid,\mu_{d}, \eta_{d}, \vu, q, n)$ to the multi-species
tumor model \eqref{Model} in the sense of Definition \ref{defn:Soln}.
\end{thm}

Before we prove our main result, we show that the convex part $F_{0}$ of the
model potential \eqref{log:pot} satisfies the assumption \eqref{Key:Delta}.

\begin{lemma}\label{l:2.2}
Let $F_{0}$ be defined as
\begin{align*}
F_0(s,r) = s\log s + r \log r + (1-s-r) \log (1-s-r)
\end{align*}
and let $\Delta_{0}$ be a compact subset of $\Delta$.  Then there exist positive constants $c_{*}, C_{*}$ depending only on $\Delta_{0}$ such that \eqref{Key:Delta} holds.
\end{lemma}

\begin{proof}
For any $(S,R) \in \Delta_{0}$, computing the gradient of $F_{0}$ leads to
\begin{align*}
\nabla F_{0}(s,r) \cdot (s-S, r-R)^{\top} & = (s - S) \log s + (r-R) \log r \\
& \quad + ((R+S) - (r+s)) \log (1 - (r+s)).
\end{align*}
For $s, S \in (0,1)$ it holds that
\begin{align*}
(s-S) \log s = \begin{cases}
> 0 & \text{ if } s < S, \\
< 0 & \text{ if } s > S, \\
= 0 & \text{ if } s = S,
\end{cases} \quad \text{ and } \quad (s - S) \log s \to \begin{cases} 0 & \text{ as } s \to 1, \\
\infty & \text{ as } s \to 0.
\end{cases}
\end{align*}
In particular, the function $(s-S) \log(s)$ is bounded from below by some negative constant.  Hence, there exists a constant $d_{1} \geq 0$ such that
\begin{align*}
(s - S) \log s \geq \frac{S}{2} | \log s | - d_{1},
\end{align*}
and it is clear that as $S \in (0,1)$ we can choose $d_{1}$ independent of $S$.  In a similar fashion, there exists a constant $d_{2} \geq 0$ (that can
be chosen independent of $R$) such that
\begin{align*}
(r - R) \log r \geq \frac{R}{2} | \log r | - d_{2}.
\end{align*}
Lastly, as $(s,r) \in \Delta$ we have $r + s \in (0,1)$ and consequently there exists a constant $d_{3} \geq 0$ independent of $R,S$ such that
\begin{align*}
((R + S) - (r+s)) \log (1 - (r+s)) \geq \frac{1 - (R+S)}{2} | \log (1 - (r+s))| - d_{3}.
\end{align*}
Summing the above then yields
\begin{align*}
& \nabla F_{0}(s,r) \cdot (s - S, r - R)^{\top} \\
& \quad \ge \frac{1}{2} \min (R, S, 1-(R+S))  \big ( | \log r | + | \log s | + | \log (1 - (r+s)) | \big ) - C(d_{1}, d_{2}, d_{3}) \\
& \quad \ge \frac{1}{4} \min (R, S, 1-(R+S)) | \nabla F_{0}(s,r) | - C(d_{1}, d_{2}, d_{3}).
\end{align*}
Now for $(R,S) \in \Delta_{0}$, we see that
\begin{align*}
\min (R, S, 1-(R+S)) \ge c_{*} > 0
\end{align*}
for some constant $c_{*}>0$ depending only on $\Delta_{0}$. This concludes the proof of Lemma~\ref{l:2.2}.
\end{proof}


\section{Approximation scheme}\label{sec:scheme}

In order to start our existence proof, we introduce a regularized version of our model. First of all,
for $\eps \in (0,1)$ intended to go to 0 in the limit, we consider a convex function
\begin{equation}\label{Fepsilon}
F_{\eps} : \RR^2 \to [0,+\infty)
\end{equation}
such that, for any $\eps \in (0,1)$, $\nabla F_{\eps}$ is globally Lipschitz continuous. Moreover, we assume that $F_{\eps}$ converges
in a suitable sense to $F_{0}$ as $\eps\searrow 0$.
Various choices are possible for $F_{\eps}$, but in light of the analysis below, we take $F_{\eps}$ as the
Moreau--Yosida approximation of $F_0$ (cf.~\cite{Brezis}), which is defined as
\begin{align*}
F_{\eps}(s,r) := \min_{(p,q) \in \RR^{2}} \left ( \frac{1}{2 \eps} | (p - s, q - r)|^{2} + F_0(p,q) \right ) \quad \text{ for } \eps \in (0,1).
\end{align*}
It is well-known that $F_\eps$ is convex and differentiable with derivative $\nabla F_\eps$ that is globally Lipschitz continuous with Lipschitz
constant scaling with $\frac{1}{\eps}$.  More importantly, thanks to the fact that $F_0$ satisfies \eqref{F0:CKRS:cond}, it turns out that $F_0$
fulfills the hypothesis of \cite[Prop.~2.10]{CKRS}, whence, by \cite[Prop.~2.13]{CKRS}, there exist positive constants $c_{*}, C_{*}$ such that
\begin{align}\label{Feps:CKRS}
c_{*} | \nabla F_\eps (s,r) - \nabla F_\eps(S,R)| \leq ( \nabla F_\eps (s,r) - \nabla F_\eps (S,R)) \cdot (s - S, r - R)^{\top} + C_{*}
\end{align}
for all $(s,r) \neq (S,R) \in \RR^{2}$.  In particular an analogue of \eqref{F0:CKRS:cond} also holds for $F_\eps$ with constants $c_*, C_*$ independent of $\eps \in (0,1)$.

\smallskip

A difficulty concerned with the regularization of $F_{0}$ stands in the fact that $F_{\eps}$ is no longer a singular function;
consequently, the uniform boundedness
properties  $0 \leq \phip$, $0 \le \phid$, $\phip + \phid \leq 1$
are not expected to hold in the approximation. For this reason, in order that the a priori estimates still
work, some terms have to be truncated in the regularized system.
In addition to that, we also include a number of regularizing terms depending
by a further parameter $\delta > 0$ which is intended to go to $0$ in the limit. Finally, we remove the explicit dependence on $\vu$ in the equations
and rewrite the transport terms in \eqref{eq:p}, \eqref{eq:d} directly in terms of the pressure $q$.
Hence, introducing the cutoff operator
\begin{equation*}
  T(r):= \max\big\{0,\min\{1,r\}\big\},
\end{equation*}
our regularized system takes the form
\begin{subequations}\label{delta:eps:Reg}
\begin{alignat}{3}
\label{eq:pa1}
 \de_{t} \phip &= M_{p} \Laplace \mu_{p} + \div (T(\phip) \nabla q) + \div \big( T(\phip)^2 \nabla \mu_p + T(\phip) T(\phid) \nabla \mu_d \big) + S_{p},  \\
\label{eq:pa2}
 \mu_{p} & = - \delta \Delta \de_t \phip + F_{\eps,p} (\phip, \phid) + F_{1,p} (\phip, \phid) - \Laplace \phip,\\
\label{eq:da1}
 \de_{t} \phid & = M_{d} \Laplace \mu_{d} + \div (T(\phid) \nabla q) + \div \big( T(\phip) T(\phid) \nabla \mu_p + T(\phid)^2 \nabla \mu_d \big) + S_{d},  \\
\label{eq:da2}
\mu_{d} & = - \delta \Delta \de_t \phid + F_{\eps,d} (\phip, \phid) + F_{1,d} (\phip, \phid) - \Laplace \varphi_{d},\\
\label{eq:Spa}
 S_{p} & = \Sigma_p(n, \phip ,\phid) + m_{pp} \phip + m_{pd} \phid, \\
\label{eq:Sda}
 S_{d} & = \Sigma_d(n,\phip, \phid) + m_{dp} \phip + m_{dd} \phid, \\
\label{eq:qa}
 \delta \de_t q & = \Delta q - \delta \Delta^{2} q + \div \big( T(\phip) \nabla  \mu_{p} + T(\phid) \nabla \mu_{d} \big) + S_p + S_d,\\
\label{eq:na}
 0 & = -\Laplace n + T(\phip) n,
\end{alignat}
\end{subequations}
furnished with the initial and boundary conditions
\begin{subequations}
\begin{alignat}{3}
\label{add:init}  \phip(0) = \varphi_{p,0,\delta}, \quad \phid(0) = \varphi_{d,0,\delta}, \quad q(0) & = 0 \text{ in } \Omega, \\
\label{add:coup} M_{i} \pdnu \mu_{i} + T(\varphi_{i}) (\nabla q + T(\phip) \nabla \mu_{p} + T(\phid) \nabla \mu_{d}) \cdot \bm{n} & = 0 \text{ on } \Gamma, \\
\label{add:dir} \pdnu \varphi_{i} = 0, \quad n = 1, \quad q = 0, \quad \Delta q & = 0 \text{ on } \Gamma,
\end{alignat}
\end{subequations}
where for each $\delta > 0$, $i = p,d$, the initial data $\varphi_{i,0,\delta} \in H^{2}_{\bm{n}}(\Omega)$ is defined as the solution
$f_i$ to
\begin{align}\label{Initial:data:Neu}
- \delta \Delta f_{i} + f_{i} = \varphi_{i,0} \text{ in } \Omega, \quad \pdnu f_{i} = 0 \text{ on } \Gamma.
\end{align}
We have used here the notation $H^2_{\bm{n}}(\Omega)$ for the space of $H^2(\Omega)$-functions satisfying
homogeneous Neumann boundary condition on $\Gamma$.
Then, it is well-known that, for each $\delta \in (0,1]$, $f_{i} \in H^{2}_{n}(\Omega)$. More precisely,
testing \eqref{Initial:data:Neu} by $f_{i}$ and $-\Delta f_{i}$, respectively, one obtains
\begin{equation}\label{ic:delta:Unif:H1}
\begin{aligned}
 2 \delta \| \nabla f_{i} \|_{L^{2}(\Omega)}^{2} + \| f_{i} \|_{L^{2}(\Omega)}^2 & \leq \| \varphi_{i,0} \|_{L^{2}(\Omega)}^2, \\
 2 \delta \| \Delta f_{i} \|_{L^{2}(\Omega)}^2 + \| \nabla f_{i} \|_{L^{2}(\Omega)}^2 & \leq \| \nabla \varphi_{i,0} \|_{L^{2}(\Omega)}^2.
\end{aligned}
\end{equation}
Furthermore, elliptic regularity arguments yield the additional estimate
\begin{align}\label{initialcond:delta}
\| f_{i} \|_{H^{2}(\Omega)} \leq C \big ( \| \Delta f_{i} \|_{L^{2}(\Omega)} + \| f \|_{L^{2}(\Omega)} \big ) \leq C \big ( 1 + \delta^{-\frac{1}{2}} \big ) \| \varphi_{i,0} \|_{H^{1}(\Omega)}.
\end{align}


\subsection{Auxiliary Cahn--Hilliard problem}

Fix now $\bar{q} \in L^2(0,T;H^1(\Omega))$ and $\bar{n} \in L^2(Q)$ with $0 \leq \bar{n} \leq 1$ almost everywhere in $Q$.  Then, we first consider
\begin{subequations}\label{Aux:CH}
\begin{alignat}{3}
\label{eq:pa12}  \de_{t} \phip & = M_{p} \Laplace \mu_{p} + \div (T(\phip) \nabla \bar{q}) + \div \big( T(\phip)^2 \nabla \mu_p + T(\phip) T(\phid) \nabla \mu_d ) + S_{p},  \\
\label{eq:pa22}
 \mu_{p} & = - \delta \Delta \de_t \phip + F_{\eps,p} (\phip, \phid) + F_{1,p} (\phip, \phid) - \Laplace \phip,\\
\label{eq:da12}
\de_{t} \phid & = M_{d} \Laplace \mu_{d} + \div (T(\phid) \nabla \bar{q}) + \div \big( T(\phip) T(\phid) \nabla \mu_p + T(\phid)^2 \nabla \mu_d ) + S_{d},  \\
\label{eq:da22}
\mu_{d} & = - \delta \Delta \de_t \phid + F_{\eps,d} (\phip, \phid) + F_{1,d} (\phip, \phid) - \Laplace \phid,\\
\label{eq:Spa2}
S_{p} & = \Sigma_p(\bar{n},\phip, \phid) + m_{pp} \phip + m_{pd} \phid, \\
\label{eq:Sda2}
S_{d} & = \Sigma_d(\bar{n},\phip, \phid) + m_{dp} \phip + m_{dd} \phid,
\end{alignat}
\end{subequations}
complemented with the initial and boundary conditions \eqref{add:init}-\eqref{add:dir}.  The above is a Cahn--Hilliard system with source term.
Note that $\bar{q}$ and $\bar{n}$ are given.  Existence of a solution can be proved
for instance via a Galerkin approximation, and we will only derive the necessary a priori estimates.

\begin{lemma}\label{lem:CH}
For each $\eps \in (0,1)$, $\delta \in (0,1)$, suppose \eqref{struct:Sig} holds, and $F_{\eps} : \RR^{2} \to [0,+\infty)$ and $F_{1}: \RR^{2} \to \RR$ are given such that $\nabla F_{\eps}$, $\nabla F_{1}$ are globally Lipschitz continuous.  Then, for given $\bar{q} \in L^{2}(0,T;H^{1}(\Omega))$ and $\bar{n} \in L^{2}(Q)$, there exists a \emph{unique} weak solution $(\phip, \mu_p, \phid, \mu_d)$ to \eqref{Aux:CH} in the following sense:
\begin{enumerate}
\item[\bf(1)] the functions have the following regularity properties:
\begin{align*}
\varphi_{i} \in H^{1}(0,T;H^{2}(\Omega)), \quad \mu_{i} \in L^{2}(0,T;H^{1}(\Omega)),
\end{align*}
with
\begin{align*}
\varphi_{i}(0) = \varphi_{0,i,\delta} \text{ in } \Omega.
\end{align*}
\item[\bf(2)] Equations \eqref{eq:pa22}, \eqref{eq:da22}, \eqref{eq:Spa2} and \eqref{eq:Sda2} hold a.e. in $Q$, and equations \eqref{eq:pa12} and \eqref{eq:da12} hold for a.e. $t \in (0,T)$ in the following weak sense:
\begin{align*}
0 & = \int_{\Omega} (\de_{t} \varphi_{p} - S_{p})\zeta + \big ( M_{p} \nabla \mu_{p} + T(\varphi_{p}) \nabla \bar{q} + T(\varphi_{p})^{2} \nabla \mu_{p} + T(\varphi_{p}) T(\varphi_{d}) \nabla \mu_{d} \big ) \cdot \nabla \zeta \, dx, \\
0 & = \int_{\Omega} (\de_{t} \varphi_{d} - S_{d})\zeta + \big ( M_{d} \nabla \mu_{d} + T(\varphi_{d}) \nabla \bar{q} + T(\varphi_{d})^{2} \nabla \mu_{d} + T(\varphi_{p}) T(\varphi_{d}) \nabla \mu_{p} \big ) \cdot \nabla \zeta \, dx
\end{align*}
for all $\zeta \in H^{1}(\Omega)$.
\end{enumerate}
\end{lemma}
\begin{proof}
In the following, the symbol $C$ denotes positive constants that depend only on the given parameters of the problem
and in particular are independent of $\phip$, $\mu_{p}$, $\phid$, $\mu_{d}$.

\paragraph{First estimate.} Testing \eqref{eq:pa12} with $\mu_{p}$, \eqref{eq:pa22} with $\de_{t} \phip$, \eqref{eq:da12} with $\mu_{d}$ and \eqref{eq:da22} with $\de_{t} \phid$, and summing yields
\begin{equation} \label{energyx}
\begin{aligned}
 & \frac{d}{dt} E_{\eps}(\phip, \phid)  + M_p \| \nabla \mu_p \|_{L^{2}(\Omega)}^2 + M_d \| \nabla \mu_d \|_{L^{2}(\Omega)}^2\\
& \qquad  + \delta \| \nabla \de_t \phip \|_{L^{2}(\Omega)}^2 + \delta \| \nabla \de_t \phid \|_{L^{2}(\Omega)}^2 + \int_\Omega | T(\phip) \nabla \mu_p + T(\phid) \nabla \mu_d |^2 \, dx \\
& \quad  = \int_\Omega ( S_p \mu_p + S_d \mu_d ) \, dx - \int_\Omega \big( T(\phip) \nabla \bar{q} \cdot \nabla \mu_p + T(\phid) \nabla \bar{q} \cdot \nabla \mu_d \big) \, dx.
\end{aligned}
\end{equation}
Note that the additional (nonnegative) contribution $\int_\Omega | T(\phid) \nabla \mu_p + T(\phid) \nabla \mu_d |^2 \, dx$ comes from the fact that the equations have been restated in terms of the pressure $q$. Note also that, here, in view of the fact that
the potential is regularized, the approximate energy $E_\eps$ is given by
\begin{align*}
  E_\eps = \frac12 \big( \| \nabla \phip \|_{L^{2}(\Omega)}^2 + \| \nabla \phid \|_{L^{2}(\Omega)}^2 \big) + \io \big( F_{\eps}(\phip, \phid) + F_1(\phip, \phid) \big) \, dx
\end{align*}
and it may be no longer coercive with respect to $\phip, \phid$. For this reason, we need to
add to \eqref{energyx} the product of \eqref{eq:pa12} by $\phip$ and
the product of \eqref{eq:da12} by $\phid$. By standard computations the sum
of these contributions can be written in the form of the following differential
inequality:
\begin{equation} \label{energyx2}
\begin{aligned}
 & \frac12 \frac{d}{dt} \big( \| \phip \|_{L^{2}(\Omega)}^2 + \| \phid \|_{L^{2}(\Omega)}^2 \big)  \\
 & \quad  \le \sigma \big( \| \nabla \mu_p \|_{L^{2}(\Omega)}^2 + \| \nabla \mu_d \|_{L^{2}(\Omega)}^2 \big) + c_\sigma \big( \| \nabla \phip \|_{L^{2}(\Omega)}^2 + \| \nabla \phid \|_{L^{2}(\Omega)}^2 \big) \\
 & \qquad + C \big ( 1 + \| \nabla \bar{q} \|_{L^{2}(\Omega)}^2 + \| \phip \|_{L^{2}(\Omega)}^2 + \| \phid \|_{L^{2}(\Omega)}^2 \big)
\end{aligned}
\end{equation}
for small constant $\sigma>0$ and correspondingly large constant $c_\sigma$.
Hence, the sum of \eqref{energyx} and large $K > 0$ times \eqref{energyx2} gives
\begin{equation} \label{energyx3}
\begin{aligned}
 & \frac{d}{dt} \Big( E_\eps(\phip, \phid) + \frac{K}{2} \big( \| \phip \|_{L^{2}(\Omega)}^2 + \|  \phid \|_{L^{2}(\Omega)}^2 \big) \Big) \\
 & \qquad + (M_p - K\sigma) \| \nabla \mu_p \|_{L^{2}(\Omega)}^2 + (M_d - K\sigma) \| \nabla \mu_d \|_{L^{2}(\Omega)}^2 \\
 & \qquad  + \delta \| \nabla \de_t \phip \|_{L^{2}(\Omega)}^2
   + \delta \| \nabla \de_t \phid \|_{L^{2}(\Omega)}^2 \\
& \quad \le \int_\Omega ( S_p \mu_p + S_d \mu_d ) \, dx - \big( T(\phip) \nabla \bar{q} \cdot \nabla \mu_p + T(\phid) \nabla \bar{q} \cdot \nabla \mu_d \big) \, dx \\
& \qquad  + K c_\sigma \big( \| \nabla \phip \|_{L^{2}(\Omega)}^2 + \| \nabla \phid \|_{L^{2}(\Omega)}^2 \big) \\
& \qquad + K C \big ( 1 + \| \nabla \bar{q} \|_{L^{2}(\Omega)}^2  + \| \phip \|_{L^{2}(\Omega)}^2 + \| \phid \|_{L^{2}(\Omega)}^2 \big).
\end{aligned}
\end{equation}
Now, we first take $K$ large enough so that the modified energy is coercive, namely
\begin{equation}\label{coerc}
 E_\eps + \frac{K}{2} \big( \| \phip \|_{L^{2}(\Omega)}^2 + \| \phid \|_{L^{2}(\Omega)}^2 \big)
  \ge k \big( \| \phip \|_{H^1(\Omega)}^2 + \| \phid \|_{H^1(\Omega)}^2 \big) - C,
\end{equation}
for some $k >0$ (note that $k$ can be chosen independently of $\eps$; indeed, by \eqref{Fepsilon}, $F_{\eps}$ is non-negative, and the smooth non-convex part $F_{1}$ has at most quadratic growth).
Then, after $K$ is fixed, we also take $\sigma$ so small that $\min \{M_p - K\sigma,M_d - K\sigma\} \ge \kappa$
for some constant $\kappa > 0$.  It remains to control the integral terms on the right-hand side of \eqref{energyx3}.  Observe that
\[
\begin{aligned}
 & \left | \int_\Omega \big( T(\phip) \nabla \bar{q} \cdot \nabla \mu_p + T(\phid) \nabla \bar{q} \cdot \nabla \mu_d \big) \, dx \right | \\
 & \quad\leq \frac{\kappa}{4} \big ( \| \nabla \mu_{p} \|_{L^{2}(\Omega)}^{2} + \| \nabla \mu_{d} \|_{L^{2}(\Omega)}^{2} \big ) + C \| \nabla \bar{q} \|_{L^{2}(\Omega)}^{2},
\end{aligned}
\]
and recalling $(f)\OO$ denotes the mean value of $f$ over $\Omega$, we have
\begin{equation}\label{source:mu}
\begin{aligned}
& \int_{\Omega} \big ( S_p (\mu_p - (\mu_p)\OO) + S_{p} (\mu_p)\OO \big) \, dx \\
& \quad = \int_{\Omega} S_p  (\mu_p - (\mu_p)\OO)  \, dx + (\mu_p)\OO \int_{\Omega} S_p \, dx \\
& \quad \leq \frac{\kappa}{4} \| \nabla \mu_p \|_{L^{2}(\Omega)}^{2} + C \big ( 1 + \| \phip \|_{L^{2}(\Omega)}^{2} + \| \phid \|_{L^{2}(\Omega)}^{2} \big ) \\
& \qquad + C_{\eps} \big ( 1 + \| \phip \|_{L^{2}(\Omega)}^{2} + \| \phid \|_{L^{2}(\Omega)}^{2} \big ),
\end{aligned}
\end{equation}
in view of the Lipschitz regularity of $F_{\eps,p}$ and $F_{1,p}$.  The term $S_{d} \mu_{d}$ is controlled analogously.  Then, collecting the above computations, \eqref{energyx3} becomes
\begin{equation}
\label{energyx4}
\begin{aligned}
 & \frac{d}{dt} \Big( E_\eps(\phip, \phid) + \frac{K}{2} \big( \| \phip \|_{L^{2}(\Omega)}^2 + \|  \phid \|_{L^{2}(\Omega)}^2 \big) \Big) \\
 & \qquad + \frac{\kappa}{2} \big ( \| \nabla \mu_p \|_{L^{2}(\Omega)}^2 + \| \nabla \mu_d \|_{L^{2}(\Omega)}^2  \big )
   + \delta \| \nabla \de_t \phip \|_{L^{2}(\Omega)}^2
   + \delta \| \nabla \de_t \phid \|_{L^{2}(\Omega)}^2 \\
& \quad \leq  C_{\eps} \big( 1 + \| \phip \|_{H^{1}(\Omega)}^2 + \| \phid \|_{H^{1}(\Omega)}^2 \big) + C \| \nabla \bar{q} \|_{L^{2}(\Omega)}^2,
\end{aligned}
\end{equation}
for some positive constants $C$, $C_{\eps}$ that are independent of $\delta$.  By virtue of the coercivity property \eqref{coerc} and  a Gronwall argument then yields
\begin{align}
\label{reg:fp01}
\| \phip \|_{L^{\infty}(0,T;H^1(\Omega))} + \| \phid \|_{L^{\infty}(0,T;H^1(\Omega))} & \le C_{\eps},\\
\label{reg:fp12}
\| \mu_p \|_{L^2(0,T;H^1(\Omega))} + \| \mu_d \|_{L^2(0,T;H^1(\Omega))} & \le C_{\eps}, \\
\label{reg:fp02}
 \| \nabla \de_{t} \phip \|_{L^{2}(0,T;L^{2}(\Omega))} + \| \nabla \de_{t} \phid \|_{L^{2}(0,T;L^{2}(\Omega))} & \le C_{\eps, \delta}.
\end{align}

\paragraph{Second estimate.}
Testing \eqref{eq:pa22} with $-\Delta \de_{t} \phip$ and \eqref{eq:da22} with $-\Delta \de_{t} \phid$ and using that $\Delta \varphi_{i,0,\delta} \in L^{2}(\Omega)$, we obtain
\begin{align*}
\| \Delta \phip \|_{H^{1}(0,T;L^{2}(\Omega))} + \| \Delta \phid \|_{H^{1}(0,T;L^{2}(\Omega))} \leq C_{\eps,\delta},
\end{align*}
and by elliptic regularity we get the following additional estimate:
\begin{align}
\label{reg:fp11}
   \| \phip \|_{H^1(0,T;H^2(\Omega))} + \| \phid \|_{H^1(0,T;H^2(\Omega))} \le C_{\eps, \delta}.
\end{align}

\paragraph{Uniqueness.} Let us denote by $\hat{\varphi}_{p}$, $\hat{\varphi}_{d}$, $\hat{\mu}_{p}$ and $\hat{\mu}_{d}$ as the differences $\varphi_{p,1} - \varphi_{p,2}$, $\varphi_{d,1} - \varphi_{d,2}$, $\mu_{p,1} - \mu_{p,2}$ and $\mu_{d,1} - \mu_{d,2}$, respectively.  Then, upon testing the difference of the equations \eqref{eq:pa12} by $\hat{\mu}_{p}$ and the difference of the equations \eqref{eq:pa22} by $\de_{t} \hat{\varphi}_{p} - \Delta \hat{\varphi}_{p}$ leads to
\begin{equation}\label{Uniq:1}
\begin{aligned}
& \frac{d}{dt} \frac{1}{2} \big ( \| \nabla \hat{\varphi}_{p} \|_{L^{2}(\Omega)}^{2} + \delta \| \Delta \hat{\varphi}_{p} \|_{L^{2}(\Omega)}^{2} \big )  \\
& \qquad + \| \Laplace \hat{\varphi}_{p} \|_{L^{2}(\Omega)}^{2} + \delta \| \nabla \de_{t} \hat{\varphi}_{p} \|_{L^{2}(\Omega)}^{2} + \int_{\Omega} \big ( M_{p} + (T_{p,2})^{2} \big ) | \nabla \hat{\mu}_{p} |^{2} \, dx \\
& \quad = \int_{\Omega} \hat{S}_{p} \hat{\mu}_{p} - \hat{\mu}_{p} \Delta \hat{\varphi}_{p} - \hat{T}_{p} \nabla \bar{q} \cdot \nabla \hat{\mu}_{p} - \big ( \hat{F}_{\eps,p} + \hat{F}_{1,p} \big ) \big ( \de_{t} \hat{\varphi}_{p} - \Delta \hat{\varphi}_{p} \big ) \, dx \\
& \qquad  - \int_{\Omega} \nabla \hat{\mu}_{p} \cdot \left ( \widehat{(T_{p})^{2}} \nabla \mu_{p,1} + \big (\hat{T}_{p} T_{d,1} + T_{p,2} \hat{T}_{d} \big ) \nabla \mu_{d,1} + T_{p,2} T_{d,2} \nabla \hat{\mu}_{d} \right ) \, dx \\
& \quad =: J_{1} + J_{2},
\end{aligned}
\end{equation}
where we used the notation $T_{p,1} = T(\varphi_{p,1})$, $\hat{T}_{p} = T_{p,1}-T_{p,2}$,
$\hat{F}_{\eps,p} = F_{\eps,p}(\varphi_{p,1}, \varphi_{d,1}) - F_{\eps,p}(\varphi_{p,2}, \varphi_{d,2})$, $\hat{S}_{p} = \Sigma_{p}(\bar{n},\varphi_{p,1}, \varphi_{d,1}) - \Sigma_{p}(\bar{n}, \varphi_{p,2}, \varphi_{d,2}) + m_{pp} \hat{\varphi}_{p} + m_{pd} \hat{\varphi}_{d}$, and $\widehat{(T_{p})^{2}} = (T_{p,1})^{2} - (T_{p,2})^{2} = \hat{T}_{p} (T_{p,1} + T_{p,2})$.  Using the Lipschitz continuity of $F_{\eps,p}$, $F_{1,p}$, $T(\cdot)$, $\Sigma_{i}$ and the boundedness of $T(\cdot)$ and $\Sigma_{i}$, we deduce
\begin{align*}
J_{1} & \leq C \big ( \| \hat{\varphi}_{p} \|_{L^{2}(\Omega)} + \| \hat{\varphi}_{d} \|_{L^{2}(\Omega)} + \| \Laplace \hat{\varphi}_{p} \|_{L^{2}(\Omega)} \big ) \big ( \| \hat{\mu}_{p} - (\hat{\mu}_{p})\OO \|_{L^{2}(\Omega)} + |(\hat{\mu}_{p})\OO| \big ) \\
& \quad + C \big ( \| \hat{\varphi}_{p} \|_{L^{2}(\Omega)} + \| \hat{\varphi}_{d} \|_{L^{2}(\Omega)} \big ) \big ( \| \Laplace \hat{\varphi}_{p} \|_{L^{2}(\Omega)} + \| \de_{t} \hat{\varphi}_{p} - ( \de_{t} \hat{\varphi}_{p} )\OO \|_{L^{2}(\Omega)} + | (\de_{t} \hat{\varphi}_{p}) \OO | \big ), \\
& \quad + C \| \hat{\varphi}_{p} \|_{L^{\infty}(\Omega)} \| \nabla \hat{\mu}_{p} \|_{L^{2}(\Omega)} \| \nabla \bar{q} \|_{L^{2}(\Omega)}, \\
J_{2} & \leq C \| \nabla \hat{\mu}_{p} \|_{L^{2}(\Omega)} \big ( \| \hat{\varphi}_{p} \|_{L^{\infty}(\Omega)} + \| \hat{\varphi}_{d} \|_{L^{\infty}(\Omega)} \big ) \big ( \| \nabla \mu_{p,1} \|_{L^{2}(\Omega)} + \| \nabla \mu_{d,1} \|_{L^{2}(\Omega)} \big ) \\
& \quad + \frac{1}{2} \int_{\Omega} (T_{p,2})^{2} | \nabla \hat{\mu}_{p} |^{2} + (T_{d,2})^{2} | \nabla \hat{\mu}_{d} |^{2} \, dx.
\end{align*}
Note that by the Lipschitz property of $\Sigma_{i}$, $F_{\eps,p}$ and $F_{1,p}$,
\begin{align*}
|(\de_{t} \hat{\varphi}_{p})\OO| & = |(\hat{S}_{p})\OO| \leq C \big ( \| \hat{\varphi}_{p} \|_{L^{2}(\Omega)} + \| \hat{\varphi}_{d} \|_{L^{2}(\Omega)} \big ), \\
|( \hat{\mu})\OO| & = |(\hat{F}_{\eps,p} + \hat{F}_{1,p})\OO| \leq  C \big ( \| \hat{\varphi}_{p} \|_{L^{2}(\Omega)} + \| \hat{\varphi}_{d} \|_{L^{2}(\Omega)} \big ),
\end{align*}
and so, upon adding \eqref{Uniq:1} to the corresponding equation for $\hat{\varphi}_{d}$ leads to
\begin{equation}\label{Uniq:2}
\begin{aligned}
&  \frac{d}{dt} \frac{1}{2} \sum_{i=p,d} \big ( \| \nabla \hat{\varphi}_{i} \|_{L^{2}(\Omega)}^{2} + \delta \| \Delta \hat{\varphi}_{i} \|_{L^{2}(\Omega)}^{2} \big )  \\
& \qquad + \frac{1}{2} \sum_{i=p,d} \left ( \| \Laplace \hat{\varphi}_{i} \|_{L^{2}(\Omega)}^{2} + \delta \| \nabla \de_{t} \hat{\varphi}_{i} \|_{L^{2}(\Omega)}^{2} + M_{i} \| \nabla \hat{\mu}_{i} \|_{L^{2}(\Omega)}^{2} \right ) \\
& \quad \leq C \Big [1 + \| \nabla \bar{q} \|_{L^{2}(\Omega)}^{2} + \sum_{ \substack{i=p,d \\ j=1,2 }} \| \nabla \mu_{i,j} \|_{L^{2}(\Omega)}^{2} \Big ] \sum_{i=p,d} \left ( \| \hat{\varphi}_{i} \|_{L^{2}(\Omega)}^{2} +  \| \Laplace \hat{\varphi}_{i} \|_{L^{2}(\Omega)}^{2} \right ),
\end{aligned}
\end{equation}
where in the above we have used the elliptic estimate
\begin{align}\label{Linfty:Ell}
\| f \|_{L^{\infty}(\Omega)} \leq C \| f \|_{H^{2}(\Omega)} \leq C \big ( \| \Laplace f \|_{L^{2}(\Omega)} + \| f \|_{L^{2}(\Omega)} \big ),
\end{align}
for $\hat{\varphi}_{p}$ and $\hat{\varphi}_{d}$ as they satisfy no-flux boundary conditions.

\smallskip

Next, we test the difference of the equations \eqref{eq:pa12} with $\hat{\varphi}_{p}$ which yields
\begin{align*}
\frac{d}{dt} \frac{1}{2} \| \hat{\varphi}_{p} \|_{L^{2}(\Omega)}^{2} & = - \int_{\Omega} \nabla \hat{\varphi}_{p} \cdot \Big ( M_{p} \nabla \hat{\mu}_{p} + \hat{T}_{p} \nabla \bar{q} + \widehat{(T_{p})^{2}} \nabla \mu_{p,1} + (T_{p,2})^{2} \nabla \hat{\mu}_{p} \Big ) \, dx \\
& \quad - \int_{\Omega} \nabla \hat{\varphi}_{p} \cdot \Big ( \big (\hat{T}_{p} T_{d,1} + T_{p,2} \hat{T}_{d} \big ) \nabla \mu_{d,1} + T_{p,2} T_{d,2} \nabla \hat{\mu}_{d} \Big ) - \hat{S}_{p} \hat{\varphi}_{p} \, dx \\
& \leq C \Big( 1 + \| \nabla \bar{q} \|_{L^{2}(\Omega)}^{2} + \sum_{j=1,2} \| \nabla \mu_{p,j} \|_{L^{2}(\Omega)}^{2} \Big ) \Big (\| \nabla \hat{\varphi}_{p} \|_{L^{2}(\Omega)}^{2} + \sum_{i=p,d} \| \hat{\varphi}_{i} \|_{L^{\infty}(\Omega)}^{2}  \Big ) \\
& \quad + \frac{M_{p}}{4} \| \nabla \hat{\mu}_{p} \|_{L^{2}(\Omega)}^{2} + \frac{M_{d}}{4} \| \nabla \hat{\mu}_{d} \|_{L^{2}(\Omega)}^{2},
\end{align*}
and upon adding the analogous estimate obtained from testing \eqref{eq:da12} with $\hat{\varphi}_{d}$ and then adding to \eqref{Uniq:2}, after
applying the elliptic estimate \eqref{Linfty:Ell}, we arrive at the following differential inequality
\begin{align*}
&  \frac{d}{dt} \frac{1}{2} \sum_{i=p,d} \big ( \| \hat{\varphi}_{i} \|_{L^{2}(\Omega)}^{2} + \| \nabla \hat{\varphi}_{i} \|_{L^{2}(\Omega)}^{2} + \delta \| \Delta \hat{\varphi}_{i} \|_{L^{2}(\Omega)}^{2} \big )  \\
& \qquad + \frac{1}{2} \sum_{i=p,d} \left ( \| \Laplace \hat{\varphi}_{i} \|_{L^{2}(\Omega)}^{2} + \delta \| \nabla \de_{t} \hat{\varphi}_{i} \|_{L^{2}(\Omega)}^{2} + \frac{1}{2} M_{i} \| \nabla \hat{\mu}_{i} \|_{L^{2}(\Omega)}^{2} \right ) \\
& \quad \leq C \Big [ 1 + \| \nabla \bar{q} \|_{L^{2}(\Omega)}^{2} + \sum_{\substack{i=p,d \\ j=1,2}} \| \nabla \mu_{i,j} \|_{L^{2}(\Omega)}^{2} \Big] \sum_{i=p,d}
   \left ( \| \hat{\varphi}_{i} \|_{H^{1}(\Omega)}^{2} + \| \Laplace \hat{\varphi}_{i} \|_{L^{2}(\Omega)}^{2} \right ),
\end{align*}
and a Gronwall argument easily entails uniqueness.
\end{proof}


\subsection{Auxiliary pressure and nutrient equations}
We now consider, for $(\phip, \mu_p, \phid, \mu_d)$ obtained from Lemma \ref{lem:CH},  the following system:
\begin{subequations}\label{Aux:q:n}
\begin{alignat}{3}
\label{eq:qa2}
\delta \big( \de_t q + \Delta^2 q ) - \Delta q & = \div \big( T(\phip) \nabla  \mu_{p} + T(\phid) \nabla \mu_{d} \big) + (S_{p} + S_{d})(n, \phip, \phid),\\
\label{eq:na2}
0 & = -\Laplace n + T(\varphi_{p}) n,
\end{alignat}
\end{subequations}
furnished with the initial-boundary conditions resulting from \eqref{add:init}-\eqref{add:dir}.

\begin{lemma}\label{lem:q:n}
Let $(\phip, \mu_p, \phid, \mu_d)$ denote a weak solution obtained from Lemma \ref{lem:CH}.  Then, there exists a \emph{unique} pair $(q,n)$ of solutions to \eqref{Aux:q:n} in the following sense:
\begin{enumerate}
\item[\bf(1)] the functions have the following regularity properties:
\begin{align*}
q & \in L^{2}(0,T;H^{3}(\Omega)) \cap L^{\infty}(0,T;H^{1}_{0}(\Omega)) \cap H^{1}(0,T;H^{-1}(\Omega)), \\
n & \in L^{\infty}(0,T;W^{2,r}(\Omega)) \text{ for any } r < \infty \text{ and } 0 \leq n \leq 1 \text{ a.e. in } Q,
\end{align*}
with
\begin{align*}
q(0) = 0 \text{ in } \Omega, \quad  q = \Delta q = 0,\; n = 1 \text{ on } \Gamma.
\end{align*}
\item[\bf(2)] Equation \eqref{eq:na2} holds a.e. in $Q$ and equation \eqref{eq:qa2} holds for a.e. $t \in (0,T)$ in the following weak sense:
\begin{align*}
0 = \delta \inner{\de_t q}{\zeta} & + \int_{\Omega} \big ( \nabla q - \delta \nabla \Delta q + T(\phip) \nabla \mu_p + T(\phid) \nabla \mu_d \big ) \cdot \nabla \zeta \, dx \\
&  - \int_{\Omega} (S_{p} + S_{d})(n, \phip, \phid) \zeta \, dx,
\end{align*}
for all $\zeta \in H^{1}_{0}(\Omega)$.
\end{enumerate}
\end{lemma}

\begin{proof} We investigate the nutrient and pressure equations separately.

\paragraph{Nutrient equation.}
Since $T(\cdot)$ is bounded and non-negative, we may first consider a parabolic regularization to \eqref{eq:na2}, namely,
we add $\gamma \de_{t}n$ on the right-hand side (for $\gamma \in (0,1)$) and we complement the resulting parabolic equation
(for example) with the initial condition $n_{\gamma}(0) := 1$ (which is consistent with the boundary datum). Then, applying the standard parabolic theory and the weak comparison principle it is easy to show that there exists a unique function
$n_{\gamma} \in L^{\infty}(0,T;H^{1}(\Omega)) \cap L^{2}(0,T;H^{2}(\Omega)) \cap H^{1}(0,T;L^{2}(\Omega))$ with $0 \leq n_{\gamma} \leq 1$
a.e.~in $Q$.  It turns out that $n_{\gamma}$ is uniformly bounded in $L^{2}(0,T;H^{1}(\Omega))$ and in passing to the limit $\gamma \to 0$
we deduce the existence of a weak solution $n \in L^{2}(0,T;H^{1}(\Omega))$ to \eqref{eq:na2} with $0 \leq n \leq 1$ a.e. in $Q$.
Then, as $T(\varphi_{p}) n \in L^{\infty}(0,T;L^{\infty}(\Omega))$, applying elliptic regularity we infer $n \in L^{\infty}(0,T;W^{2,r}(\Omega))$
for any $r < \infty$.

\paragraph{Pressure equation.} Given $n$, $\phip$, $\mu_{p}$, $\phid$ and $\mu_{d}$, we test
\eqref{eq:qa2} with $q - \Delta q$. Using the boundary conditions $q = \Delta q = 0$ on $\Gamma$, we then
obtain
\begin{equation}\label{apri:q}
\begin{aligned}
& \frac{d}{dt} \frac{\delta}{2} \big ( \| q \|_{L^{2}(\Omega)}^{2} + \| \nabla q \|_{L^{2}(\Omega)}^{2} \big ) + (1 + \delta) \| \Delta q \|_{L^{2}(\Omega)}^{2} + \delta \| \nabla \Delta q \|_{L^{2}(\Omega)}^{2} + \| \nabla q \|_{L^{2}(\Omega)}^{2} \\
& \quad = \int_{\Omega} \big ( T_{p} \nabla \mu_p + T_d \nabla \mu_d \big ) \cdot \nabla \big ( \Delta q -  q \big ) + (S_{p} + S_{d}) (q - \Delta q) \, dx \\
& \quad \leq \frac{C}{\delta} \big ( \| \nabla \mu_p \|_{L^{2}(\Omega)}^{2} + \| \nabla \mu_d \|_{L^{2}(\Omega)}^{2} \big ) + \frac{\delta}{2} \big ( \| \nabla \Delta q \|_{L^{2}(\Omega)}^{2} + \| \nabla q \|_{L^{2}(\Omega)}^{2} \big ) \\
& \qquad + \frac{C}{\delta} \big (1 + \| \phip \|_{L^{2}(\Omega)}^{2} + \| \phid \|_{L^{2}(\Omega)}^{2} \big ) + \frac{\delta}{2} \big ( \| q \|_{L^{2}(\Omega)}^{2} + \| \Delta q \|_{L^{2}(\Omega)}^{2} \big ).
\end{aligned}
\end{equation}
Integrating in time, using $q(0) = 0$, and applying first a Gronwall argument and then elliptic regularity leads to
\begin{align}\label{apri:q1}
\| q \|_{L^{\infty}(0,T;H^{1}(\Omega))} + \|q \|_{L^{2}(0,T;H^{3}(\Omega))} \leq C_{\delta}.
\end{align}
Then, by a comparison of terms in \eqref{eq:qa2} we infer that
\begin{align}\label{apri:q2}
\| \de_{t} q \|_{L^{2}(0,T;H^{-1}(\Omega))} \leq C_{\delta}.
\end{align}

\paragraph{Uniqueness.} Let $\hat{n} := n_{1} - n_{2}$ and $\hat{q} := q_{1} - q_{2}$ denote the difference between two solution pairs $(q_{1}, n_{1})$ and $(q_{2}, n_{2})$ corresponding to the same data $(\phip, \mu_{p}, \phid, \mu_{d})$.  Then it is straightforward to see that
\begin{align}\label{Uniq:aux:n:q}
0 = - \Delta \hat{n} + T(\phip) \hat{n}, \quad \delta (\de_t \hat{q} + \Delta^{2} \hat{q}) - \Delta \hat{q} = \hat{\Sigma}_{p} + \hat{\Sigma}_{d},
\end{align}
where for $i = p,d$,
\begin{align*}
\hat{\Sigma}_{i} := \Sigma_{i}(n_{1}, \phip, \phid) - \Sigma_{i}(n_{2}, \phip, \phid).
\end{align*}
By testing the first equation of \eqref{Uniq:aux:n:q} with $\hat{n}$ we easily deduce that $\hat{n} = 0$ by the Poincar\'{e} inequality.  Then, testing the second equation of \eqref{Uniq:aux:n:q} with $\hat{q}$ and noting that $\hat{\Sigma}_{i} = 0$ due to $n_{1} = n_{2}$, the uniqueness of solutions is clear.
\end{proof}


\subsection{Fixed point argument}\label{sec:fp}
We will now apply a fixed point argument locally in time, and consider for some $T_{0} \in (0,T]$ the pair $(\bar{q}, \bar{n}) \in L^{2}(0,T_{0};H^{1}(\Omega)) \times L^{2}(0,T_{0};L^{2}(\Omega))$ with $0 \leq \bar{n} \leq 1$ a.e. in $\Omega \times (0,T_{0})$.  Let us introduce the mapping $\mathcal{T} : (\bar{q}, \bar{n}) \to (q,n)$, where $(q,n)$ is the unique solution pair to \eqref{Aux:q:n} with $(\phip, \mu_p, \phid, \mu_d)$ as the unique solution quadruple to \eqref{Aux:CH}.  To specify the domain of $\mathcal{T}$ we define
\begin{align*}
X := \big \{ (q,n) : \| q \|_{L^{2}(0,T_{0};H^{1}(\Omega))} + \| n \|_{L^{2}(0,T_{0};L^{2}(\Omega))} \leq R, \; 0 \leq n \leq 1 \text{ a.e. in } \Omega \times (0,T_{0}) \big \},
\end{align*}
where $R > 0$ is arbitrary but otherwise fixed.  For example, one can take $R = 1$.
Let us mention that from \eqref{reg:fp01}-\eqref{reg:fp12} and \eqref{apri:q} one obtains the estimate
\begin{align*}
\| q \|_{L^{2}(0,t;H^{1}(\Omega))}^{2} \leq  t \| q \|_{L^{\infty}(0,T;H^{1}(\Omega))}^{2} \leq t C_{\eps,\delta,R}
\end{align*}
for any $t \in (0,T]$.  Similarly, since $0 \leq n \leq 1$ a.e. in $Q$, we get
\begin{align*}
\| n \|_{L^{2}(0,t;L^{2}(\Omega))}^{2} \leq t | \Omega |.
\end{align*}
Consequently, for $T_{0}$ sufficiently small (in a way that possibly depends on $\eps$, $\delta$ and $R$), we have
\begin{align}
\label{reg:fp32}
\| q \|_{L^{2}(0,T_{0};H^{1}(\Omega))} + \| n \|_{L^{2}(0,T_{0};L^{2}(\Omega))} \leq C_{\eps, \delta, R} T_{0}^{\frac{1}{2}} \leq R.
\end{align}
This implies that for such a choice of $T_{0}$, the operator $\mathcal{T}$ maps $X$ (which is a convex closed subset
of the product Banach space $L^{2}(0,T_{0};H^{1}(\Omega)) \times L^{2}(0,T_{0};L^{2}(\Omega))$) into itself.

\paragraph{Continuity.}
We now aim to show that $\mathcal{T}: X \to X$ is continuous with respect to the norm of $L^{2}(0,T_{0};H^{1}(\Omega)) \times L^{2}(0,T_{0};L^{2}(\Omega))$, keeping in mind that thanks to the uniqueness results
for the auxiliary problems \eqref{Aux:CH} and \eqref{Aux:q:n}, $\mathcal{T}$ is a single-valued mapping.  Let $(\bar{q}_{k}, \bar{n}_{k})_{k \in \N} \subset X$ be a sequence that converges strongly
to a limit $(\bar{q}, \bar{n})$ in $X$.  We denote $(q_{k}, n_{k}) = \mathcal{T}(\bar{q}_{k}, \bar{n}_{k})$ and $(q,n) := \mathcal{T}(\bar{q}, \bar{n})$.  Then, it is easy to see that
from Lemma \ref{lem:CH} (more precisely \eqref{reg:fp01}-\eqref{reg:fp11}) there exists a corresponding sequence $(\varphi_{p,k}, \mu_{p,k}, \varphi_{d,k}, \mu_{d,k})_{k \in \N}$ such that
\begin{align*}
\| \varphi_{i,k} \|_{H^{1}(0,T_{0};H^{2}(\Omega))} + \| \mu_{i,k} \|_{L^{2}(0,T_{0};H^{1}(\Omega))} \leq C_{\eps, \delta, R}
\end{align*}
for $i = p,d$ and some constant $C=C_{\eps,\delta,R}$ independent of $k$.  Then, standard compactness results \cite[\S 8, Cor. 4]{Simon} yield
\begin{align*}
\varphi_{i,k} & \to \varphi_{i} \text{ strongly in } C^{0}([0,T_{0}];W^{1,r}(\Omega)) \cap C^{0}(\overline{\Omega} \times [0,T_{0}]), \\
\mu_{i,k} & \to \mu_{i} \text{ weakly in } L^{2}(0,T_{0};H^{1}(\Omega)),
\end{align*}
along a non-relabelled subsequence for $i = p,d$, and any $r \in [1,\infty)$ in two dimensions and $r \in [1,6)$ in three dimensions.
Hence, along a non-relabelled subsequence, $\varphi_{p,k} \to \varphi_{p}$ uniformly in $\overline{\Omega} \times [0,T_{0}]$ and thus
$T(\varphi_{p,k}) \to T(\varphi_{p})$ uniformly in $\overline{\Omega} \times [0,T_{0}]$.  Moreover,
one can easily check that the limit functions $\varphi_i$, $\mu_i$ solve \eqref{Aux:CH} with
$q,n$ in place of $\bar{q},\bar{n}$. Next, taking the difference of \eqref{eq:na2} for two indices $a$ and $b$ leads to
\begin{align*}
- \Delta ( n_{a} - n_{b}) + ( T(\varphi_{p,a}) - T(\varphi_{p,b}) ) n_{a} + T(\varphi_{p,b}) (n_{a} - n_{b}) = 0,
\end{align*}
and by testing with $n_{a} - n_{b}$ we obtain by the Poincar\'{e} inequality
\begin{equation}\label{Cauchy}
\begin{aligned}
& \| \nabla (n_{a} - n_{b})\|_{L^{2}(0,T_{0};L^{2}(\Omega))}^{2} \\
& \quad \leq \| n_{a} - n_{b} \|_{L^{2}(0,T_{0};L^{2}(\Omega))} \| T(\varphi_{p,a}) - T(\varphi_{p,b})\|_{L^{2}(0,T_{0};L^{2}(\Omega))} \\
& \quad \leq C \| \nabla (n_{a} - n_{b})\|_{L^{2}(0,T_{0};L^{2}(\Omega))} \| T(\varphi_{p,a}) - T(\varphi_{p,b})\|_{L^{2}(0,T_{0};L^{2}(\Omega))}
\end{aligned}
\end{equation}
after neglecting the non-negative term $T(\varphi_{p,b}) | n_{a} - n_{b}|^{2}$.  Applying the uniform convergence of $T(\varphi_{p,k})$ we see that $\{ n_{k} \}_{k \in \N}$ is a Cauchy sequence in $L^{2}(0,T_{0};H^{1}(\Omega))$ and thus $n_{k} \to n_{*}$ strongly in $L^{2}(0,T_{0};L^{2}(\Omega))$ for some limit function $n_{*}$.  Meanwhile,  from the a priori estimates \eqref{apri:q1}-\eqref{apri:q2} and standard compactness results, along a non-relabelled subsequence it holds that
\begin{align*}
q_{k} \to q_{*} \text{ strongly in } L^{2}(0,T;H^{1}(\Omega)).
\end{align*}
Let us mention here that thanks to the strong convergence of $n_{k} \to n_{*}$ in $L^{2}(0,T_{0};L^{2}(\Omega))$,
along a further subsequence we have a.e.~convergence in $\Omega \times (0,T_{0})$.  Continuity of $\Sigma_{i}$, $i = p,d$, and boundedness are sufficient to ensure that the source terms $\Sigma_{i}(n_{k}, \varphi_{p,k}, \varphi_{d,k})$, $i =p,d$, converge to $\Sigma_{i}(n, \phip, \phid)$ strongly in $L^{2}(0,T_{0};L^{2}(\Omega))$.

Hence, along a non-relabelled subsequence $\mathcal{T}(\bar{q}_{k}, \bar{n}_{k}) \to (q_{*}, n_{*})$.  On the other hand,
it is easy to check that $(q_{*}, n_{*})$ solve \eqref{Aux:q:n} (with the limit $\varphi_i,\mu_i$). Then,
thanks to the uniqueness of the solutions for the auxiliary equations \eqref{Aux:q:n}, one infers that, necessarily,
$(q_{*}, n_{*}) = (q,n) = \mathcal{T}(\bar{q}, \bar{n})$ and the whole sequence converges.  This shows the required continuity of
the map~$\mathcal{T}$.

\paragraph{Compactness.} To apply Schauder's fixed point theorem to $\mathcal{T}$, it remains to show that $\mathcal{T} : X \to X$ is a compact mapping.
This amounts to prove for any sequence $(\bar{q}_{k}, \bar{n}_{k})_{k \in \N} \subset X$, there exists a subsequence $(\bar{q}_{k_{l}}, \bar{n}_{k_{l}})_{l \in \N}$ such that $(q_{k_{l}}, n_{k_{l}}) := \mathcal{T}(\bar{q}_{k_{l}}, \bar{n}_{k_{l}})$ converges strongly to some limit $(q,n)$ in $L^{2}(0,T_{0};H^{1}(\Omega)) \times L^{2}(0,T_{0};L^{2}(\Omega))$.  Note that by the definition of $X$ we have
\begin{align*}
\| q_{k} \|_{L^{2}(0,T_{0};H^{1}(\Omega))} + \| n_{k} \|_{L^{2}(0,T_{0};L^{2}(\Omega))} \leq R
\end{align*}
and $0 \leq n_{k} \leq 1$ a.e. in $\Omega \times (0,T_{0})$. This boundedness and
a similar argument to the proof of the continuity of $\mathcal{T}$ permit us to conclude. Indeed, by repeating the a priori estimates given above, one can easily prove that the sequence
$(q_k,n_k)$ is uniformly bounded in a better space, whence follows the desired compactness assertion.

\smallskip

We now state the main result of this section.
\begin{thm}[Local existence]\label{thm:reg:local}
Let Assumption~\ref{ass:pd} hold. Moreover,
for each $\eps \in (0,1)$, $\delta \in (0,1)$
let us assume that $F_{\eps} : \RR^{2} \to [0,+\infty)$ and $F_{1}: \RR^{2} \to \RR$ are given such
that $\nabla F_{\eps}$, $\nabla F_{1}$ are globally Lipschitz continuous. Then, there exist a time $T_{0} \in (0,T]$
and functions $(\phip, \mu_p, \phid, \mu_d, q, n)$ such that
\begin{enumerate}
\item[\bf(1)] the following regularity properties
\begin{align*}
\varphi_{i} & \in H^{1}(0,T_{0};H^{2}(\Omega)) \quad \text{ for } i = p,d, \\
\mu_{i} & \in L^{2}(0,T_{0};H^{1}(\Omega)) \quad \text{ for } i = p,d,  \\
q & \in L^{2}(0,T_{0};H^{3}(\Omega)) \cap L^{\infty}(0,T_{0};H^{1}_{0}(\Omega)) \cap H^{1}(0,T_{0};H^{-1}(\Omega)), \\
n & \in L^{\infty}(0,T_{0};W^{2,r}(\Omega)) \text{ for any } r <\infty \text{ and } 0 \leq n \leq 1 \text{ a.e. in } \Omega \times (0,T_{0}),
\end{align*}
hold together with
\begin{align*}
\varphi_{i}(0) = \varphi_{0,i,\delta},~i=p,d, \; q(0) = 0 \text{ in } \Omega, \quad \Laplace q = 0, \; n = 1 \text{ on } \de \Omega \times (0,T_{0}).
\end{align*}
\item[\bf(2)] Equations \eqref{eq:pa2}, \eqref{eq:da2}, \eqref{eq:Spa}, \eqref{eq:Sda} and \eqref{eq:na} hold a.e.~in $\Omega \times (0,T_{0})$, and equations \eqref{eq:pa1}, \eqref{eq:da1} and \eqref{eq:qa} hold for a.e. $t \in (0,T_{0})$ in the following weak sense:
\begin{align*}
0 & = \int_{\Omega} (\de_{t} \varphi_{p} - S_{p})\zeta + \big ( M_{p} \nabla \mu_{p} + T(\varphi_{p}) \big ( \nabla q + T(\varphi_{p}) \nabla \mu_{p} +  T(\varphi_{d}) \nabla \mu_{d} \big ) \big ) \cdot \nabla \zeta \, dx, \\
0 & = \int_{\Omega} (\de_{t} \varphi_{d} - S_{d})\zeta + \big ( M_{d} \nabla \mu_{d} + T(\varphi_{d}) \big ( \nabla q + T(\varphi_{d}) \nabla \mu_{d} + T(\varphi_{p}) \nabla \mu_{p} \big ) \big ) \cdot \nabla \zeta \, dx, \\
0 & = \delta \inner{\de_t q}{\xi} + \int_{\Omega} \big ( \nabla q - \delta \nabla \Delta q + T(\phip) \nabla \mu_{p} + T(\phid) \nabla \mu_{d} \big ) \cdot \nabla \xi -(S_{p} + S_{d}) \xi \, dx
\end{align*}
for all $\zeta \in H^{1}(\Omega)$ and $\xi \in H^{1}_{0}(\Omega)$.
\end{enumerate}
\end{thm}


\subsection{A priori estimates}
We now derive some a priori estimates for the solution $(\phip, \mu_{p}, \phid, \mu_{d}, q, n)$ to \eqref{delta:eps:Reg}
obtained from Theorem~\ref{thm:reg:local}. All these estimates will be independent of the final time $T_0$, which will
allow us to extend the solution up to the full time interval $[0,T]$. For this reason, although with some abuse of notation,
we shall directly work on the original time interval $[0,T]$ and postpone the details of the extension argument to the next
subsection.
Below the symbol $C$ denotes constants that are independent of $\delta$ and $\eps$.

\paragraph{First estimate.} From the nutrient equation \eqref{eq:na}, we obtain from the boundedness of the cut-off operator $T$ and of $n$ the estimate
\begin{align*}
\| \nabla n \|_{L^{2}(\Omega)}^{2} + \int_{\Omega} T(\phip) |n-1|^{2} \, dx = -\int_{\Omega} T(\phip) (n-1) \, dx \leq C.
\end{align*}
Hence, integrating in time yields
\begin{align*}
\| n \|_{L^{2}(0,T;H^{1}(\Omega))} \leq C.
\end{align*}
The weak comparison principle then yields that $0 \leq n \leq 1$ a.e.~in $\Omega \times (0,T)$.
Hence, by elliptic regularity, we arrive at
\begin{align}\label{Unif:delta:1}
\| n \|_{L^{\infty}(0,T;W^{2,r}(\Omega))} \leq C \quad \forall r < \infty.
\end{align}

\paragraph{Second estimate.}
Testing \eqref{eq:pa1} with $\mu_{p}$, \eqref{eq:pa2} with $\de_{t} \phip$, \eqref{eq:da1} with $\mu_{d}$ and \eqref{eq:da2} with
$\de_{t} \phid$, and summing leads to an analogous identity to \eqref{energyx} but with $\bar{q}$ replaced by $q$.  Then, adding the resulting identity to that obtained from testing \eqref{eq:qa} with $q$ leads to the equality
\begin{equation}\label{Apri:1}
\begin{aligned}
& \frac{d}{dt} \big ( E_{\eps}(\phip, \phid) + \frac{\delta}{2} \| q \|_{L^{2}(\Omega)}^{2} \big )
+ \sum_{i=p,d} \big ( M_{i} \| \nabla \mu_{i} \|_{L^{2}(\Omega)}^{2} + \delta \| \nabla \de_{t} \varphi_{i} \|_{L^{2}(\Omega)}^{2}  \big ) \\
& \qquad + \delta \| \Delta q \|_{L^{2}(\Omega)}^{2} + \| \nabla q + T(\phip) \nabla \mu_{p} + T(\phid) \nabla \mu_{d} \|_{L^{2}(\Omega)}^{2} \\
& \quad = \int_{\Omega} (S_{p} \mu_{p} + S_{d} \mu_{d}) + (S_{p} + S_{d}) q \, dx.
\end{aligned}
\end{equation}
Testing now \eqref{eq:pa1} with $\phip$, \eqref{eq:da1} with $\phid$ and summing the obtained relations yields
\begin{equation}\label{Apri:2}
\begin{aligned}
& \frac{1}{2} \frac{d}{dt} \big ( \| \phip \|_{L^{2}(\Omega)}^{2} + \| \phid \|_{L^{2}(\Omega)}^{2} \big ) \\
& \quad = - \sum_{i = p,d} \int_{\Omega}  M_{i} \nabla \mu_{i} \cdot \nabla \varphi_{i} - S_{i} \varphi_{i} \, dx \\
& \qquad - \int_{\Omega} \big ( \nabla q + T(\phip) \nabla \mu_{p} + T(\phid) \nabla \mu_{d} \big ) \cdot \big ( T(\phip) \nabla \phip + T(\phid) \nabla \phid \big ) \, dx.
\end{aligned}
\end{equation}
Summing \eqref{Apri:1} and \eqref{Apri:2} then gives
\begin{equation}\label{Apri:3}
\begin{aligned}
& \frac{d}{dt} \int_{\Omega} \left ( F_{\eps}(\phip, \phid) + F_{1}(\phip, \phid)
  + \sum_{i=p,d} \frac{1}{2} \big ( | \nabla \varphi_{i} |^{2} + | \varphi_{i} |^{2} \big ) + \frac{\delta}{2} | q|^{2} \right ) \, dx \\
& \qquad + \delta \| \Delta q \|_{L^{2}(\Omega)}^{2} + \sum_{i=p,d} \Big ( \frac{1}{2} M_{i} \| \nabla \mu_{i} \|_{L^{2}(\Omega)}^{2}
  + \delta \| \nabla \de_{t} \varphi_{i} \|_{L^{2}(\Omega)}^{2} \Big )\\
& \qquad + \frac{1}{2} \| \nabla q + T(\phip) \nabla \mu_{p} + T(\phid) \nabla \mu_{d}\|_{L^{2}(\Omega)}^{2} \\
& \quad \leq C + C \sum_{i=p,d} \big ( \| \varphi_{i} \|_{L^{2}(\Omega)}^{2} + \| \nabla \varphi_{i} \|_{L^{2}(\Omega)}^{2} \big ) + \int_{\Omega} S_{p} \mu_{p} + S_{d} \mu_{d} + (S_{p} + S_{d}) q \, dx.
\end{aligned}
\end{equation}
It remains to control the integral on the right-hand side of \eqref{Apri:3}.  To handle the pressure term we consider,
for a.e.~$t \in (0,T)$, the function $N_{q}(t) \in H^{2}(\Omega) \cap H^{1}_{0}(\Omega)$ as the unique solution to the Poisson problem
\begin{align*}
-\Delta N_{q}(t) = q(t) \text{ in } \Omega, \quad N_{q}(t) = 0 \text{ on } \Gamma.
\end{align*}
As $q(t) \in L^{2}(\Omega)$, elliptic regularity shows that $\| N_{q} \|_{H^{2}(\Omega)} \leq C_{*} \| q \|_{L^{2}(\Omega)}$.
Furthermore, it can be shown that (see for example \cite[\S 2.2]{GLDirichlet})
\begin{align*}
\inner{\de_{t} q}{N_{q}} = \frac{1}{2} \frac{d}{dt} \| \nabla N_{q} \|_{L^{2}(\Omega)}^{2},
\end{align*}
where $\inner{\cdot}{\cdot}$ is the duality pairing between $H^{1}_{0}(\Omega)$ and $H^{-1}(\Omega)$.
We additionally claim that $N_{q}(0) = 0$.  Indeed, as $q(0) = 0$ from \eqref{add:init}, the only solution to the Laplace equation with zero Dirichlet condition is zero.  Then, upon testing \eqref{eq:qa} with $N_{q}$ leads to
\begin{align*}
& \frac{\delta}{2} \frac{d}{dt} \| \nabla N_{q} \|_{L^{2}(\Omega)}^{2} + \|q\|_{L^{2}(\Omega)}^{2} + \delta \| \nabla q \|_{L^{2}(\Omega)}^{2} \\
& \quad  = \int_{\Omega} (S_{p} + S_{d}) N_{q} - (T(\phip) \nabla \mu_{p} + T(\phid) \nabla \mu_{d}) \cdot \nabla N_{q}  \, dx \\
& \quad \leq C \big ( 1 + \sum_{i=p,d} \| \varphi_{i} \|_{L^{2}(\Omega)} \big ) \| N_{q} \|_{L^{2}(\Omega)} + \sum_{i=p,d} \| \nabla \mu_{i} \|_{L^{2}(\Omega)} \| \nabla N_{q} \|_{L^{2}(\Omega)},
\end{align*}
where we have also used that $(- \delta \Delta q, \Delta N_{q}) = \delta ( \Delta q, q) = - \delta \| \nabla q \|_{L^{2}(\Omega)}^{2}$.  Therefore, by Young's and Poincar\'e's
inequalities, as well as the estimate $\norm{N_{q}}_{H^{2}(\Omega)} \leq C_{*} \norm{q}_{L^{2}(\Omega)}$, we arrive at the following estimate
\begin{equation}\label{Apri:pressure}
\begin{aligned}
& \frac{\delta}{2} \frac{d}{dt} \| \nabla N_{q} \|_{L^{2}(\Omega)}^{2} + \frac{1}{2} \| q \|_{L^{2}(\Omega)}^{2} + \delta \| \nabla q \|_{L^{2}(\Omega)}^{2} \\
& \quad \leq C \big (1 + \sum_{i=p,d} \big ( \|  \nabla \mu_{i} \|_{L^{2}(\Omega)}^{2} + \| \varphi_{i} \|_{L^{2}(\Omega)}^{2} \big ) \big ).
\end{aligned}
\end{equation}
By virtue of the computations performed in \eqref{source:mu} we infer that
\begin{equation}\label{Apri:S:mu}
\begin{aligned}
& \int_{\Omega} S_{p} \mu_{p} + S_{d} \mu_{d} \, dx \\
& \quad \leq \frac{M_{p}}{4} \| \nabla \mu_{p} \|_{L^{2}}^{2} + \frac{M_{d}}{4} \| \nabla \mu_{d} \|_{L^{2}(\Omega)}^{2}
  + C_{\eps} \big ( 1 + \| \phip \|_{L^{2}(\Omega)}^{2} + \| \phid \|_{L^{2}(\Omega)}^{2} \big ).
\end{aligned}
\end{equation}
Then, let $\kappa$ be a sufficiently small constant such that $\kappa C \leq \frac{1}{4} \min (M_{p}, M_{d})$, where $C$ is the constant on the right-hand side of \eqref{Apri:pressure}, and adding $\kappa$ times \eqref{Apri:pressure} to \eqref{Apri:3} yields
\begin{equation}\label{Apri:4}
\begin{aligned}
& \frac{d}{dt} \left  ( \int_{\Omega} (F_{\eps} + F_{1})(\phip, \phid) \, dx + \sum_{i=p,d} \frac{1}{2} \| \varphi_{i} \|_{H^{1}(\Omega)}^{2} + \frac{\delta}{2} \big ( \| q \|_{L^{2}(\Omega)}^{2} + \kappa \| \nabla N_{q} \|_{L^{2}(\Omega)}^{2} \big )  \right ) \\
& \qquad + \sum_{i=p,d} \big ( \tfrac{1}{4} M_{i} \| \nabla \mu_{i} \|_{L^{2}(\Omega)}^{2} + \delta \| \nabla \de_{t} \varphi_{i} \|_{L^{2}(\Omega)}^{2} \big ) + \frac{\kappa}{4}  \| q \|_{L^{2}(\Omega)}^{2} + \delta \kappa \| \nabla q \|_{L^{2}(\Omega)}^{2}  \\
& \qquad + \delta \| \Delta q \|_{L^{2}(\Omega)}^{2} + \frac{1}{2} \| \nabla q + T(\phip) \nabla \mu_{p} + T(\phid) \nabla \mu_{d} \|_{L^{2}(\Omega)}^{2} \\
& \quad \leq C_{\eps} + C_{\eps} \sum_{i=p,d} \| \varphi_{i} \|_{H^{1}(\Omega)}^{2},
\end{aligned}
\end{equation}
where we have estimated the last term on the right-hand side of \eqref{Apri:3} as follows:
\begin{align*}
\int_{\Omega} (S_{p} + S_{d}) q \, dx \leq \frac{\kappa}{4} \| q \|_{L^{2}(\Omega)}^{2} + C \big ( 1 + \| \phip \|_{L^{2}(\Omega)}^{2} + \| \phid \|_{L^{2}(\Omega)}^{2} \big ).
\end{align*}
Then, applying Gronwall's inequality to \eqref{Apri:4} yields the following estimates uniform in $\delta$:
\begin{equation}\label{Unif:delta:2}
\begin{aligned}
\| (F_{\eps} + F_{1})(\phip, \phid)\|_{L^{\infty}(0,T;L^{1}(\Omega))} + \| \phip \|_{L^{\infty}(0,T;H^{1}(\Omega))} + \| \phid \|_{L^{\infty}(0,T;H^{1}(\Omega))} & \leq C_{\eps}, \\
\| \nabla \mu_{p} \|_{L^{2}(0,T;L^{2}(\Omega))}  + \| \nabla \mu_{d} \|_{L^{2}(0,T;L^{2}(\Omega))} + \| q \|_{L^{2}(0,T;L^{2}(\Omega))} & \leq C_{\eps}, \\
\sqrt{\delta} \big ( \| q \|_{L^{\infty}(0,T; L^{2}(\Omega))} + \| \nabla \de_{t} \phip \|_{L^{2}(0,T;L^{2}(\Omega))} + \| \nabla \de_{t} \phid \|_{L^{2}(0,T;L^{2}(\Omega))} \big )& \leq C_{\eps},
\end{aligned}
\end{equation}
also thanks to the fact that $q(0) = \Delta q(0) = N_{q}(0) = 0$ and that
\begin{align*}
\| \varphi_{i,0,\delta} \|_{H^{1}(\Omega)} \leq C \| \varphi_{i,0} \|_{H^{1}(\Omega)}
\end{align*}
from \eqref{ic:delta:Unif:H1}. Then, testing \eqref{eq:qa} wth $q$ and estimating the right-hand side gives
\begin{align*}
& \frac{\delta}{2} \frac{d}{dt} \norm{q}_{L^{2}(\Omega)}^{2} + \norm{\nabla q}_{L^{2}(\Omega)}^{2} + \delta \norm{\nabla \Laplace q}_{L^{2}(\Omega)}^{2} \\
& \quad \leq C \big ( 1 + \norm{q}_{L^{2}(\Omega)}^{2} + \sum_{i=p,d} \norm{\varphi_{i}}_{L^{2}(\Omega)}^{2} + \norm{\nabla \mu_{i}}_{L^{2}(\Omega)}^{2} \big ) + \frac{1}{2} \norm{\nabla q}_{L^{2}(\Omega)}^{2}.
\end{align*}
In light of \eqref{Unif:delta:2}, and recalling the initial condition $q(0) = 0$, we find that
\begin{align}
\label{Unif:delta:2a}
\norm{\nabla q}_{L^{2}(0,T;L^{2}(\Omega))} \leq C_{\eps}.
\end{align}

\paragraph{Third estimate.}
Thanks to the Lipschitz regularity of $F_{\eps,i}$ and $F_{1,i}$ for $i = p,d$, it is easy to see that by \eqref{Unif:delta:2}
\begin{align*}
| (\mu_{i})\OO |^{2} \leq C_{\eps} \big ( 1 + \| \phip \|_{L^{2}(\Omega)}^{2} + \| \phid \|_{L^{2}(\Omega)}^{2} \big ) \in L^{\infty}(0,T).
\end{align*}
Hence, by Poincar\'{e}'s inequality and \eqref{Unif:delta:2}, we deduce
\begin{align}\label{Unif:delta:3}
\| \mu_{p} \|_{L^{2}(0,T;H^{1}(\Omega))} + \| \mu_{d} \|_{L^{2}(0,T;H^{1}(\Omega))} \leq C_{\eps}.
\end{align}

\paragraph{Fourth estimate.}
Testing \eqref{eq:pa2} with $\Delta \phip$, and in light of \eqref{Unif:delta:3} and the Lipschitz regularity of $F_{\eps,p}$ and $F_{1,p}$, we have
\begin{align}\label{Apri:5}
\frac{1}{2} \| \Delta \phip \|_{L^{2}(\Omega)}^{2} + \frac{d}{dt} \frac{\delta}{2} \| \Delta \phip \|_{L^{2}(\Omega)}^{2} \leq C_{\eps} \big ( 1 + \| \mu_{p} \|_{L^{2}(\Omega)}^{2} + \| \phip \|_{L^{2}(\Omega)}^{2} + \| \phid \|_{L^{2}(\Omega)}^{2} \big ).
\end{align}
Recalling \eqref{initialcond:delta} we see that
\begin{align}\label{Lap:delta}
\delta \| \Delta \varphi_{p,0, \delta} \|_{L^{2}(\Omega)}^{2} \leq C \delta (1 + \delta^{-1}) \| \varphi_{i,0} \|_{H^{1}(\Omega)}^{2} \leq C.
\end{align}
Thus, integrating \eqref{Apri:5} in time and applying the elliptic estimate
$$
\norm{u}_{H^{2}(\Omega)} \leq C \big ( \norm{\Delta u}_{L^{2}(\Omega)} + \norm{u}_{L^{2}(\Omega)} \big )
$$
(holding when $u$ satisfies, as here, no-flux conditions),
we obtain
\begin{equation}\label{Unif:delta:4}
\begin{aligned}
\| \phip \|_{L^{2}(0,T;H^{2}(\Omega))} + \| \phid \|_{L^{2}(0,T;H^{2}(\Omega))} & \leq C_{\eps}, \\
\sqrt{\delta} \big ( \| \phip \|_{L^{\infty}(0,T;H^{2}(\Omega))} + \| \phid \|_{L^{\infty}(0,T;H^{2}(\Omega))} \big ) & \leq C_{\eps}.
\end{aligned}
\end{equation}
Then, by inspection of \eqref{eq:pa1} we find that
\begin{align*}
\| \de_{t} \phip \|_{H^{1}(\Omega)'} \leq C \big ( \| \nabla q \|_{L^{2}(\Omega)} + \| \nabla \mu_{p} \|_{L^{2}(\Omega)} + \| \nabla \mu_{d} \|_{L^{2}(\Omega)} + \| S_{p} \|_{L^{2}(\Omega)} \big ),
\end{align*}
with a similar relation holding for $\phid$. Hence, we infer that
\begin{align}\label{Unif:delta:5a}
\| \de_{t} \phip \|_{L^{2}(0,T;H^{1}(\Omega)')} + \| \de_{t} \phid \|_{L^{2}(0,T;H^{1}(\Omega)')} \leq C_{\eps}.
\end{align}

\paragraph{Fifth estimate.}
Testing \eqref{eq:qa} with $q -\Delta q \in H^{1}_{0}(\Omega)$ and performing standard computations leads to the analogue of \eqref{apri:q}.
Then, multiplying both sides of \eqref{apri:q} by $\delta$ and using a Gronwall argument yields
\begin{equation}\label{Unif:delta:5b}
\begin{aligned}
\delta \| q \|_{L^{\infty}(0,T;H^{1}(\Omega))} + \delta \| \nabla \Delta q \|_{L^{2}(0,T;L^{2}(\Omega))} & \leq C_{\eps}, \\
\sqrt{\delta} \big ( \| \nabla q \|_{L^{2}(0,T;L^{2}(\Omega))} + \| \Delta q \|_{L^{2}(0,T;L^{2}(\Omega))} \big ) & \leq C_{\eps}.
\end{aligned}
\end{equation}
Then, by inspection of \eqref{eq:qa}, and recalling \eqref{Unif:delta:2}, \eqref{Unif:delta:2a} and \eqref{Unif:delta:5b}, we infer
\begin{equation}
\label{Unif:delta:6}
\begin{aligned}
\delta \| \de_{t} q \|_{L^{2}(0,T;H^{-1}(\Omega))} & \leq C \sum_{i=p,d} \big ( 1 + \| \nabla \mu_{i} \|_{L^{2}(0,T;L^{2}(\Omega))} + \| \varphi_{i} \|_{L^{2}(0,T;L^{2}(\Omega))} \big ) \\
& \quad +  C \| \nabla q \|_{L^{2}(0,T;L^{2}(\Omega))} + C \delta \| \nabla \Delta q \|_{L^{2}(0,T;L^{2}(\Omega))} + C \\
& \leq C_{\eps}.
\end{aligned}
\end{equation}
Meanwhile, testing \eqref{eq:pa2} with $-\delta \Delta \de_{t} \phip$ and \eqref{eq:da2} with $- \delta \Delta \de_{t} \phid$ we obtain using \eqref{Lap:delta}
\begin{align}
\label{Unif:delta:7a}
\delta \big ( \| \Delta \de_{t} \phip \|_{L^{2}(0,T;L^{2}(\Omega))} + \| \Delta \de_{t} \phid \|_{L^{2}(0,T;L^{2}(\Omega))} \big ) & \leq C_{\eps}.
\end{align}
On the other hand, testing \eqref{eq:pa1} with $\sqrt{\delta} \de_{t} \varphi_{p}$, we obtain
\begin{align*}
\sqrt{\delta} \norm{\de_{t} \varphi_{p}}_{L^{2}(\Omega)}^{2} & \leq C \big (1 +  \sum_{i=,p,d}  \big ( \norm{\nabla \mu_{i}}_{L^{2}(\Omega)}^{2} + \norm{\varphi_{i}}_{L^{2}(\Omega)}^{2} \big )
+ \norm{\nabla q}_{L^{2}(\Omega)}^{2} + \delta \norm{\nabla \de_{t} \varphi_{p}}_{L^{2}(\Omega)}^{2} \big ) \\
& \quad  + \frac{1}{2} \sqrt{\delta} \norm{\de_{t} \varphi_{p}}_{L^{2}(\Omega)}^{2}.
\end{align*}
Recalling \eqref{Unif:delta:2} and \eqref{Unif:delta:2a}, we then deduce that
\begin{align*}
\sqrt{\delta} \norm{\de_{t} \varphi_{p}}_{L^{2}(0,T;L^{2}(\Omega))}^{2} + \sqrt{\delta} \norm{\de_{t} \varphi_{d}}_{L^{2}(0,T;L^{2}(\Omega))}^{2} \leq C_{\eps},
\end{align*}
whence by repeating the same argument on $\varphi_d$ and by applying elliptic regularity, \eqref{Unif:delta:7a} yields
\begin{align}\label{Unif:delta:7}
 \delta \| \phip \|_{H^1(0,T;H^2(\Omega))} + \delta \| \phid \|_{H^1(0,T;H^2(\Omega))} \le C_{\eps}.
\end{align}


\subsection{Extension to $[0,T]$}
Thanks to the a priori estimates \eqref{Unif:delta:1}, \eqref{Unif:delta:2}, \eqref{Unif:delta:3},
\eqref{Unif:delta:4}, \eqref{Unif:delta:5a}, \eqref{Unif:delta:5b},
\eqref{Unif:delta:6}, \eqref{Unif:delta:7a} and \eqref{Unif:delta:7}, which have a uniform character with respect to the time
variable, we can extend the local solution obtained from Theorem \ref{thm:reg:local} up to the full reference interval $[0,T]$.
This can be achieved by means of a standard contradiction argument which we now outline.
Suppose there exists a maximal time of existence $T_{m} \in (0,T]$
for the weak solution $(\phip, \mu_{p}, \phid, \mu_{d}, q, n)$ to \eqref{delta:eps:Reg}.
To be precise, $T_m$ is defined as the largest time such that $(\phip, \mu_{p}, \phid, \mu_{d}, q, n)$ exists with the regularity properties
specified in the statement of Theorem \ref{thm:reg:local}.
We want to prove that, in fact, $T_m=T$. If this is not the case,
repeating the a priori estimates mentioned above (but now working
on the maximal time interval $[0,T_m]$), we deduce in particular that
\begin{align*}
  \| \phip \|_{C^{0}([0,T_m];H^{2}(\Omega))} + \| \phid \|_{C^{0}([0,T_m];H^{2}(\Omega))}
   + \| q \|_{C^{0}([0,T_m];H^{1}_{0}(\Omega))}
  \leq C_{\eps,\delta},
\end{align*}
where $C_{\eps,\delta}$ is independent of $T_m$. Note that, to obtain the above bound, we used in particular
\eqref{Unif:delta:7a} with the continuous embedding $H^{1}(0,T_m) \subset C^{0}([0,T_m])$
and \eqref{Unif:delta:5b}-\eqref{Unif:delta:6} with the continuous embedding
\begin{align*}
 L^{2}(0,T_{m};H^{3}(\Omega) \cap H^{1}_0(\Omega)) \cap H^{1}(0,T_{m};H^{-1}(\Omega)) \subset C^0([0,T_m];H^{1}_{0}(\Omega)).
\end{align*}
In particular, the triple $( \phip(t),\phid(t),q(t) )$ remains bounded
in $H^2(\Omega)\times H^2(\Omega)\times H^1_0(\Omega)$, and actually (strongly) converges
in the same space to a limit $( \phip(T_m),\phid(T_m),q(T_m) )$, as $t\nearrow T_m$.
This allows us to restart
the system by taking $\phip(T_{m})$, $\phid(T_{m})$ and $q(T_m)$ as new ``initial'' data
(note that the other equations of the system have a quasi-static nature; hence they do
not involve any initial data). To be precise, we should observe that we performed
the fixed point argument by assuming the initial condition $q(0)=0$, while we are restarting
the argument from $q(T_m)\not=0$. On the other hand, it is easy to realize that the
choice $q(0)=0$ was taken just for convenience (indeed, that initial datum will disappear
when taking the limit $\delta\to 0$) and the argument still works for any datum
in $H^1_0$ (as is $q(T_m)$). Hence, restarting from $T_m$ we get a new local solution which
is defined on an interval of the form $(T_m,T_m+\epsilon)$ for some $\epsilon>0$
and still enjoys the regularity properties detailed in Theorem~\ref{thm:reg:local}.
This contradicts the maximality of $T_m$. Hence $T_m=T$.


\subsection{Passing to the limit $\delta \to 0$}
We now pass to the limit $\delta \to 0$ to obtain a weak solution $(\phip^{\eps}, \mu_{p}^{\eps}, \phid^{\eps}, \mu_{d}^{\eps}, q^{\eps}, n^{\eps})$
defined over $(0,T)$ to the following problem:
\begin{subequations}\label{eps:Reg}
\begin{alignat}{3}
\label{eq:eps:pa1}
 \de_{t} \phip &= M_{p} \Laplace \mu_{p} + \div (T(\phip) \nabla q) + \div \big( T(\phip)^2 \nabla \mu_p + T(\phip) T(\phid) \nabla \mu_d \big) + S_{p},  \\
\label{eq:eps:pa2}
 \mu_{p} & = F_{\eps,p} (\phip, \phid) + F_{1,p} (\phip, \phid) - \Laplace \phip,\\
\label{eq:eps:da1}
 \de_{t} \phid & = M_{d} \Laplace \mu_{d} + \div (T(\phid) \nabla q) + \div \big( T(\phip) T(\phid) \nabla \mu_p + T(\phid)^2 \nabla \mu_d \big) + S_{d},  \\
\label{eq:eps:da2}
\mu_{d} & =  F_{\eps,d} (\phip, \phid) + F_{1,d} (\phip, \phid) - \Laplace \varphi_{d},\\
\label{eq:eps:Spa}
 S_{p} & = \Sigma_p(n, \phip ,\phid) + m_{pp} \phip + m_{pd} \phid, \\
\label{eq:eps:Sda}
 S_{d} & = \Sigma_d(n,\phip, \phid) + m_{dp} \phip + m_{dd} \phid, \\
\label{eq:eps:qa}
0& = \Delta q + \div \big( T(\phip) \nabla  \mu_{p} + T(\phid) \nabla \mu_{d} \big) + S_p + S_d,\\
\label{eq:eps:na}
 0 & = -\Laplace n + T(\phip) n,
\end{alignat}
\end{subequations}
furnished with the initial and boundary conditions
\begin{subequations}
\begin{alignat}{3}
\phip(0) = \varphi_{p,0}, \quad \phid(0) = \varphi_{d,0} & \text{ in } \Omega, \label{eq:eps:bc:1} \\
M_{i} \pdnu \mu_{i} + T(\varphi_{i}) ( \nabla q + T(\phip) \nabla \mu_{p} + T(\phid) \nabla \mu_{d}) \cdot \norma = 0,
\quad \pdnu \varphi_{i} = 0 & \text{ on } \Gamma, \label{eq:eps:bc:2} \\
n = 1, \quad q = 0 & \text{ on } \Gamma. \label{eq:eps:bc:3}
\end{alignat}
\end{subequations}
Note that in \eqref{eps:Reg} the regularized convex part $F_{\eps}$ of the potential $F$ is still present.
\begin{thm}\label{thm:reg:eps}
 Let Assumption~\ref{ass:pd} hold, let $\eps \in (0,1)$ and $\delta \in (0,1)$,
 and let $F_{\eps} : \RR^{2} \to [0,+\infty)$ be the Moreau-Yosida approximation of $F_0$
 as detailed in Sec.~\ref{sec:scheme}. Let also $\varphi_{i,0,\delta} \in H^{2}_{n}(\Omega)$
 be the unique solution to \eqref{Initial:data:Neu}.
 Then, there exists $\delta_{0} > 0$ such that for all $\delta < \delta_{0}$, the weak solution
 $(\phip^{\delta,\eps}, \mu_{p}^{\delta,\eps}, \phid^{\delta,\eps}, \mu_{d}^{\delta,\eps}, q^{\delta,\eps}, n^{\delta,\eps})$
 to \eqref{delta:eps:Reg} defined on $[0,T]$ and obtained from Theorem \ref{thm:reg:local} satisfies the following properties:
 \begin{enumerate}
\item[\bf(1)] there exist functions $(\phip^{\eps}, \mu_{p}^{\eps}, \phid^{\eps}, \mu_{d}^{\eps}, q^{\eps}, n^{\eps})$ such that for $i = p,d$ and any $r < \infty$ in two dimensions and any $r \in [1,6)$ in three dimensions,
\begin{equation*}
\begin{alignedat}{3}
\varphi_{i}^{\delta,\eps} & \to \varphi_{i}^{\eps} && \text{ weakly* in } L^{\infty}(0,T;H^{1}(\Omega)) \cap L^{2}(0,T;H^{2}(\Omega)) \cap H^{1}(0,T;H^{1}(\Omega)'), \\
\varphi_{i}^{\delta,\eps} & \to \varphi_{i}^{\eps} && \text{ strongly in } C^{0}([0,T];L^{r}(\Omega)) \cap L^{2}(0,T;W^{1,r}(\Omega)) \text{ and a.e. in } Q, \\
\mu_{i}^{\delta,\eps} & \to \mu_{i}^{\eps} && \text{ weakly in } L^{2}(0,T;H^{1}(\Omega)), \\
q^{\delta,\eps} & \to q^{\eps} && \text{ weakly in } L^{2}(0,T;H^{1}(\Omega)), \\
n^{\delta,\eps} & \to n^{\eps} && \text{ weakly* in } L^{\infty}(0,T;W^{2,r}(\Omega)) \text{ and strongly in } L^{2}(0,T;H^{1}(\Omega)).
\end{alignedat}
\end{equation*}
\item[\bf(2)] The tuple $(\phip^{\eps}, \mu_{p}^{\eps}, \phid^{\eps}, \mu_{d}^{\eps}, q^{\eps}, n^{\eps})$
satisfies equations \eqref{eq:eps:pa2}, \eqref{eq:eps:da2}, \eqref{eq:eps:Spa}, \eqref{eq:eps:Sda}, \eqref{eq:eps:na}
a.e.~in $Q$, whereas equations \eqref{eq:eps:pa1}, \eqref{eq:eps:da1} and \eqref{eq:eps:qa} hold for a.e.~$t \in (0,T)$ in the following weak sense:
\begin{align*}
0 & = \langle \de_{t} \phip^{\eps}, \zeta \rangle + \int_{\Omega} \big ( M_{p} \nabla \mu_{p}^{\eps} + T(\phip^{\eps}) \big ( \nabla q^{\eps} + T(\phip^{\eps}) \nabla \mu_{p}^{\eps} + T(\phid^{\eps}) \nabla \mu_{d}^{\eps} \big ) \big ) \cdot \nabla \zeta - S_{p} \zeta \, dx, \\
0 & = \langle \de_{t} \phid^{\eps}, \zeta \rangle + \int_{\Omega} \big ( M_{d} \nabla \mu_{d}^{\eps} + T(\phid^{\eps}) \big ( \nabla q^{\eps} + T(\phid^{\eps}) \nabla \mu_{d}^{\eps} + T(\phip^{\eps}) \nabla \mu_{p}^{\eps} \big ) \big ) \cdot \nabla \zeta - S_{d} \zeta \, dx, \\
0 & = \int_{\Omega} \big ( \nabla q^{\eps} + T(\phip^{\eps}) \nabla \mu_{p}^{\eps} + T(\phid^{\eps}) \nabla \mu_{d}^{\eps} \big ) \cdot \nabla \xi - (S_{p} + S_{d}) \xi \, dx
\end{align*}
for all $\zeta \in H^{1}(\Omega)$ and $\xi \in H^{1}_{0}(\Omega)$.
Moreover, $0 \leq n^{\eps} \leq 1$ a.e. in $Q$, and $\varphi_{i}^{\eps}(0) = \varphi_{0,i}$ a.e. in $\Omega$.
\end{enumerate}
\end{thm}
\begin{proof}
Recalling the estimate \eqref{ic:delta:Unif:H1}, we immediately infer the following properties of the initial data $(\varphi_{p,0,\delta}, \varphi_{d,0,\delta})$:
\begin{align*}
\| \varphi_{p,0,\delta} \|_{H^{1}(\Omega)} + \| \varphi_{d,0,\delta} \|_{H^{1}(\Omega)} \leq C, \\
\varphi_{p,0,\delta} \to \varphi_{p,0}, \quad \varphi_{d,0,\delta} \to \varphi_{d,0} \text{ weakly in } H^{1}(\Omega) \text{ and  strongly in } L^{2}(\Omega).
\end{align*}
Furthermore, this choice of initial data for the regularized system \eqref{delta:eps:Reg} implies that the estimate \eqref{Unif:delta:2} is uniform in $\delta \in (0,\delta_{0})$.

Then, most of the weak/weak* convergence properties in the statement are directly deduced
from the uniform estimates \eqref{Unif:delta:1}, \eqref{Unif:delta:2},
\eqref{Unif:delta:3} and \eqref{Unif:delta:4}, while the strong convergences follow from using \cite[\S 8, Cor. 4]{Simon}.
On the other hand, the strong convergence of $n^{\delta,\eps}$ is proved, similarly as before, by
a Cauchy argument which we now sketch. Let (a small) $\eta > 0$ and (a large) $C_* > 0$ be given but otherwise arbitrary.
Then, thanks to the a.e.~convergence of $\phip^{\delta,\eps}$ to $\phip^{\eps}$ in $Q$, by Egorov's theorem there exists a
measurable subset $X_{\eta} \subset Q$ with $C_*|X_{\eta}| < \frac{1}{4} \eta$ and $\phip^{\delta,\eps} \to \phip^{\eps}$
uniformly in the complement $Q \setminus X_{\eta}$. By this uniform convergence, there
exists $\delta_{*} > 0$ such that for any two indices $0 < \delta_{1}, \delta_{2} < \delta_{*}$,
\begin{align*}
  C_* \int_{Q \setminus X_{\eta}} | T(\phip^{\delta_{1},\eps}) - T(\phip^{\delta_{2},\eps}) |^{2} \, dx \, dt < \frac{\eta}{2}.
\end{align*}
Then, following the computation in \eqref{Cauchy} and using the boundedness of $T$, we find that
\begin{align*}
& \| n^{\delta_{1},\eps} - n^{\delta_{2},\eps} \|_{L^{2}(0,T;H^{1}(\Omega))}^{2} \leq C_* \| T(\phip^{\delta_{1},\eps}) - T(\phip^{\delta_{2},\eps})\|_{L^{2}(0,T;L^{2}(\Omega))}^{2} \\
& \quad \leq C_* \int_{Q \setminus X_{\eta}} | T(\phip^{\delta_{1},\eps}) - T(\phip^{\delta_{2},\eps}) |^{2} \, dx \, dt
 + C_* \int_{X_{\eta}} | T(\phip^{\delta_{1},\eps}) - T(\phip^{\delta_{2},\eps}) |^{2} \, dx \, dt \\
& \quad < \frac{\eta}{2} + 2 C_* | X_{\eta} | < \eta,
\end{align*}
for $0 < \delta_{1},\delta_{2} < \delta_{*}$. Here $C_*$ is exactly the constant $C$ in \eqref{Cauchy}.
This shows that $\{n^{\delta,\eps}\}_{\delta \in (0,\delta_{*})}$ is a Cauchy sequence in $L^{2}(0,T;H^{1}(\Omega))$.
The property $0 \leq n^{\eps} \leq 1$ a.e. in $Q$ can be deduced also from a weak comparison principle.

Now passing to the limit $\delta \to 0$ in \eqref{eq:Spa}, \eqref{eq:Sda}, \eqref{eq:na} lead to \eqref{eq:eps:Spa}, \eqref{eq:eps:Sda} and \eqref{eq:eps:na}, respectively.
Let us fix $\zeta \in L^{2}(0,T;H^{1}(\Omega))$ and test \eqref{eq:pa2} with $\zeta$. Then,
\begin{align*}
\int_{0}^{T} \int_{\Omega} \big ( \mu_{p}^{\delta,\eps} + \Delta \phip^{\delta,\eps} - (F_{\eps,p}+ F_{1,p})(\phip^{\delta,\eps}, \phid^{\delta,\eps}) \big ) \zeta - \delta \nabla \de_{t} \phip^{\delta,\eps} \cdot \nabla \zeta \, dx \, dt = 0.
\end{align*}
Using the weak convergences of $\mu_{p}^{\delta,\eps}$, $\Delta \phip^{\delta,\eps}$ in $L^{2}(0,T;L^{2}(\Omega))$ and the Lipschitz continuity of $F_{\eps,p}$ and $F_{1,p}$, as well
as the boundedness $\| \sqrt{\delta} \nabla \de_{t} \phip^{\delta,\eps} \|_{L^{2}(0,T;L^{2}(\Omega))} \leq C_{\eps}$
resulting from \eqref{Unif:delta:2}, passing to the limit $\delta \to 0$ in the above equality leads to
\begin{align*}
\int_{0}^{T} \int_{\Omega} \big ( \mu_{p}^{\eps} + \Delta \phip^{\eps} - (F_{\eps,p} + F_{1,p})(\phip^{\eps}, \phid^{\eps}) \big ) \zeta \, dx \, dt = 0.
\end{align*}
Since the above identity holds for arbitrary $\zeta \in L^{2}(0,T;H^{1}(\Omega))$ and all the terms in the integrand belong to $L^{2}(0,T;L^{2}(\Omega))$, the fundamental lemma of calculus of variations then yields \eqref{eq:eps:pa2}.

In a similar fashion, we infer from testing \eqref{eq:pa1} with an arbitrary test function $\zeta \in L^{2}(0,T;H^{1}(\Omega))$ and then passing to the limit $\delta \to 0$ the identity
\begin{align*}
0 & = \int_{0}^{T} \langle \de_{t} \phip^{\eps}, \zeta \rangle \, dt - \int_{0}^{T} \int_{\Omega} S_{p}(n^{\eps}, \phip^{\eps}, \phid^{\eps}) \zeta \, dx \, dt \\
& \quad + \int_{0}^{T} \int_{\Omega}  \big ( M_{p} \nabla \mu_{p}^{\eps} + T(\phip^{\eps}) \big ( \nabla q^{\eps} + T(\phip^{\eps}) \nabla \mu_{p}^{\eps} + T(\phid^{\eps}) \nabla \mu_{d}^{\eps} \big ) \big ) \cdot \nabla \zeta \, dx \, dt,
\end{align*}
where in the above we used the strong $L^2$-convergences of $n^{\delta,\eps}$ and $\varphi_{i}^{\delta,\eps}$
with the generalized Lebesgue dominated convergence theorem and the assumption \eqref{struct:Sig}
to deduce that $S_{p}(n^{\delta,\eps}, \phip^{\delta,\eps}, \phid^{\delta,\eps})$ converges to $S_{p}(n^{\eps}, \phip^{\eps}, \phid^{\eps})$
strongly in $L^{2}(0,T;L^{2}(\Omega))$. Furthermore, by the continuity and boundedness of $T$, it is easy to see that
\begin{align*}
  T(\phip^{\delta,\eps}) \to T(\phip^{\eps}) \text{ weakly}* \text{ in $L^\infty(Q)$ and strongly in $L^p(Q)$ for all $p\in[1,\infty)$.}
\end{align*}
Moreover, the strong convergence of the initial
data $\varphi_{p,0,\delta}$ to $\varphi_{p,0}$ in $L^{2}(\Omega)$ and the strong convergence of $\phip^{\delta,\eps}$ to $\phip^{\eps}$ in $C^{0}([0,T];L^{2}(\Omega))$
yield $\phip^{\eps}(0) = \varphi_{p,0}$ as an equality in $L^{2}(\Omega)$.

Lastly, it remains to pass to the limit in \eqref{eq:qa}. Consider testing \eqref{eq:qa} with the product $\eta(t) \xi(x)$ for arbitrary
test functions $\eta \in C^{1}(0,T)$ with $\eta(T) = 0$ and $\xi \in H^{2}(\Omega) \cap H^{1}_{0}(\Omega)$, then we have
\begin{align*}
0 & = \int_{0}^{T} \int_{\Omega} - \delta q^{\delta,\eps} \xi \de_{t} \eta -  (S_{p} + S_{d})(n^{\delta,\eps}, \phip^{\delta,\eps}, \phid^{\delta,\eps}) \eta(t) \xi \, dx \, dt \\
& \quad + \int_{0}^{T} \eta(t) \int_{\Omega} \big ( \nabla q^{\delta,\eps} + T(\phip^{\delta,\eps}) \nabla \mu_{p}^{\delta,\eps} + T(\phid^{\delta,\eps}) \nabla \mu_{d}^{\delta,\eps} \big ) \cdot \nabla \xi + \delta \Delta q^{\delta,\eps} \cdot \Delta \xi \, dx \, dt.
\end{align*}
Thanks to $\| q^{\delta,\eps} \|_{L^{2}(0,T;L^{2}(\Omega))} \leq C_{\eps}$ from \eqref{Unif:delta:2} and $\sqrt{\delta} \| \Delta q^{\delta,\eps} \|_{L^{2}(0,T;L^{2}(\Omega))} \leq C_{\eps}$ from \eqref{Unif:delta:5b}, after passing to the limit we obtain
\begin{align*}
0 = \int_{0}^{T} \eta(t) \int_{\Omega} \big ( \nabla q^{\eps} + T(\phip^{\eps}) \nabla \mu_{p}^{\eps}
+ T(\phid^{\eps}) \nabla \mu_{d}^{\eps} \big ) \cdot \nabla \xi - (S_{p} + S_{d})(n^{\eps}, \phip^{\eps}, \phid^{\eps}) \xi \, dx \, dt,
\end{align*}
holding for all $\xi \in H^{2}(\Omega) \cap H^{1}_{0}(\Omega)$ and $\eta \in C^{1}(0,T)$ with $\eta(T) = 0$.  Using the density of $H^{2}(\Omega) \cap H^{1}_{0}(\Omega)$ in $H^{1}_{0}(\Omega)$ and the fundamental lemma of calculus of variations, we obtain the weak formulation of \eqref{eq:eps:qa} as stated in Theorem \ref{thm:reg:eps}.
\end{proof}


\section{Alternative proof via Galerkin approximation}\label{sec:Galerkin}

We prove here existence of a solution $(\phip^{\eps}, \mu_{p}^{\eps}, \phid^{\eps}, \mu_{d}^{\eps}, q^{\eps}, n^{\eps})$ to \eqref{eq:eps:pa1}-\eqref{eq:eps:na} with the initial and boundary
conditions \eqref{eq:eps:bc:1}-\eqref{eq:eps:bc:3} by means of the alternative Faedo-Galerkin argument. This can be stated as follows:
\begin{thm}\label{thm:Galerkin}
Let Assumption \ref{ass:pd} hold, let $\eps \in (0,1)$ and let $F_{\eps} : \RR^{2} \to [0,+\infty)$ be the Moreau-Yosida approximation of $F_0$
 as detailed in Sec.~\ref{sec:scheme}.  Then, there exists a tuple $(\phip^{\eps}, \mu_{p}^{\eps}, \phid^{\eps}, \mu_{d}^{\eps}, q^{\eps}, n^{\eps})$
 satisfying assertion~{\bf (2)} of Theorem \ref{thm:reg:eps}.  Furthermore, the following energy identity also holds:
\begin{equation}\label{Galerkin:Ener:Id}
\begin{aligned}
& \frac{d}{dt} \Big ( \int_{\Omega} (F_{\eps} + F_{1})(\phip^{\eps}, \phid^{\eps}) + \frac{1}{2} \big ( | \nabla \phip^{\eps} |^{2} + | \nabla \phid^{\eps} |^{2} \big ) \, dx \Big ) \\
& \qquad + M_{p} \| \nabla \mu_{p}^{\eps} \|_{L^{2}(\Omega)}^{2} + M_{d} \| \nabla \mu_{d}^{\eps} \|_{L^{2}(\Omega)}^{2} + \| \nabla q^{\eps} + T(\phip^{\eps}) \nabla \mu_{p}^{\eps} + T(\phid^{\eps}) \nabla \mu_{d}^{\eps} \|_{L^{2}(\Omega)}^{2} \\
& \quad = \int_{\Omega} (S_{p} + S_{d}) q^{\eps} + S_{p} \mu_{p}^{\eps} + S_{d} \mu_{d}^{\eps} \, dx.
\end{aligned}
\end{equation}
\end{thm}

\begin{proof}
Let $\{w_{i}\}_{i \in \N}$ and $\{y_{i}\}_{i \in \N}$ denote the eigenfunctions of the Neumann-Laplacian operator $-\Delta + I$ and
of the Dirichlet-Laplacian operator $-\Delta$, respectively. It is then well-known that $\{w_{i}\}_{i \in \N}$ and $\{y_{i}\}_{i \in \N}$
may be normalized in such a way to form orthogonal bases of $H^{1}(\Omega)$ and $H^{1}_{0}(\Omega)$, respectively,
that are also orthonormal with respect to the $L^{2}(\Omega)$-scalar product.  Furthermore, elliptic regularity
then yields that $w_{i} \in H^{2}_{\norma}(\Omega)$ and $y_{i} \in H^{2}(\Omega) \cap H^{1}_{0}(\Omega)$ for all $i \in \N$.

For fixed $k \in \N$, we seek functions $\varphi_{p,k}, \mu_{p,k}, \varphi_{d,k}, \mu_{d,k} \in \mathrm{span}\{w_{1}, \dots, w_{k}\}$, $q_{k}, n_{k}-1 \in \mathrm{span}\{y_{1}, \dots, y_{k}\}$ which solve for all $1 \le j \le k$ and $i \in \{p,d\}$:
\begin{subequations}\label{ODE}
\begin{alignat}{3}
0 & = (\varphi_{i,k}', w_{j}) + (M_{i} \nabla \mu_{i,k}, \nabla w_{j}) - (T(\varphi_{i,k}) \vu_{k}, \nabla w_{j}) - (S_{i,k}, w_{j}),  \label{Gal:varphi} \\
0 & = (\mu_{i,k}, w_{j}) - (\nabla \varphi_{i,k}, \nabla w_{j}) - ((F_{\eps,i} + F_{1,i})(\varphi_{p,k}, \varphi_{d,k}), w_{k}), \label{Gal:mu} \\
\vu_{k} & = - \nabla q_{k} - T(\varphi_{p,k}) \nabla \mu_{p,k} - T(\varphi_{d,k}) \nabla \mu_{d,k}, \label{Gal:velo} \\
0 & = (\vu_{k}, \nabla y_{j}) + (S_{p,k} + S_{d,k}, y_{j}),  \label{Gal:div} \\
S_{i,k} & = S_{i}(n_{k}, \varphi_{p,k}, \varphi_{d,k}), \label{Gal:S} \\
0 & = (\nabla n_{k}, \nabla y_{j}) + (T(\varphi_{p,k}) n_{k}, y_{j}), \label{Gal:n}
\end{alignat}
\end{subequations}
furnished with the initial condition $\varphi_{i,k}(0) = \Pi_{k}(\varphi_{i,0})$, where $\Pi_{k}$ denotes the projection to the finite dimension subspace $\mathrm{span}\{w_{1}, \dots, w_{k}\}$ and $\varphi_{i,k}' = \frac{d}{dt} \varphi_{i,k}$.
Moreover, $(\cdot,\cdot)$ denotes the usual scalar product in $L^2(\Omega)$.
We claim that the above system \eqref{ODE} admits a (local in time) solution. Indeed, we have
\begin{align*}
&\varphi_{i,k}(t)=\sum_{j=1}^{k}\alpha_{i,j}^{k}(t)w_j, \quad \mu_{i,k}(t)=\sum_{j=1}^k\beta_{i,j}^{k}(t)w_j,\\
&q_k(t)=\sum_{j=1}^{k}\gamma_j^k(t) y_j, \quad n_k(t)=1+\sum_{j=1}^{k}\eta_j^k(t) y_j,
\end{align*}
and we set $\boldsymbol{\alpha}_{i}^k:=(\alpha_{i,1}^k,\dots,\alpha_{i,k}^k)$,
$\boldsymbol{\beta}_{i}^k:=(\beta_{i,1}^k,\dots,\beta_{i,k}^k)$, $i=p,d$, $\boldsymbol{\gamma}^k:=(\gamma_1^k,\dots,\gamma_k^k)$
and $\boldsymbol{\eta}^k:=(\eta_1^k,\dots,\eta_k^k)$.
Then, from \eqref{Gal:mu}
the $\boldsymbol{\beta}_{i}^k$ can be immediately expressed as
globally Lipschitz continuous functions of the $\boldsymbol{\alpha}_{i}^k$ ($i=p,d$).
From \eqref{Gal:div}, taking \eqref{Gal:velo} into account, we can also express the $\boldsymbol{\gamma}^k$
as globally Lipschitz continuous functions of the $\boldsymbol{\alpha}_{i}^k$ ($i=p,d$)
 and $\boldsymbol{\eta}^k$.
It is now possible to express the $\boldsymbol{\eta}^k$ in terms of the $\boldsymbol{\alpha}_{p}^k$
via the equation \eqref{Gal:n} (see \cite[\S 6]{GLNeumann}). Indeed, \eqref{Gal:n} leads to the following
algebraic system
\begin{align}
&\lambda_\ell\eta^k_\ell+\sum_{j=1}^{k}\big(T(\varphi_{p,k})y_j,y_\ell\big)\eta_j^k=-\big(T(\varphi_{p,k}),y_\ell\big)\,,
\qquad \ell=1,\cdots k\,,
\label{algsy}
\end{align}
where $\lambda_\ell$ are the eigenvalues of the Dirichlet-Laplacian operator.
We now check that
the matrix $\mathbb{A}:=(a_{jl})_{j,l=1\dots k}$ of entries
 $a_{jl}:=\lambda_\ell\delta_{jl}+\big(T(\varphi_{p,k})y_l,y_j\big)$ is invertible. To this aim, it
is enough to prove that $\mathbb{A}$ is positive definite. Take $\boldsymbol{\zeta}:=(\zeta_1,\dots,\zeta_k)\in \mathbb{R}^k$ and set
$\phi:=\sum_{j=1}^{k}\zeta_j\,y_j$. We have
\begin{align}
\sum_{j,l=1}^{k} a_{jl}\zeta_j\zeta_l=\Vert\nabla\phi\Vert_{L^2(\Omega)}^2+\big(T(\varphi_{p,k})\phi,\phi\big)
\geq\Vert\nabla\phi\Vert_{L^2(\Omega)}^2\geq\lambda_1\Vert\phi\Vert_{L^2(\Omega)}^2\,.
\end{align}
Therefore, since the $y_j$, $j=1,\dots k$, are linearly independent, we have
$\sum_{j,l=1}^{k} a_{jl}\zeta_j\zeta_l=0$ iff $\phi=0$ iff $\boldsymbol{\zeta}=\boldsymbol{0}$,
and this implies that $\mathbb{A}$ is positive definite and hence invertible, for each choice of the $\boldsymbol{\alpha}_{p}^k$.
Hence, by inverting system \eqref{algsy} we can express the $\boldsymbol{\eta}^k$ as globally Lipschitz continuous functions
of the $\boldsymbol{\alpha}_{p}^k$. This also entails that the $\boldsymbol{\gamma}^k$ are globally
Lipschitz continuous functions of the $\boldsymbol{\alpha}_{i}^k$ only.
From \eqref{Gal:varphi}, taking also \eqref{Gal:velo} into account, we are led to an
ODE system of $2k$ equations in the $2k$ unknowns $\alpha_{i,j}^k$ ($i=p,d$)
in normal form, with some initial conditions that can be deduced from $\varphi_{i,k}(0) = \Pi_{k}(\varphi_{i,0})$.
The Cauchy-Lipschitz theorem can be applied to this system
and we can therefore guarantee the existence of its unique maximal solution
$\boldsymbol{\alpha}_{i}^k\in C^1([0,T_k^\ast);\mathbb{R}^k)$,
for some $T_k^\ast\in (0,\infty]$. Finally, $\boldsymbol{\beta}_{i}^k$, $i=p,d$,
$\boldsymbol{\gamma}^k$ and $\boldsymbol{\eta}^k$ are obtained as well.
We have thus proven that the finite-dimensional problem \eqref{Gal:varphi}--\eqref{Gal:n},
endowed with the initial condition $\varphi_{i,k}(0) = \Pi_{k}(\varphi_{i,0})$,
has a (unique) maximal solution
$\boldsymbol{\alpha}_{i}^k,\boldsymbol{\beta}_{i}^k,\boldsymbol{\gamma}^k,\boldsymbol{\eta}^k
\in C^1([0,T_k^\ast);\mathbb{R}^k)$.

\smallskip

We now derive estimates that are uniform in $k$.
Multiplying \eqref{Gal:varphi} with $\beta_{i,j}^k$, 
 \eqref{Gal:mu} with $\alpha_{i,j}^k$, \eqref{Gal:velo} with $\vu_{k}$, \eqref{Gal:varphi} with $\alpha_{i,j}^k$,
 upon summing the ensuing identities (over $j=1\dots k$, and $i=p,d$) and arguing as for \eqref{Apri:1}, \eqref{Apri:2},
  we are led to an analogous inequality to \eqref{Apri:3} with $\delta$ set to be zero.  The source term involving $S_{p,k} \mu_{p,k} + S_{d,k} \mu_{d,k}$ can be handled as in \eqref{Apri:S:mu}, as due to the sublinear growth of $\nabla F_{\eps}$, $\nabla F_{1}$ one obtains that
\begin{align}\label{Galerkin:mean:mu}
| (\mu_{i,k})\OO | \leq C_{\eps} \big (1 + \| \varphi_{p,k}\|_{L^{2}(\Omega)} + \| \varphi_{d,k} \|_{L^{2}(\Omega)} \big ).
\end{align}
Meanwhile, for the source term involving $(S_{p,k} + S_{d,k}) q_{k}$, we observe from \eqref{Gal:S} that
\begin{align*}
\| S_{p,k} \|_{L^{2}(\Omega)} + \| S_{d,k} \|_{L^{2}(\Omega)} \leq C (1 + \| \varphi_{p,k} \|_{L^{2}(\Omega)} + \| \varphi_{d,k} \|_{L^{2}(\Omega)} )\,.
\end{align*}
Testing \eqref{Gal:velo} with $\nabla q_{k}$ and \eqref{Gal:div} with $\gamma_j^k$, summing
over $j=1\dots k$ and combining the resulting identities yields 
\begin{align*}
\| \nabla q_{k} \|_{L^{2}(\Omega)}^{2} \leq C \| S_{p,k} + S_{d,k} \|_{L^{2}(\Omega)} \| q_{k} \|_{L^{2}(\Omega)} + \big ( \| \nabla \mu_{p,k} \|_{L^{2}(\Omega)} + \| \nabla \mu_{d,k} \|_{L^{2}(\Omega)} \big ) \|\nabla q_{k} \|_{L^{2}(\Omega)}.
\end{align*}
Young's inequality and Poincar\'{e}'s inequality then give
\begin{align}\label{Galerkin:q:est}
\| q_k \|_{H^{1}(\Omega)} \leq C \big ( 1 + \| \varphi_{p,k} \|_{L^{2}(\Omega)} + \| \varphi_{d,k} \|_{L^{2}(\Omega)} + \| \nabla \mu_{p,k} \|_{L^{2}(\Omega)} + \| \nabla \mu_{d,k} \|_{L^{2}(\Omega)} \big ).
\end{align}
A further application of Young's inequality entails
\begin{align*}
 \int_{\Omega} |(S_{p,k} + S_{d,k}) q_{k}| \, dx & \leq \frac{M_{p}}{4} \| \nabla \mu_{p,k} \|_{L^{2}(\Omega)}^{2} + \frac{M_{d}}{4} \| \nabla \mu_{d,k} \|_{L^{2}(\Omega)}^{2} \\
 & \quad + C \big ( 1 + \| \varphi_{p,k} \|_{L^{2}(\Omega)}^{2} + \| \varphi_{d,k} \|_{L^{2}(\Omega)}^{2} \big ),
\end{align*}
and consequently we obtain the analogue to \eqref{Apri:4} with $\delta$ and $\kappa$ set to zero.  In particular,
 by employing \eqref{Galerkin:mean:mu} and the Poincar\'{e} inequality to control
 the sequence of $\mu_{i,k}$ in $L^{2}(0,T;L^{2}(\Omega))$, 
 and also using the convergences $\varphi_{i,k}(0)\to \varphi_{i,0}$ in $H^1(\Omega)$ as $k\to\infty$,
 for $i=p,d$ (cf.~(A4)), we first get $T_k^\ast=+\infty$ (notice that
 $\Vert\varphi_{i,k}(t)\Vert=|\boldsymbol{\alpha}_i^k(t)|$), and then we obtain the following uniform estimate
 which holds for any given $0<T<+\infty$,
\begin{equation}\label{Galerkin:1}
\begin{aligned}
& \| (F_{\eps} + F_{1})(\varphi_{p,k}, \varphi_{d,k}) \|_{L^{\infty}(0,T;L^{1}(\Omega))} + \| \varphi_{i,k} \|_{L^{\infty}(0,T;H^{1}(\Omega))}
\\
& + \| \mu_{i,k} \|_{L^{2}(0,T;H^{1}(\Omega))} + \| \vu_{k} \|_{L^{2}(0,T;L^{2}(\Omega))} 
\leq C_{\eps}\,.
\end{aligned}
\end{equation}
Taking the $L^2-$norm of $\nabla q_k$ in \eqref{Gal:velo}
and using the above estimates for $\vu_k,\mu_{i,k}$ yields
\begin{align}
  & \Vert q_k\Vert_{L^2(0,T;H^1_0(\Omega))}\leq C\,.\label{best7}
\end{align}
Using the sublinear growth of $\nabla F_{\eps}$, $\nabla F_{1}$ again, together with \eqref{Galerkin:1},
we get also
\begin{align}
  & \Vert (F_{\eps,i} + F_{1,i})(\varphi_{p,k}, \varphi_{d,k})\Vert_{L^\infty(0,T;L^2(\Omega))}\leq C\,,\qquad i=p,d\,.\label{best9}
\end{align}

Testing \eqref{Gal:mu} with the coefficients of $\Delta \varphi_{i,k}$, and using the sublinear growth of $\nabla F_{1}$, as well as the convexity of $F_{\eps}$, by means of a standard argument we infer the boundedness of $\Delta \varphi_{i,k}$ in $L^{2}(0,T;L^{2}(\Omega))$.  Elliptic regularity then yields
\begin{align}\label{Galerkin:3}
\| \varphi_{i,k} \|_{L^{2}(0,T;H^{2}(\Omega))} \leq C_{\eps}\,.
\end{align}
Moreover, by comparison in \eqref{Gal:varphi}, and by relying on the basic estimates \eqref{Galerkin:1},
a standard argument entails also a bound for the sequences of the time derivatives
 $\varphi_{p,k}^{\,\prime}$, $\varphi_{d,k}^{\,\prime}$, namely
\begin{align}\label{Galerkin:4}
\| \varphi_{i,k}^{\,\prime} \|_{L^{2}(0,T;(H^{1}(\Omega))')}  \leq C_{\eps}\,.
\end{align}
Concerning the estimate for the sequence of $n_k$,
 testing \eqref{Gal:n} with $\eta_j^k$ and summing over $j=1\dots k$, we get
 \begin{align*}
& \Vert\nabla n_k\Vert_{L^2(\Omega)}^2+\big(T(\varphi_{p,n})n_k,n_k\big)
=\big(T(\varphi_{p,n}),n_k\big)\,.
\end{align*}
 Using the boundedness and the nonnegativity of $T(\cdot)$ we immediately infer the bound
 \begin{align}
 &\Vert n_k\Vert_{L^\infty(0,T;H^1(\Omega))}\leq C\,.\label{best8}
 \end{align}
Thanks to estimates \eqref{Galerkin:1}--\eqref{best8}, 
 and to Aubin-Lions lemma,
we deduce there exist $\varphi_p$, $\varphi_d$, $\mu_p$, $\mu_d$, $\vu$, $q$ and $n$ in the
regularity class
\begin{align}
\varphi_p,\varphi_d & \in L^2(0,T;H^2(\Omega))\cap L^\infty(0,T;H^1(\Omega))
\cap H^1(0,T;H^1(\Omega)^{\,\prime})\,,\label{reg1}\\[1mm]
\mu_p,\mu_d &\in L^2(0,T;H^1(\Omega))\,,\label{reg2}\\[1mm]
\vu &\in L^2(0,T;L^2(\Omega)^3)\,,\label{reg3}\\[1mm]
 q &\in L^2(0,T;H^1_0(\Omega))\,,\label{reg4}\\[1mm]
n &\in 1+L^\infty(0,T;H^1_0(\Omega))\,,\label{reg5}
\end{align}
such that, for a not relabelled subsequence, we have
\begin{align}
\varphi_{i,k} \to \varphi_i\quad & \mbox{weakly} * \text{ in }L^2(0,T;H^2(\Omega))\cap H^1(0,T;H^1(\Omega)^{\,\prime}) \cap L^\infty(0,T;H^1(\Omega))\,,\label{conv1}\\[1mm]
\varphi_{i,k}\to\varphi_i\quad & \mbox{strongly in }L^2(0,T;H^{2-\sigma}(\Omega))\quad(\sigma>0)\,,\mbox{ and a.e. in }Q_T\,,\label{conv2}\\[1mm]
\mu_{i,k} \to \mu_i\quad &\mbox{weakly in }L^2(0,T;H^1(\Omega))\,,\label{conv3}\\[1mm]
 \vu_k \to \vu\quad & \mbox{weakly in }L^2(0,T;L^2(\Omega)^3)\,,\label{conv4}\\[1mm]
 q_k \to q\quad & \mbox{weakly in } L^2(0,T;H^1_0(\Omega))\,,\label{conv6}\\[1mm]
n_k \to n\quad & \mbox{weakly}^\ast\mbox{ in } L^\infty(0,T;H^1(\Omega))\,,\label{conv7}
\end{align}
Moreover, thanks to the continuity of $F_{\eps,i} + F_{1,i}$ and to \eqref{best9}, \eqref{conv2}, we also have
\begin{align}
 & (F_{\eps,i} + F_{1,i})(\varphi_{p,k}, \varphi_{d,k})\to (F_{\eps,i} + F_{1,i})(\varphi_{p}, \varphi_{d})
 \quad\mbox{weakly}* \mbox{ in }L^\infty(0,T;L^2(\Omega))\,.\label{conv8}
\end{align}

We now need to derive strong convergence for the sequence of $n_k$
(or at least for a subsequence), in order to pass to the limit
in \eqref{Gal:S}, by showing that $\{n_{k}\}_{k \in \N}$ is a Cauchy sequence in $L^{2}(0,T;H^{1}(\Omega))$.
Here we cannot directly adapt the argument of the proof of Theorem \ref{thm:reg:eps}, since here
a uniform $L^\infty$ bound for the Galerkin approximants $n_k$ seems not to be available from a weak comparison principle. However,
the argument is still simple. Indeed, after writing \eqref{Gal:n} for two different indexes $k$ and $\ell$, taking the difference,
multiplying the resulting identity by $\eta_j^k-\eta_j^\ell$ and summing over $j=1\dots k$, we obtain
\begin{align*}
 & \Vert\nabla(n_k-n_\ell)\Vert_{L^2(\Omega)}^2+\big(T(\varphi_{p,k})(n_k-n_\ell),n_k-n_\ell\big)
 =-\big((T(\varphi_{p,k})-T(\varphi_{p,\ell}))n_\ell,n_k-n_\ell\big)\,.
\end{align*}
Thanks to the nonnegativity of $T(\cdot)$ and to the Poincar\'{e} inequality, then there follows that
\begin{align*}
 \Vert\nabla(n_k-n_\ell)\Vert_{L^2(\Omega)}^2&\leq C\Vert(T(\varphi_{p,k})-T(\varphi_{p,\ell}))n_\ell\Vert_{L^2(\Omega)}^2
 \leq C\Vert\varphi_{p,k}-\varphi_{p,\ell}\Vert_{L^\infty(\Omega)}^2\Vert n_\ell\Vert_{L^2(\Omega)}^2\,.
\end{align*}
We now use
 the strong convergence \eqref{conv2}, which implies
that $\{\varphi_{p,k}\}$ is a Cauchy sequence in $L^2(0,T;L^\infty(\Omega))$ (by taking $\sigma>0$ small enough),
 to deduce that
$\{n_k\}$ is a Cauchy sequence in $L^2(0,T;H^1_0(\Omega))$ 
 and hence we get
\begin{align}\label{conv9}
   &n_k\to n\quad\mbox{strongly in }L^2(0,T;H^1(\Omega))\,.
\end{align}
The weak/strong convergences \eqref{conv1}--\eqref{conv9} are now enough
to pass to the limit in the approximate problem \eqref{Gal:varphi}--\eqref{Gal:n}
by means of a standard argument and prove that
the tuple $(\phip, \mu_{p}, \phid, \mu_{d}, q, n)$ satisfies assertion~{\bf (2)} of Theorem \ref{thm:reg:eps}.
In particular, from \eqref{Gal:n} we obtain in the limit that $(\nabla n,\nabla y)+\big(T(\varphi_p)n,y\big)=0$
for all $y\in H_0^1(\Omega)$ which is the weak form of \eqref{eq:n}. Then, the weak comparison principle and elliptic regularity
can now be applied, leading, respectively, to the pointwise bound $0\leq n\leq 1$ a.e in $\Omega\times(0,T)$
and to the further regularity $n\in L^\infty(0,T;W^{2,r}(\Omega))$ for every $r<\infty$.

To show that \eqref{Galerkin:Ener:Id} holds for the weak solution, it suffices to show that we have sufficient regularity to
test the equation for $\varphi_{i}^{\eps}$ with $\mu_{i}^{\eps}$, the equation for $\mu_{i}^{\eps}$ with $\de_{t} \varphi_{i}^{\eps}$,
the equation for $\vu^{\eps}$ with $\vu^{\eps}$, to sum the resulting relations and to integrate by parts.
This is actually possible, and the main technical detail lies in showing that
\begin{align}\label{Tech}
\sum_{i=p,d} \langle \de_{t} \varphi_{i}^{\eps}, \mu_{i}^{\eps} \rangle = \frac{d}{dt} \int_{\Omega} (F_{\eps} + F_{1})(\phip^{\eps}, \phid^{\eps}) + \frac{1}{2} ( | \nabla \phip^{\eps} |^{2} + | \nabla \phid^{\eps} |^{2} ) \, dx.
\end{align}
To see this, we define the convex and lower semicontinuous functional $\mathcal{G} : L^{2}(\Omega) \times L^{2}(\Omega) \to \mathbb{R} \cup \{+\infty \}$ by
\begin{align*}
\mathcal{G}(\phip, \phid) := \begin{cases}
\int_{\Omega} F_{\eps}(\phip, \phid) + \frac{1}{2} ( | \nabla \phip|^{2} + | \nabla \phid|^{2}) \, dx & \text{ if } (\phip, \phid) \in (H^{1}(\Omega))^{2} \\
& \text{ and } F_{\eps}(\phip, \phid) \in L^{1}(\Omega), \\
+\infty & \text{ otherwise}.
\end{cases}
\end{align*}
Then, by \cite[Proposition 2.8]{Ba} the subdifferential of $\mathcal{G}$ can be characterized as
\begin{align*}
\de \mathcal{G}(\phip, \phid) & = ( - \Delta \phip + F_{\eps,p} (\phip,\phid), - \Delta \phid + F_{\eps,d}(\phip, \phid) )
\end{align*}
for all $(\phip, \phid) \in D(\de \mathcal{G}) = (H^{2}_{n}(\Omega) )^{2}$. Moreover, thanks to the
Lipschitz continuity of $F_1$, we have
\begin{align}
\sum_{i=p,d} \langle \de_{t} \varphi_{i}^{\eps}, F_{1,i}(\phip^{\eps}, \phid^{\eps}) \rangle = \frac{d}{dt} \int_{\Omega} F_{1}(\phip^{\eps},\phid^{\eps}) \, dx,\label{chain}
\end{align}
whereas applying standard chain rule formulas for monotone operators in Hilbert spaces
(see \cite[Lemme~3.3, p.~73]{Brezis}, cf.~also \cite[Proposition 4.2]{CKRS}), we deduce that
\begin{align*}
\langle (\de_{t} \phip^{\eps}, \de_{t} \phid^{\eps}), \de \mathcal{G}(\phip^{\eps}, \phid^{\eps}) \rangle = \frac{d}{dt} \mathcal{G}(\phip^{\eps}, \phid^{\eps}).
\end{align*}
Using that $\mu_{i}^{\eps} = - \Delta \varphi_{i}^{\eps} + (F_{\eps,i} + F_{1,i})(\phip^{\eps}, \phid^{\eps})$ holds a.e. in $\Omega \times (0,T)$, the required assertion \eqref{Tech} is proved.
\end{proof}
\begin{remark}
It is worth observing that the Faedo-Galerkin approximants $\mu_{i,k}$, being linear combinations of the
Neumann eigenfunctions $\{w_j\}$, satisfy a no-flux condition on $\Gamma$, which is different with
respect to the coupled condition expected to hold in the limit (i.e., the first of \eqref{bc}).
On the other hand, this is not a contradiction in view of the fact that in the limit $k\to\infty$
we will recover the first of \eqref{bc} (in an implicit form) from the limit of
equations \eqref{Gal:varphi}. Indeed, $\mu_{i,k}$ converges to $\mu_i$ in $H^1(\Omega)$,
but not in $H^2(\Omega)$ (cf.~\eqref{conv3}). Then, as $H^2_{{\norma}}(\Omega)$ is 
dense in $H^1(\Omega)$, the no-flux condition may be (and in fact is) lost in the limit
and replaced by a different one.
\end{remark}

\section{Passing to the limit $\eps \to 0$}\label{sec:limit}

Let $(\phip^{\eps}, \mu_{p}^{\eps}, \phid^{\eps}, \mu_{d}^{\eps}, q^{\eps}, n^{\eps})$ denote a weak solution to \eqref{eps:Reg} obtained
either from Theorem~\ref{thm:reg:eps} or from Theorem~\ref{thm:Galerkin}.
Introducing the velocity variable as
$\vu^{\eps} := - \nabla q^{\eps} - T(\phip^{\eps}) \nabla \mu_{p}^{\eps} - T(\phid^{\eps}) \nabla \mu_{d}^{\eps}$,
we can now rewrite \eqref{eps:Reg} as
\begin{subequations}\label{eps:sys}
\begin{alignat}{3}
\label{eps:pa1}
 \de_{t} \phip^{\eps} &= M_{p} \Laplace \mu_{p}^{\eps} - \div (T(\phip^{\eps}) \vu^{\eps})  + S_{p},  \\
\label{eps:pa2}
 \mu_{p}^{\eps} & = F_{\eps,p} (\phip^{\eps}, \phid^{\eps}) + F_{1,p} (\phip^{\eps}, \phid^{\eps}) - \Delta \phip^{\eps},\\
\label{eps:da1}
 \de_{t} \phid^{\eps} & = M_{d} \Laplace \mu_{d}^{\eps} - \div (T(\phid^{\eps}) \vu^{\eps}) + S_{d},  \\
\label{eps:da2}
\mu_{d}^{\eps} & =  F_{\eps,d} (\phip^{\eps}, \phid^{\eps}) + F_{1,d} (\phip^{\eps}, \phid^{\eps}) - \Delta \phid^{\eps},\\
\label{eps:Spa}
 S_{p} & = \Sigma_p(n^{\eps}, \phip^{\eps} ,\phid^{\eps}) + m_{pp} \phip^{\eps} + m_{pd} \phid^{\eps}, \\
\label{eps:Sda}
 S_{d} & = \Sigma_d(n^{\eps},\phip^{\eps}, \phid^{\eps}) + m_{dp} \phip^{\eps} + m_{dd} \phid^{\eps}, \\
\label{eps:velo}
\vu^{\eps} & = - \nabla q^{\eps} - T(\phip^{\eps}) \nabla \mu_{p}^{\eps} - T(\phid^{\eps}) \nabla \mu_{d}^{\eps} , \\
\label{eps:qa}
\div \vu^{\eps}& = S_p + S_d,\\
\label{eps:na}
 0 & = -\Laplace n^{\eps} + T(\phip^{\eps}) n^{\eps},
\end{alignat}
\end{subequations}
furnished with the initial-boundary conditions \eqref{eq:eps:bc:1}-\eqref{eq:eps:bc:3} (in fact, the system is satisfied in
the weak form specified in the statement; nevertheless, it is probably clearer to report the equations in
their strong formulation).

The aim of this section is to derive uniform a priori estimates in $\eps$ and then pass to the limit $\eps \to 0$.  Let us point out that the estimate \eqref{Unif:delta:1} involving $n^{\eps}$ is already uniform in $\eps$.  For convenience, we will drop the superscript $\eps$ in the variables, and denote with the symbol $C$ positive constants that are independent of $\eps$.


\subsection{A priori estimates}
\label{subs:apr}

We will now derive a number of estimates uniform with respect to $\eps$. We start
controlling the mean values of $\phip$ and $\phid$.  Denoting
\begin{align*}
\bm{y}(t) := ((\phip)\OO(t), (\phid)\OO(t)), \quad (\bm{\Sigma})\OO = ((\Sigma_{p})\OO, (\Sigma_{d})\OO),
\end{align*}
then by testing \eqref{eps:pa1} and \eqref{eps:da1} with $1$ leads to the following system of ODE's:
\begin{align}\label{sist:mean}
\frac{d}{dt} \bm{y}(t) = (\bm{\Sigma})\OO(t) + \underline{\underline{M}} \bm{y}(t)
\end{align}
for any $0 \leq t \leq T$.  Thanks to \eqref{struct:Sig}, \eqref{struct:2} and \eqref{assump:initial<1} we infer that the vector
$\bm{y}(t) = ( (\phip)\OO(t), (\phid)\OO(t))$ belongs to the interior $\inte \Delta_{0}$ for all times $t \in [0,T]$.
Indeed, at the time $t = 0$, $\bm{y}(0) \in \inte \Delta_{0}$ by \eqref{assump:initial<1}.  Suppose
that there exists a time $t_{*}$ such that $\bm{y}(t_{*}) \in \de \Delta_{0}$.  Then, taking $t = t_{*}$
in the above ODE, multiplying with the outer unit normal $\norma$ to $\Delta_{0}$
and applying \eqref{struct:2}, we necessarily have that
\begin{align*}
\frac{d}{dt} \bm{y}(t_{*}) \cdot \norma < 0.
\end{align*}
As a consequence, $\bm{y}(t)\in\inte\Delta_0$ for $t$ in a right neighbourhood of $t_*$, whence it is
apparent that $\bm{y}(t)$ can never leave $\Delta_0$. From this we deduce that there exist positive constants $0 < c_{1} < c_{2} < 1$
independent of $\eps$ such that
\begin{align}\label{eps:Unif:1}
c_{1} \leq (\phip)\OO(t), (\phid)\OO(t) \leq c_{2}, \quad c_{1} \leq (\phip + \phid)\OO(t) \leq c_{2} \quad \forall t \in [0,T].
\end{align}
Testing now \eqref{eps:pa1} with $\mu_{p}$, \eqref{eps:da1} with $\mu_{d}$, \eqref{eps:pa2} with $\de_{t} \phip$, \eqref{eps:da2} with $\de_{t} \phid$, \eqref{eps:velo}
with $\vu$ and summing leads to
\begin{equation}\label{eps:main:1}
\begin{aligned}
& \frac{d}{dt} \int_{\Omega} F_{\eps}(\phip, \phid) + F_{1}(\phip, \phid) + \frac{1}{2} \big (|\nabla \phip |^{2} + | \nabla \phid |^{2} \big ) \, dx \\
& \qquad + M_{p} \| \nabla \mu_{p} \|_{L^{2}(\Omega)}^{2} + M_{d} \| \nabla \mu_{d} \|_{L^{2}(\Omega)}^{2} + \| \vu \|_{L^{2}(\Omega)}^{2} \\
& \quad = \int_{\Omega} S_{p} \mu_{p} + S_{d} \mu_{d} + q (S_{p} + S_{d}) \, dx.
\end{aligned}
\end{equation}
In the above we used Darcy's law and integration by parts to deduce that
\begin{align*}
\int_{\Omega} \big ( T(\phip) \nabla \mu_{p} + T(\phid) \nabla \mu_{d} \big ) \cdot \vu \, dx = \int_{\Omega} - \nabla q \cdot \vu - | \vu |^{2} \, dx = \int_{\Omega} q (S_{p} + S_{d}) - | \vu |^{2} \, dx.
\end{align*}
Let us now observe that, by the boundedness of $\Sigma_{p}$, we have
\begin{align*}
\int_{\Omega} S_{p} \mu_{p} \, dx & \leq C \| \mu_{p} - (\mu_{p})\OO \|_{L^{1}(\Omega)} + C | (\mu_{p}) \OO | + \sum_{i=p,d} \int_{\Omega} m_{pi} \varphi_{i}  (\mu_{p} - (\mu_{p})\OO + (\mu_{p})\OO) \, dx \\
& \leq C \| \mu_{p} - (\mu_{p})\OO \|_{L^{1}(\Omega)} + C | (\mu_{p}) \OO | + \sum_{i=p,d} \int_{\Omega} m_{pi} (\varphi_{i} - (\varphi_{i})\OO) (\mu_{p} - (\mu_{p})\OO) \, dx,\\
& \quad + (\mu_{p})\OO \int_{\Omega} m_{pp} \phip + m_{pd} \phid \, dx \\
& \leq C \| \mu_{p} - (\mu_{p})\OO \|_{L^{1}(\Omega)} + C | (\mu_{p}) \OO | + C \sum_{i=p,d} \| \nabla \varphi_{i} \|_{L^{2}(\Omega)} \|\nabla \mu_{p} \|_{L^{2}(\Omega)},
\end{align*}
where we have used that $(( \phip)\OO, (\phid)\OO)$ never leaves the set $\Delta_{0}$ and so $m_{pp} (\phip)\OO + m_{pd} (\phid)\OO$ is bounded.
An analogous estimate holds for $S_{d} \mu_{d}$, whence, by the Poincar\'{e} and Young inequalities, we obtain
\begin{equation}\label{eps:source:1}
\begin{aligned}
\left | \int_{\Omega} S_{p} \mu_{p} + S_{d} \mu_{d} \right | & \leq C \big ( | (\mu_{p})\OO | + |(\mu_{d})\OO| \big ) + \frac{M_{p}}{4} \| \nabla \mu_{p} \|_{L^{2}(\Omega)}^{2}  \\
& \quad + \frac{M_{d}}{4} \| \nabla \mu_{d} \|_{L^{2}(\Omega)}^{2} + C \big ( 1 + \| \nabla \phip \|_{L^{2}(\Omega)}^{2} + \| \nabla \phid \|_{L^{2}(\Omega)}^{2} \big ).
\end{aligned}
\end{equation}
For the term involving the pressure $q$, we have
\begin{align*}
\left | \int_{\Omega} (S_{p} + S_{d})q \, dx \right | \leq C_{\eta} \big (1 + \| \phip \|_{L^{2}(\Omega)}^{2} + \| \phid \|_{L^{2}(\Omega)}^{2} \big ) + \eta \| q \|_{L^{2}(\Omega)}^{2}
\end{align*}
for some positive constant $\eta$ to be fixed below. To get an $L^{2}$-estimate of the pressure,
we use the Poincar\'{e} inequality for $H^{1}_{0}(\Omega)$-functions and Darcy's law to deduce that
\begin{align}\label{eps:pressure}
\| q \|_{L^{2}(\Omega)}^{2} \leq C \| \nabla q \|_{L^{2}(\Omega)}^{2} \leq C \big ( \| \vu \|_{L^{2}(\Omega)}^{2} + \| \nabla \mu_{p} \|_{L^{2}(\Omega)}^{2} + \| \nabla \mu_{d} \|_{L^{2}(\Omega)}^{2} \big ).
\end{align}
Take now $\eta$ sufficiently small so that
\begin{equation}\label{eps:source:2}
\begin{aligned}
\left | \int_{\Omega} (S_{p} + S_{d})q \, dx \right | & \leq \frac{1}{2} \| \vu \|_{L^{2}(\Omega)}^{2} + \frac{M_{p}}{4} \| \nabla \mu_{p} \|_{L^{2}(\Omega)}^{2} + \frac{M_{d}}{4} \| \nabla \mu_{d} \|_{L^{2}(\Omega)}^{2} \\
& \quad + C \big ( 1 + \| \phip \|_{L^{2}(\Omega)}^{2} + \| \phid \|_{L^{2}(\Omega)}^{2} \big ).
\end{aligned}
\end{equation}
Then, substituting \eqref{eps:source:1} and \eqref{eps:source:2} into \eqref{eps:main:1} yields
\begin{equation}\label{eps:main:2}
\begin{aligned}
& \frac{d}{dt} \int_{\Omega} F_{\eps}(\phip, \phid) + F_{1}(\phip, \phid) + \frac{1}{2} \big (|\nabla \phip |^{2} + | \nabla \phid |^{2} \big ) \, dx \\
& \qquad + \frac{M_{p}}{2} \| \nabla \mu_{p} \|_{L^{2}(\Omega)}^{2} + \frac{M_{d}}{2} \| \nabla \mu_{d} \|_{L^{2}(\Omega)}^{2} + \frac{1}{2} \| \vu \|_{L^{2}(\Omega)}^{2} \\
& \quad \leq C \big ( 1 + \| \phip \|_{H^{1}(\Omega)}^{2} + \| \phid \|_{H^{1}(\Omega)}^{2} + | (\mu_{p})\OO| + |(\mu_{d})\OO| \big ).
\end{aligned}
\end{equation}
The key point is now to derive uniform estimates on the mean values $|(\mu_{p})\OO|$ and $|(\mu_{d})\OO|$ in order to obtain useful a priori bounds from \eqref{eps:main:2}.
To this aim, we test \eqref{eps:pa2} with $\phip - (\phip)\OO$ and \eqref{eps:da2} with $\phid - (\phid)\OO$. Summing the resulting
relations gives
\begin{equation}\label{eps:mean:1}
\begin{aligned}
& \| \nabla \phip \|_{L^{2}(\Omega)}^{2} + \| \nabla \phid \|_{L^{2}(\Omega)}^{2} + \int_{\Omega} \nabla F_{\eps}(\phip, \phid) \cdot (\phip - (\phip)\OO, \phid - (\phid)\OO)^{\top} \, dx \\
& \quad \leq \int_{\Omega} (\mu_{p} - (\mu_{p})\OO) (\phip - (\phip)\OO) + (\mu_{d} - (\mu_{d})\OO) (\phid - (\phid)\OO) \, dx \\
& \qquad + C \big ( 1 + \| \phip \|_{L^{2}(\Omega)} + \| \phid \|_{L^{2}(\Omega)} \big )  \big (\| \nabla \phip \|_{L^{2}(\Omega)} + \| \nabla \phid \|_{L^{2}(\Omega)} \big ).
\end{aligned}
\end{equation}
At this point we will use the fact that $F_\eps$ satisfies \eqref{Feps:CKRS}, and consider $s = \phip$, $r = \phid$, $S = (\phip)\OO$, $R = (\phid)\OO$.  Then, we find that
\begin{equation}\label{Feps:est:1}
\begin{aligned}
& c_* | \nabla F_\eps (\phip, \phid) - \nabla F_\eps ((\phip)\OO, (\phid)\OO)| \\
& \quad \leq (\nabla F_\eps (\phip, \phid) - \nabla F_\eps ((\phip)\OO, (\phid)\OO) ) \cdot (\phip - (\phip)\OO, \phid - (\phid)\OO)^{\top} + C_*.
\end{aligned}
\end{equation}
Since $((\phip)\OO, (\phid)\OO) \in \Delta_{0}$ for all $t \in [0,T]$, we recall another property of the derivative of the Yosida approximation, namely
\begin{align*}
| \nabla F_\eps (p,q)| \leq | (\partial F_0)^{\circ}(p,q)| \quad \forall (p,q) \in \Delta,
\end{align*}
where $\partial$ denotes here the subdifferential in the sense of convex analysis and $(\partial F_0)^{\circ}(p,q)$ is
the element of minimum norm in the set $\partial F_0(p,q)$, that, at least in principle, could contain more than
one element. Here, however, $F_0$ is assumed to be $C^1$ in $\Delta$ and, consequently,
$| (\partial F_0)^{\circ}(p,q)| = | (\nabla F_0)(p,q)| < \infty$. Then, integrating \eqref{Feps:est:1} and rearranging leads to
\begin{align*}
  c_* \| \nabla F_\eps(\phip, \phid) \|_{L^{1}(\Omega)} & \leq c_* \| (\nabla F_0) ((\phip)\OO, (\phid)\OO) \|_{L^{1}(\Omega)} + C_* | \Omega | \\
   & \quad + \int_{\Omega} \nabla F_\eps (\phip, \phid) \cdot (\phip - (\phip)\OO, \phid - (\phid)\OO)^{\top} \, dx \\
   & \quad + C \| (\nabla F_0) ((\phip)\OO, (\phid)\OO) \|_{L^{2}(\Omega)} \big ( \| \nabla \phip \|_{L^{2}(\Omega)} + \| \nabla \phid \|_{L^{2}(\Omega)} \big ).
\end{align*}
Substituting this inequality into \eqref{eps:mean:1} then yields (cf.~also \cite{Brezis})
\begin{equation}\label{eps:mean:2}
\begin{aligned}
& \| \nabla \phip \|_{L^{2}(\Omega)}^{2} + \| \nabla \phid \|_{L^{2}(\Omega)}^{2} + \| \nabla F_{\eps}(\phip, \phid)\|_{L^{1}(\Omega)} \\
& \quad \leq \frac{M_{p}}{4} \| \nabla \mu_{p} \|_{L^{2}(\Omega)}^{2} + \frac{M_{d}}{4} \| \nabla \mu_{d} \|_{L^{2}(\Omega)}^{2} + C \big ( 1 + \| \phip \|_{H^{1}(\Omega)}^{2} + \| \phid \|_{H^{1}(\Omega)}^{2} \big ).
\end{aligned}
\end{equation}
Then, in light of \eqref{eps:mean:2}, observe that, by testing \eqref{eps:pa2} and \eqref{eps:da2} with $\pm 1$,
we obtain
\begin{equation}\label{eps:mean:3}
\begin{aligned}
 | (\mu_{p})\OO| +  | (\mu_{d}) \OO| &\leq C \big ( 1 + \| \phip \|_{L^{2}(\Omega)} + \| \phid \|_{L^{2}(\Omega)} \big ) + \| \nabla F_{\eps}(\phip, \phid) \|_{L^{1}(\Omega)} \\
& \leq \frac{M_{p}}{4} \| \nabla \mu_{p} \|_{L^{2}(\Omega)}^{2} + \frac{M_{d}}{4} \| \nabla \mu_{d} \|_{L^{2}(\Omega)}^{2} \\
& \quad + C \big ( 1 + \| \phip \|_{H^{1}(\Omega)}^{2} + \| \phid \|_{H^{1}(\Omega)}^{2} \big ).
\end{aligned}
\end{equation}
Returning to \eqref{eps:main:2} and substituting the estimate \eqref{eps:mean:3}, we infer
\begin{equation}\label{eps:main:3}
\begin{aligned}
&\frac{d}{dt} \int_{\Omega} F_{\eps}(\phip, \phid) + F_{1}(\phip, \phid) + \frac{1}{2} \big ( | \nabla \phip |^{2} + | \nabla \phid |^{2} \big ) \, dx \\
& \qquad + \frac{M_{p}}{4} \| \nabla \mu_{p} \|_{L^{2}(\Omega)}^{2} + \frac{M_{d}}{4} \| \nabla \mu_{d} \|_{L^{2}(\Omega)}^{2} + \frac{1}{2} \| \vu \|_{L^{2}(\Omega)}^{2} \\
& \quad \leq C \big ( 1 + \| \phip \|_{H^{1}(\Omega)}^{2} + \| \phid \|_{H^{1}(\Omega)}^{2} \big ).
\end{aligned}
\end{equation}
To \eqref{eps:main:3} we now add the following inequality obtained from testing \eqref{eps:pa1} with $\phip$ and \eqref{eps:da1} with $\phid$ and summing (cf. \eqref{Apri:2}):
\begin{align*}
\frac{1}{2} \frac{d}{dt} \big ( \| \phip \|_{L^{2}(\Omega)}^{2} + \| \phid \|_{L^{2}(\Omega)}^{2} \big )
& \leq \frac{M_{p}}{8} \| \nabla \mu_{p} \|_{L^{2}(\Omega)}^{2} + \frac{M_{d}}{8} \| \nabla \mu_{d} \|_{L^{2}(\Omega)}^{2} + \frac{1}{4} \| \vu \|_{L^{2}(\Omega)}^{2} \\
& \quad + C \big ( 1 + \|\phip \|_{H^{1}(\Omega)}^{2} + \| \phid \|_{H^{1}(\Omega)}^{2} \big ).
\end{align*}
By definition of the Yosida approximation, we have
\begin{align*}
F_\eps (s,r) \leq F_0(s,r) \quad \forall (s,r) \in \RR^{2}.
\end{align*}
Hence, recalling \eqref{fin:ener}, we arrive at
\begin{align*}
\int_{\Omega} F_{\eps}(\varphi_{p,0}, \varphi_{d,0}) + F_{1}(\varphi_{p,0}, \varphi_{d,0}) \, dx \leq C.
\end{align*}
Applying Gronwall's inequality in \eqref{eps:main:3}, we deduce
\begin{equation}\label{eps:main:4}
\begin{aligned}
& \| F_{\eps}(\phip, \phid)\|_{L^{\infty}(0,T;L^{1}(\Omega))} + \| \phip \|_{L^{\infty}(0,T;H^{1}(\Omega))} + \| \phid \|_{L^{\infty}(0,T;H^{1}(\Omega))} \\
& \quad + \| \nabla \mu_{p} \|_{L^{2}(0,T;L^{2}(\Omega))} + \| \nabla \mu_{d} \|_{L^{2}(0,T;L^{2}(\Omega))} + \| \vu \|_{L^{2}(0,T;L^{2}(\Omega))} \leq C.
\end{aligned}
\end{equation}
Thus, returning to \eqref{eps:mean:1}, using the above estimate and performing
some easy calculations, we infer
\begin{align}\label{nabla:Feps}
\| \nabla F_{\eps}(\phip, \phid)\|_{L^{1}(\Omega)} \leq C \big ( 1 + \| \nabla \mu_{p} \|_{L^{2}(\Omega)} + \| \nabla \mu_{d} \|_{L^{2}(\Omega)} \big ),
\end{align}
whence $\nabla F_{\eps}(\phip, \phid)$ is bounded in $L^{2}(0,T;L^{1}(\Omega))$.
In turn, by the first line of \eqref{eps:mean:3} we find that $|(\mu_{p})\OO|$ and $|(\mu_{d})\OO|$
are bounded in $L^{2}(0,T)$. Hence, recalling \eqref{eps:main:4} and using once more the Poincar\'{e} inequality, we get
\begin{align}\label{eps:main:5}
\| \mu_{p} \|_{L^{2}(0,T;H^{1}(\Omega))} + \| \mu_{d} \|_{L^{2}(0,T;H^{1}(\Omega))} \leq C.
\end{align}
Furthermore, recalling \eqref{eps:pressure}, thanks to \eqref{eps:main:4} we now have
\begin{align}\label{eps:pressure:est}
\| q \|_{L^{2}(0,T;H^{1}(\Omega))} \leq C.
\end{align}
Next, we infer estimates on $F_{\eps,p}$ by testing \eqref{eps:pa2} with $F_{\eps,p}(\phip, \phid) - (F_{\eps,p}(\phip, \phid))\OO$, leading to
\begin{align*}
& \| F_{\eps,p}(\phip, \phid) - (F_{\eps,p}(\phip, \phid))\OO \|_{L^{2}(\Omega)}^{2} \\
& \quad\quad+ \int_{\Omega}  F_{\eps,pp}(\phip, \phid) |\nabla \phip|^{2} + F_{\eps,pd}(\phip,\phid) \nabla \phip \cdot \nabla \phid \, dx \\
& \quad= \int_{\Omega} \big ( (\mu_{p} - (\mu_{p})\OO) -  F_{1,p}(\phip, \phid) \big ) (F_{\eps,p}(\phip, \phid) - (F_{\eps,p}(\phip, \phid))\OO)  \, dx,
\end{align*}
where
\begin{align*}
F_{\eps,pp} = \frac{\de^{2} F_{\eps}}{\de \phip^{2}}, \quad F_{\eps,pd} = \frac{\de^{2} F_{\eps}}{\de \phip \de \phid}.
\end{align*}
Adding the similar identity obtained testing \eqref{eps:da2} with $F_{\eps,d}(\phip, \phid) - (F_{\eps,d}(\phip, \phid))\OO$
and employing the Poincar\'{e} inequality together with the linear growth of $\nabla F_{1}$, it is not difficult to deduce
\begin{align*}
& \| F_{\eps,p}(\phip, \phid) - (F_{\eps,p}(\phip, \phid))\OO \|_{L^{2}(\Omega)}^{2} + \|F_{\eps,d}(\phip, \phid) - (F_{\eps,d}(\phip, \phid))\OO \|_{L^{2}(\Omega)}^{2} \\
& \quad + \int_{\Omega} (\nabla \phip, \nabla \phid) \cdot D^{2} F_{\eps}(\phip, \phid) ( \nabla \phip, \nabla \phid)^{\top} \, dx \\
& \quad \leq C \big (1 + \| \nabla \mu_{p} \|_{L^{2}(\Omega)}^{2} + \| \nabla \mu_{d} \|_{L^{2}(\Omega)}^{2} + \| \phip \|_{L^{2}(\Omega)}^{2} + \| \phid \|_{L^{2}(\Omega)}^{2} \big ) .
\end{align*}
Since $F_{\eps}$ is convex, the Hessian $D^{2}F_{\eps}$ is non-negative and consequently we can neglect the integral term on the left-hand side, leading to
(cf.~\eqref{eps:main:4})
\begin{equation}\label{eps:main:6}
\begin{aligned}
& \| F_{\eps,p}(\phip, \phid) - (F_{\eps,p}(\phip, \phid))\OO \|_{L^{2}(0,T;L^{2}(\Omega))} \\
 & \quad + \|F_{\eps,d}(\phip, \phid) - (F_{\eps,d}(\phip, \phid))\OO \|_{L^{2}(0,T;L^{2}(\Omega))} \leq C.
\end{aligned}
\end{equation}
Upon recalling the boundedness of $\nabla F_{\eps}(\phip, \phid)$ in $L^{2}(0,T;L^{1}(\Omega))$
resulting from \eqref{nabla:Feps}, we deduce a control of the quantities $\|(F_{\eps,p})\OO\|_{L^{2}(0,T)}$ and $\|(F_{\eps,d})\OO \|_{L^{2}(0,T)}$.
Hence, from \eqref{eps:main:6} we eventually obtain
\begin{align}\label{eps:main:7}
\| F_{\eps,p}(\phip, \phid) \|_{L^{2}(0,T;L^{2}(\Omega))} + \| F_{\eps,d}(\phip, \phid)\|_{L^{2}(0,T;L^{2}(\Omega))} \leq C.
\end{align}
Viewing \eqref{eps:pa2} and \eqref{eps:da2} as elliptic equations for $\phip$ and $\phid$, respectively, with
right-hand sides bounded in $L^{2}(0,T;L^{2})$ and no-flux boundary conditions, the elliptic regularity theory gives
\begin{align}\label{eps:main:8}
\| \phip \|_{L^{2}(0,T;H^{2}(\Omega))} + \| \phid \|_{L^{2}(0,T;H^{2}(\Omega))} \leq C.
\end{align}
Lastly, from inspection of \eqref{eps:pa1} and \eqref{eps:da1}, and thanks to the estimate \eqref{eps:main:4} we have
\begin{align}\label{eps:main:9}
\| \de_{t} \phip \|_{L^{2}(0,T;H^{1}(\Omega)')} + \| \de_{t} \phid \|_{L^{2}(0,T;H^{1}(\Omega)')} \leq C.
\end{align}


\subsection{Compactness and passing to the limit}
Thanks to the uniform estimates \eqref{eps:main:4}, \eqref{eps:main:5}, \eqref{eps:pressure:est}, \eqref{eps:main:7}, \eqref{eps:main:8}
and \eqref{eps:main:9}, by standard compactness arguments we infer the existence of functions $(\phip, \mu_{p}, \phid, \mu_{d}, q, \vu)$ and  of a pair $(\eta_p, \eta_d)$ such that
\begin{subequations}\label{limm}
\begin{alignat}{3}
\varphi_{i}^{\eps} & \to \varphi_{i} && \text{ weakly}* \text{ in } L^{\infty}(0,T;H^{1}(\Omega)) \cap L^{2}(0,T;H^{2}(\Omega)) \cap H^{1}(0,T;H^{1}(\Omega)'), \label{limm:11}\\
\varphi_{i}^{\eps} & \to \varphi_{i} && \text{ strongly in } C^{0}([0,T];L^{p}(\Omega)) \cap L^{2}(0,T;W^{1,p}(\Omega)), \label{limm:12}\\
\varphi_{i}^{\eps} & \to \varphi_{i} && \text{ a.e. in } \Omega \times (0,T), \label{limm:13}\\
\mu_{i}^{\eps} & \to \mu_{i} && \text{ weakly in } L^{2}(0,T;H^{1}(\Omega)), \label{limm:14}\\
\vu^{\eps} & \to \vu && \text{ weakly in } L^{2}(0,T;L^{2}(\Omega)), \label{limm:15}\\
q^{\eps} & \to q && \text{ weakly in } L^{2}(0,T;H^{1}(\Omega)), \label{limm:16}\
\end{alignat}
\end{subequations}
and
\begin{subequations}
\begin{alignat}{3}
F_{\eps, p}(\fhi_p,\fhi_d)&\to \eta_p&& \text{ weakly in } L^{2}(0,T;L^{2}(\Omega)),\label{limm:17}\\
F_{\eps, d}(\fhi_p,\fhi_d)&\to \eta_d&& \text{ weakly in } L^{2}(0,T;L^{2}(\Omega)), \label{limm:18}
\end{alignat}
\end{subequations}
for any $p < \infty$ in two dimensions and any $p \in [1,6)$ in three dimensions.
Using a similar argument as in the proof of Theorem \ref{thm:reg:eps}, by the a.e convergence of $\phip^{\eps}$ to $\phip$ in $\Omega \times (0,T)$
and Egorov's theorem, we can show that $\{n^{\eps} \}_{\eps \in (0,1)}$ is a Cauchy family in
$L^{2}(0,T;H^{1}(\Omega))$.  Then, there also exists a function $n \in L^{\infty}(0,T;W^{2,r}(\Omega))$, for any $r < \infty$,
with $0 \leq n \leq 1$ a.e.~in $\Omega \times (0,T)$, such that
\begin{align*}
n^{\eps} \to n \text{ strongly in } L^{2}(0,T;H^{1}(\Omega)).
\end{align*}
It now remains to pass to the limit $\eps \to 0$ in \eqref{eps:sys}.  Actually, in view of the above convergence properties,
the argument is very similar to that used before when we pass to the limit $\delta\to 0$. Hence, we just outline the differences which are
mainly related to the terms depending on $F_\eps$. Actually, combining \eqref{limm:12}, \eqref{limm:17}-\eqref{limm:18}
with the standard monotonicity argument in \cite[Prop.~1.1, p.~42]{Ba}, we readily deduce that $(\phip, \phid) \in \Delta$
a.e.~in $\Omega \times (0,T)$ and that $\eta_p=F_{0,p}(\phip, \phid)$, $\eta_d=F_{0,d}(\phip, \phid)$. This in particular implies that the truncation operator $T(\cdot)$ disappears in the limit
formulation of the problem; namely, we have $T(\phip)=\phip$ and $T(\phid)=\phid$ a.e.~in~$Q$.

Let us also point out that from the structural assumption \eqref{struct:2} and from the derivation of \eqref{eps:Unif:1} the limit
functions $\phip$ and $\phid$ satisfy $((\phip)\OO(t), (\phid)\OO(t)) \in \Delta_0$ and
\begin{align*}
0 < c_{1} \leq (\phip)\OO(t), (\phid)\OO(t) \leq c_{2} < 1, \quad c_{1} \leq (\phip + \phid)\OO(t) \leq c_{2}
\end{align*}
for all $t \in [0,T]$. Hence, we have proved that the tuple $(\phip, \mu_{p}, \phid, \mu_{d}, \vu, q, n)$
is a weak solution to system~\eqref{Model} in the sense of Definition~\ref{defn:Soln}.
This concludes the proof of Theorem~\ref{thm:main}.


\section*{Acknowledgements}

This research has been performed in the framework of the project Fondazione Cariplo-Regione Lombardia  MEGAsTAR
``Matema\-tica d'Eccellenza in biologia ed ingegneria come acceleratore
di una nuova strateGia per l'ATtRattivit\`a dell'ateneo pavese''. The present paper
also benefits from the support of the MIUR-PRIN Grant 2015PA5MP7 ``Calculus of Variations'' for GS,
and of the GNAMPA (Gruppo Nazionale per l'Analisi Matematica, la Probabilit\`a e le loro Applicazioni)
of INdAM (Istituto Nazionale di Alta Matematica) for SF, ER, and~GS.
SF is ``titolare di un Assegno di Ricerca dell'Istituto Nazionale di Alta Matematica".


\end{document}